\documentclass[10pt,a4paper]{article} 
\usepackage[T1]{fontenc}
\usepackage[cp1251]{inputenc}
\usepackage[english]{babel}
\usepackage{amsmath,amssymb,amsfonts,amsthm}
\usepackage{graphicx}
\usepackage{enumerate}
\usepackage[nodisplayskipstretch]{setspace}
\usepackage[section]{placeins}
\usepackage{geometry}
\usepackage{float}
\usepackage{caption,subfig}
\usepackage{wrapfig}
\usepackage[colorlinks,citecolor=blue,linkcolor=blue,urlcolor=black]{hyperref}
\usepackage{hypcap}
\usepackage{breakurl}
\usepackage[numbers,sort&compress]{natbib}
\allowdisplaybreaks[4]

\newcommand{\R}{{\mathbb R}}
\newcommand{\Rn}{{\mathbb R}^n}
\newcommand{\range}{\mathop{\mathrm{range}}}
\newcommand{\T}{{\mathop{\mathrm{T}}}}
\newcommand{\res}{\mathop{\mathrm{Res}}}
\newcommand{\ii}{\mathrm{i}}

\newcommand{\la}{\lambda}

\theoremstyle{definition}
 \newtheorem{definition}{Definition}[section]
 \newtheorem{remark}{Remark}[section]
 
\theoremstyle{plain}
 \newtheorem{theorem}{Theorem}[section]
 \newtheorem{lemma}{Lemma}[section]
 \newtheorem*{lemma*}{Lemma}
 \newtheorem{corollary}{Corollary}[section]
\newtheorem{proposition}{Proposition}[section]

\begin{document}

\title{\textbf{Combined methods for solving time-varying semilinear differential-algebraic equations with the use of spectral projectors and applications}}

\author{Maria Filipkovska (Filipkovskaya)\medskip \\
  \it\small Institute of Analysis and Scientific Computing, Technische Universit\"{a}t Wien, and \\
  \it\small B. Verkin Institute for Low Temperature Physics and Engineering  \\
  \it\small of the National Academy of Sciences of Ukraine, and   \\ \it\small Friedrich-Alexander-Universit\"{a}t Erlangen-N\"{u}rnberg    \\
 \small filipkovskaya@ilt.kharkov.ua;\;  maria.filipk@gmail.com }%

\date{ }
\maketitle

 \begin{abstract} 
Two combined methods for computing solutions of time-varying semilinear differential-algebraic equations (descriptor systems) are obtained. When constructing the methods, time-varying spectral projectors which can be found numerically are used. This enables one to numerically solve the differential-algebraic equation (DAE) in the original form without additional analytical transformations. The convergence and correctness of the developed methods are proved.  The methods are applicable to the semilinear DAEs with the continuous nonlinear part which may not be differentiable in time.  The global Lipschitz condition and other conditions of this kind are not used in the presented theorems on the global solvability of DAEs and on the convergence of the methods. This extends the scope of the methods.  The obtained theorems ensure both the existence of a unique global exact solution and the convergence of the methods, which enables one to compute an approximate solution on any given time interval. Numerical examples illustrating the capabilities of the methods and their effectiveness in various situations are provided. To demonstrate the practical application of the obtained methods and theorems, the numerical and theoretical analyses of mathematical models of the dynamics of electric circuits are carried out. It is shown that their results are consistent.
\end{abstract} 
 

\smallskip
{\small\texttt{MSC2000:}{65L20, 65L80, 34A09, 34A12, 47N40}

}

\maketitle

 \section{Introduction}\label{Intro}

Consider an implicit differential equation of the form
\begin{equation}\label{DAE}
\frac{d}{dt}[A(t)x(t)]+B(t)x(t)=f(t,x(t)),
\end{equation}
where $t\in [t_+,\infty)$, $t_+\ge 0$, $f\in C([t_+,\infty)\times\Rn,\Rn)$ and $A,\, B\in C([t_+,\infty),\mathrm{L}(\Rn))$  \,(\,$\mathrm{L}(X,Y)$ denotes the space of continuous linear operators acting from the vector space $X$ to the vector space~$Y$, and $\mathrm{L}(X):=\mathrm{L}(X,X)$).
The operators $A(t)$, $B(t)$ (depending on the parameter $t$) can be degenerate (i.e., noninvertible).  Equations of the type \eqref{DAE} with a degenerate (for some $t$) operator $A(t)$ are called \emph{degenerate differential equations} or \emph{differential-algebraic equations} (\emph{DAEs}).  DAEs of the form \eqref{DAE} are commonly referred to as \emph{semilinear}. Since the operators $A(t)$, $B(t)$ are time-varying, equation \eqref{DAE} is called a \emph{time-varying} semilinear DAE or a time-varying degenerate differential equation (DE).  In what follows, for the sake of generality, equation~\eqref{DAE}, where $A\in C([t_+,\infty),\mathrm{L}(\Rn))$  is an arbitrary operator function (i.e., the operator $A(t)$ is not necessarily degenerate), will be called a \emph{time-varying semilinear differential-algebraic equation}.

The initial condition for the DAE is given by
 \begin{equation}\label{ini}
x(t_0)=x_0,
 \end{equation}
where $t_0\ge t_+$. The presence of a degenerate operator at the derivative in the DAE means the presence of algebraic constraints, namely, the graphs of the solutions must lie in the manifold generated by the ``algebraic part'' of the DAE and the initial points $(t_0,x_0)$ must also belong to this manifold (see Remark~\ref{RemConsistInis}).

A function $x\in C([t_0,t_1),\Rn)$\, ($[t_0,t_1)\subseteq [t_+,\infty)$) is said to be a \emph{solution of equation  \eqref{DAE} on $[t_0,t_1)$} if the function  $A(t)x(t)$ is continuously differentiable on $[t_0,t_1)$ and $x(t)$ satisfies \eqref{DAE} on $[t_0,t_1)$.    If the solution $x(t)$ of \eqref{DAE} satisfies the initial condition \eqref{ini}, then it is called a \emph{solution of the initial value problem} (\emph{the IVP}) \eqref{DAE}, \eqref{ini}.

Notice that if we consider a time-varying semilinear DAE of the form
\begin{equation}\label{DAE1}
A(t)\frac{d}{dt}x(t)+B(t)x(t)=f(t,x(t))
\end{equation}
instead of \eqref{DAE}, then we must require greater smoothness for its solution than for \eqref{DAE}. Namely, a function $x\in C^1([t_0,t_1),\Rn)$ satisfying equation \eqref{DAE1} on $[t_0,t_1)$ is called a solution of \eqref{DAE1}.

DAEs or degenerate DEs are also called descriptor equations, algebraic-differential systems and differential equations (or dynamical systems) on manifolds.  These equations are used to describe mathematical models in control theory, radioelectronics,  cybernetics, mechanics, economics, ecology, chemical kinetics and gas industry (see, e.g., \cite{Brenan-C-P,BKT,CheAkLe,Kunkel_Mehrmann, Lamour-Marz-Tisch,Riaza,Vlasenko1}). It is known that the dynamics of electrical circuits is modeled using DAEs which, in general, cannot be reduced to explicit ordinary differential equations (ODEs).  In Sections~\ref{ApplPIMM} and \ref{ApplDE} we will consider two mathematical models having the form of the DAE \eqref{DAE}, which describes transient processes in electrical circuits.

In the present paper, two combined numerical methods for the time-varying semilinear DAEs, which have the first and second orders of convergence, are developed. To obtain these methods, we use, in particular, the time-varying spectral projectors described in Sections \ref{IndexProjector} and \ref{NumProj} and the scheme with recalculation (the ``predictor-corrector'' scheme).
The methods are called \emph{combined} since each of them is essentially a combination of two methods, namely, a difference method for the ``differential part'' and an iteration method for the ``algebraic part'' of the DAE (this combination is used in method 1 presented in Section \ref{Num-meth}), and in addition recalculation is used in method 2 presented in Section \ref{Num-Mod_meth}.
The theorems on the convergence and the orders of accuracy of the methods are proved in Sections~\ref{Num-meth}, \ref{Num-Mod_meth}. These theorems contain conditions for the existence and uniqueness of exact solutions, which in conjunction with conditions for the convergence of the methods ensures the correctness of the methods as well.

Earlier, numerical methods for time-invariant semilinear DAEs were obtained (see \cite{Fil.CombMeth}) using time-invariant spectral projectors. The presence of time-varying operators in the DAE \eqref{DAE} (as well as in \eqref{DAE1}) significantly complicates the construction of the numerical methods and the proof of their convergence; however, the approach developed for the proof of the existence and uniqueness of global solutions in \cite{Fil.DE-1} enables one to solve this problem as well. Furthermore, in \cite{Fil.CombMeth} an approximation by a centered difference was used in obtaining the second-order method, but this increases the quantitative characteristic of stability and, accordingly, a smaller step size may be required in calculations. Therefore, in the present paper, instead of a centered difference we use the recalculation technique to achieve the second order of convergence.

The \emph{main aims of the present work} were: (1) to develop numerical methods, which have certain advantages described below, for time-varying semilinear DAEs, and (2) to demonstrate that the analytical methods, developed earlier to study the solvability and stability of time-varying semilinear DAEs, can be applied to the development of the methods of numerical analysis. The numerical and theoretical analyses of mathematical models, presented in Sections \ref{ApplPIMM} and \ref{ApplDE}, are a demonstration of the effectiveness of the comprehensive analysis of practical problems by using the methods presented in \cite{Fil.DE-1,Fil.DE-2} and the present paper.

It is clear that for different problems related to the numerical solution of various types of DAEs different approaches may be more or less suitable (see the literature mentioned below).

The obtained numerical methods, theorems on their convergence and other results presented in the paper have the following \emph{advantages and distinctive features}:
 \begin{enumerate}[1.]
\leftmargin=0pt
\item  The developed methods are applicable to DAEs of the type \eqref{DAE} and \eqref{DAE1} with the continuous nonlinear part which may not be differentiable in $t$ (see Propositions~\ref{remNum-meth},~\ref{remModNum-meth}).  This is important for applications, since such equations arise in various practical problems. For example, the functions of currents and voltages in electric circuits may not be differentiable (or be piecewise differentiable) or may be approximated by nondifferentiable functions.  As examples, nonsinusoidal currents and voltages of the ``sawtooth'', ``triangular'' and ``rectangular'' shapes  \cite{Erickson-Maks} can be considered, but more complex shapes are also occurred.  In Sections \ref{NumAnalPIMM} and \ref{NestElCirc-NumAnal}, the examples of numerical solutions for electrical circuits with nondifferentiable (on a given time interval) functions of voltages are presented (Fig.~\ref{ExPila_h_001} and~\ref{Example6_1-3}).

\item  The presented theorems and propositions on the convergence of the methods give conditions which ensure the existence of a unique exact solution of the IVP for the DAE on $[t_0,\infty)$ and enable one to compute approximate solutions on any given time interval $[t_0,T]$.
    This is important to note, since when proving the convergence of a method one often assumes in advance that there is a unique exact solution on the interval where the computation will be carried out, while the calculation of the allowable length of this interval is a separate problem. Also, one often uses theorems allowing one to prove the existence and uniqueness of an exact solution only on a sufficiently small (local) time interval, but in this case the numerical method can be correctly applied only on this small interval.

\item The global Lipschitz condition and other conditions of this kind, including the global condition of the contractivity (the Lipschitz condition with a constant less than 1), are not used in the theorems on the DAE global solvability and on the convergence of the methods. The global Lipschitz condition is not fulfilled for mathematical models of electrical circuits with certain nonlinear parameters (e.g., in the form of power functions mentioned in Section \ref{ApplPIMM}, \ref{ApplDE}).
    In general, various types of DEs with non-Lipschitz or non-globally Lipschitz functions (see, e.g., \cite{IzgiCetin,TambueMukam} and references therein) arise in applications.

\item It is not required that the DAE in question (i.e.,  \eqref{DAE} or \eqref{DAE1}) be a regular DAE of tractability index 1 (see the definitions in \cite[p.~319-320]{Lamour-Marz-Tisch}), i.e., that  $\lambda A(t)+B(t)-\frac{\partial f}{\partial x}(t,x)$ is a regular pencil of index~1, or that the DAE has index 1 in any other sense given in the literature on DAEs \cite{Kunkel_Mehrmann,Riaza,Ascher-Petz,Brenan-C-P,Hairer-W}. Instead, the pencil  $\lambda A(t)+B(t)$ associated with the linear part of the DAE is required to be a regular pencil of index not higher than 1. The definition of the pencil index and the discussion of this issue are given in Section \ref{Sect-Index}.
    Note that the purpose of the paper was to construct numerical methods for the DAE (i.e., equation \eqref{DAE} or \eqref{DAE1} with the degenerate operator $A(t)$\,). Since $\lambda A(t)+B(t)$ has index~0 in the case when $A(t)$ is nondegenerate (for all~$t$), we will formulate statements (theorems, lemmas, etc.) and carry out proofs for the case when $A(t)$ is degenerate but not identically zero and $\lambda A(t)+B(t)$ is a regular pencil of index~1.
    Of course, for the cases when $A(t)$ is nondegenerate or zero (for all~$t$), the results obtained in the paper remain valid.

\item The time-varying spectral projectors used in developing the numerical methods can be numerically found, as shown in Section \ref{NumProj} (the description of the projectors is given in Section \ref{IndexProjector}), which enables one to numerically solve the DAE \eqref{DAE} in the original form, i.e., additional analytical transformations are not required for the application of the methods.
\end{enumerate}

Generally, there are already a lot of works on numerical methods for solving various types of DAEs. Key ideas and main approaches to the numerical solution of DAEs are described in the monographs \cite{Ascher-Petz,Brenan-C-P,Hairer-W,Kunkel_Mehrmann,Lamour-Marz-Tisch,Riaza} (we also refer to \cite{Fil.CombMeth,Linh-Mehrmann,Knorrenschild} for an overview).
Since we consider the semilinear DAEs with regular characteristic pencils of index 1, which can be reduced to semi-explicit DAEs with the help of spectral projectors (see Section \ref{IndexProjector}), we will mainly discuss index-1 DAEs, including semi-explicit DAEs of index 1.
The main idea of many works is the reduction of a DAE to an ODE or the replacement of a DAE by a stiff ODE for the further application of methods for solving ODEs, or the use of these methods directly for solving DAEs.
For example, to solve autonomous semi-explicit DAEs of index 1, the $\varepsilon$-embedding method is applied in \cite{Hairer-W,Knorrenschild}. Also, for the nonautonomous semi-explicit DAE $dy/dt= f(t,y,z)$, $0=g(t,y,z)$ of index 1 (it means that $[\partial g/\partial z]^{-1}$ exists and is bounded in a neighborhood of the exact solution \cite{Brenan-C-P}\,), \cite{Brenan-C-P,Kunkel_Mehrmann, Ascher-Petz,Knorrenschild} provide the $\varepsilon$-embedding method in which the Runge-Kutta (RK) method is applied to the corresponding stiff system $dy/dt= f(t,y,z,\varepsilon)$, $\varepsilon\, dz/dt= g(t,y,z,\varepsilon)$, where $\varepsilon>0$ is a small parameter, and then $\varepsilon = 0$ is set in the resulting formulas. A solution of the stiff system of ODEs  $dy/dt=f(t,y,z,\varepsilon)$, $\varepsilon\, dz/dt=g(t,y,z,\varepsilon)$ ($\varepsilon>0$) in general does not approach the solution of the reduced DAE (obtained by setting $\varepsilon=0$)\; $dy/dt=f(t,y,z,0)$, $0=g(t,y,z,0)$; however, under certain conditions it is possible to obtain a stiff ODE system whose solutions approximate solutions of the reduced DAE \cite{Knorrenschild,Brenan-C-P}. Note that in order to apply the $\varepsilon$-embedding method, it is necessary to reduce a semi-explicit DAE to the corresponding stiff ODE system that requires a sufficient smoothness of the nonlinear function in the DAE (the continuity of the function is not enough here).
As mentioned above, one of the advantages of the methods proposed in the present paper is that they are applicable to semilinear DAEs with the continuous nonlinear part which may not be differentiable in $t$.
The backward differentiation formulas (BDF), RK methods and general linear multi-step methods were presented for semi-explicit DAEs and regular nonlinear DAEs of index 1 in \cite{Kunkel_Mehrmann,Lamour-Marz-Tisch,Ascher-Petz,Brenan-C-P} and also for regular time-invariant quasilinear DAEs in \cite{Hairer-W}. In \cite{Kunkel_Mehrmann,Linh-Mehrmann}, the collocation RK method, the BDF  method and a half-explicit method were given for regular strangeness-free DAEs.  The RK methods are also used to solve semi-explicit DAEs of index 2 or 3 in \cite{Ascher-Petz,Brenan-C-P,Hairer-W} and \cite{SatoMdA}.
There exist numerical methods for the reduction of the index of DAEs (see, e.g., \cite{Ascher-Petz,Kunkel_Mehrmann,Lamour-Marz-Tisch,Hairer-W}).

The stiff ODE systems mentioned above arise in various applications. For example, the ODE systems containing fast and slow motions (variables) are stiff ODE systems which describe motions of a rigid body about its center of mass \cite{CheAkLe,BogSob}. The averaging method enabling one to obtain approximate solutions of the stiff (perturbed) system on a large time interval is presented in \cite{CheAkLe}, and the method of integral manifolds, which allows one to reduce the original perturbed system to an ODE, has been considered in \cite{BogSob}. A solution of the reduced DAE can be a good approximation to the solution of the corresponding stiff ODE system under certain conditions.

The projector based analysis of DAEs and the application of BDF, RK and general linear methods to DAEs with properly stated leading term is described in \cite{Lamour-Marz-Tisch}. The technique described in this monograph uses the concept of the tractability index. As mentioned above, under the restrictions used in the present paper, the DAEs may not have tractability index 1. Furthermore, another projector method is used in this paper.

\textbf{The paper has the following structure:}
In Section~\ref{IndexProjector}, we consider constraints on the operator coefficients (more precisely, on the characteristic operator pencils) of the DAEs, define the index of a regular operator pencil, and describe the spectral projectors and the method for reducing the time-varying semilinear DAE to a semi-explicit form by using these projectors.
Section \ref{GlobSolv} provides theorems and propositions (proved in prior papers), which give conditions for the existence, uniqueness and boundedness of the exact global solutions, as well as certain definitions and the remarks on the solution properties and the theorem applications.
These results were obtained in the prior papers, but since they are necessary for the formulation and proof of the main results of the present paper, we provide them here.
In Section~\ref{NumMs}, an algorithm for calculating the spectral projectors is described, and two combined methods for solving the time-varying semilinear DAE are obtained, as well as the theorems on the convergence, correctness and the orders of convergence (the global errors) of the methods are proved.
In addition, the important propositions on the convergence of the methods, when weakening the smoothness requirements for the nonlinear functions in the DAEs, are given in Subsections~\ref{Num-meth}, \ref{Num-Mod_meth}.  In Sections \ref{ApplPIMM} and \ref{ApplDE}, the theoretical and numerical analyses of mathematical models of the dynamics of electric circuits are carried out. This, on the one hand, demonstrates the application of the obtained theorems and methods to real physical problems, and on the other hand, shows that the theoretical and numerical results are consistent.  In Section~\ref{CompareMeth}, numerical examples illustrating the proved convergence of the methods are given, and the comparative analysis of the methods is carried out.

In the paper, \emph{the following notation is used}:  $I_X$ denotes the identity operator in the space $X$; $\delta_{ij}$ is the Kronecker delta; $C(0,\infty)=C((0,\infty),(0,\infty))$.
A function, for example $f$, is often denoted by the same symbol $f(x)$ as its value at the point $x$ in order to explicitly indicate its argument (or arguments), but it will be clear from the context what exactly is meant. Notice that when the formula breaks at the multiplication sign, we denote it by~$\times$.

 \section{Preliminaries}

 \subsection{Index of a regular pencil of operators, spectral projectors and their application}\label{IndexProjector}

 \subsubsection{Index of a regular pencil $\lambda A(t)+B(t)$}\label{Sect-Index}

Consider the \emph{operator pencil $\lambda A(t)+B(t)$} ($\lambda$ is a complex parameter) \emph{associated with the linear part} $\dfrac{d}{dt}[A(t)x(t)]+B(t)x(t)$ of the DAE \eqref{DAE} or with the linear part $A(t)\dfrac{d}{dt}x(t)+B(t)x(t)$ of the DAE \eqref{DAE1}. The pencil associated with the linear part of the DAE is called \emph{characteristic}.

In the present paper, it is assumed that $\lambda A(t)+B(t)$ is a \emph{regular pencil of index not higher than~1} (index 0 or 1). This means that for each $t\ge t_+$ the pencil is regular, i.e., the set of its regular points is not empty (for the regular points $\lambda$ there exists the resolvent of pencil $(\lambda A(t)+B(t))^{-1}$), and there exist functions $C_1,\, C_2\colon [t_+,\infty)\to (0,\infty)$ such that for each $t\in[t_+,\infty)$ the pencil resolvent
 $$
R(\lambda,t):=(\lambda A(t)+B(t))^{-1}
 $$
satisfies the condition  (cf. \cite{Vlasenko1,Vlasenko2000})
 \begin{equation}\label{index1}
\|R(\lambda,t)\|\le C_1(t) \quad \text{for all $\la$ with} \quad  |\lambda|\ge C_2(t)
 \end{equation}
(hence, for any fixed $t$ the resolvent is bounded in a neighborhood of infinity). The condition \eqref{index1} means that either the point $\mu=0$ is a simple pole of the resolvent $(A(t)+\mu B(t))^{-1}$ of the pencil $A(t)+\mu B(t)$ (this is equivalent to the fact that $\lambda= \infty$ is a removable singularity of $R(\lambda,t)$), or $\mu = 0$ is a regular point of the pencil ${A(t)+\mu B(t)}$ (i.e., the operator $A(t)$ is nondegenerate).

If $\mu =0$ is a regular point of the pencil $A(t)+\mu B(t)$ for each $t$, then $\lambda A(t)+B(t)$ is a \emph{regular pencil of index~0}.\,
If $\mu =0$ is a simple pole of the resolvent $(A(t)+\mu B(t))^{-1}$ for each $t$, i.e., $A(t)$ is degenerate for all $t$ and the condition \eqref{index1} is satisfied,  then $\lambda A(t)+B(t)$ is a \emph{regular pencil of index~1}. In the general case, the definition for the index of a regular pencil is given below.

In \cite[Section 6.2]{Vlasenko1}, for the regular pencil $\lambda A+B$ of time-invariant square matrices $A$ and $B$, the maximum length of the chain of root vectors of the pencil $A+\mu B$ at the point $\mu=0$ is referred to as the \emph{index} of the pencil $\lambda A+B$.
This definition can be naturally associated with the following generalization (cf. \cite[Section 3.3.1]{Vlasenko1}, \cite{Vlasenko2000}) of the condition \eqref{index1}: Assume that the pencil $\lambda A(t)+B(t)$ is regular (for each $t\ge t_+$) and there exist functions $C_1,\, C_2\colon [t_+,\infty)\to (0,\infty)$ such that for each $t\in[t_+,\infty)$,
\begin{equation}\label{pencil_index}
\|R(\lambda,t)\|\le C_1(t)|\lambda|^{\nu-1} \quad \text{for all $\la$ with}\quad |\lambda|\ge C_2(t)\qquad (\nu\in{\mathbb N}),
\end{equation}
then $\lambda A(t)+B(t)$ is called a \emph{regular pencil of index not higher than $\nu$}.  It follows from \eqref{pencil_index} that
$$
\big\|\big(A(t)+\mu B(t)\big)^{-1}\big\|\le C_1(t)\, |\mu|^{-\nu}\quad \text{for all $\mu$ with}\quad 0 <|\mu|\le \big(C_2(t)\big)^{-1}.
$$
Accordingly, we define the index of a regular pencil $\lambda A(t)+B(t)$ in the following way.

\begin{definition}
Let $A,\, B\colon \mathcal{T}\to \mathrm{L}(X,Y)$ (or $A(t)$,~$B(t)$ are matrices corresponding to the linear operators $A(t),\,B(t)\in \mathrm{L}(X,Y)$ with respect to some bases in the spaces $X$, $Y$, which depend on the parameter $t\in \mathcal{T}$), where $X=Y=\Rn$ or $={\mathbb C}^n$ and $\mathcal{T}\subseteq\R$ is some interval. Also, let the operator (or matrix) pencil $\lambda A(t)+B(t)$, where $\lambda$ is a complex parameter, be regular for each $t\in \mathcal{T}$.
For a fixed $t\in \mathcal{T}$, if the point $\mu=0$ is a pole of the resolvent $(A(t)+\mu B(t))^{-1}$, then the order $\nu$ ($\nu\in {\mathbb N}$) of the pole is called the \textbf{\emph{index of the regular pencil $\lambda A(t)+B(t)$}}, and in the case when $\mu=0$ is a regular point of the pencil $A(t)+\mu B(t)$, the \emph{index of the regular pencil $\lambda A(t)+B(t)$ is $\nu=0$}.
If the pencil $\lambda A(t)+B(t)$ has index $\nu$ for each $t\in \mathcal{T}$, then we will say that $\lambda A(t)+B(t)$ is a \textbf{\emph{regular pencil of index $\nu$}} ($\nu\in{\mathbb N}\cup\{0\}$).
\end{definition}

Various notions of index of a matrix (operator) pencil, index of a DAE and a relationship between them are discussed in \cite[Remark~2.1]{Fil.MPhAG} and in detail in \cite{Kunkel_Mehrmann,Lamour-Marz-Tisch,Riaza}.

We fix $t$ and will give several comparisons with a time-invariant regular pencil $\lambda A+B$ of index 1 below.
First, note that if the regular pencil $\lambda A+B$ ($A$ is degenerate) has index 1 in the above sense,  then the \emph{index of nilpotency} of the matrix pencil $\lambda A+B$ equals~1 as defined in \cite[p.~717--718]{Gear_Petz84}. Further, the index of nilpotency of the matrix pencil $\lambda A+B$ is called the \emph{Kronecker index} of the regular matrix pair $\{A,B\}$ which forms the matrix pencil $\lambda A+B$ \cite[Def.~1.4]{Lamour-Marz-Tisch}.  In addition, the Kronecker index or nilpotency index of the regular pencil $\lambda A+B$ (the regular matrix pair~$\{A,B\}$ or matrix pencil $\lambda A+B$) coincides with the Kronecker index (see \cite[Def.~1.4]{Lamour-Marz-Tisch}) or the index (see \cite[p.~718]{Gear_Petz84}) of a regular linear time-invariant DAE  ${\frac{d}{dt}[Ax]+Bx=f(t)}$.
The linear time-invariant DAE with the regular pencil $\lambda A+B$ of index 1 in the sense as defined in the present paper is a regular linear DAE with \emph{tractability index}~1 \cite[Def.~2.25]{Lamour-Marz-Tisch} (\emph{for a semilinear DAE this correspondence does not exist}) and its Kronecker index as well as its index of nilpotency is~1.  Thus, for the regular pencil $\lambda A+B$ of index 1 all these notions give the same index value (equal to 1) for the linear time-invariant DAE as well as the pencil, regardless of how the index is defined.  In addition, the regular pencil $\lambda A+B$ of index not higher than 1 (as defined in the present paper) satisfies the ``rank-degree'' criterion defined in \cite[p.~30]{Chistyakov-Shcheglova}.

Note that the approach that uses the \emph{differentiation index} (see \cite{Brenan-C-P}) deals with the differentiation of algebraic equations in a differential-algebraic system and with the index of a DAE while \emph{we work with the index of the pencil associated with the linear part of the DAE and do not use manipulations with equations to find it}.
Moreover, we consider the case when the nonlinear function $f$ in the DAE \eqref{DAE} (or \eqref{DAE1}) may not be differentiable with respect to $t$ and, if $f$ is not differentiable, the differentiation index of the DAE cannot be determined.
The differentiation index, roughly speaking, is defined as the number of differentiations needed to reduce a DAE to an explicit ODE (the precise definition is given in \cite[Section~2]{Brenan-C-P}).
Similarly, when finding the \emph{strangeness index} \cite{Kunkel_Mehrmann} of a linear DAE (or a pair of matrices),  differentiation as well as additional operations are applied. For a linear time-invariant DAE the differentiation index
 (if it can be determined)
and the nilpotency (or Kronecker) index coincide.

 \subsubsection{Spectral projectors of the Riesz type and their application}\label{Sect-Projector}

Suppose that $\lambda A(t)+B(t)$ is a regular pencil of index not higher than $\nu$, i.e.,  the condition \eqref{pencil_index} holds, then for each $t\in[t_+,\infty)$ there exist the two pairs of mutually complementary projectors $P_1(t)$, $P_2(t)$ and $Q_1(t)$, $Q_2(t)$ (i.e., $P_i(t) P_j(t)= \delta_{ij} P_i(t)$, $P_1(t)+P_2(t)=I_{\Rn}$,\; $Q_i(t) Q_j(t)= \delta_{ij} Q_i(t)$, $Q_1(t)+Q_2(t)=I_{\Rn}$) which can be obtained by the
formulas  (see \cite{Vlasenko1,Vlasenko2000})
 \begin{equation}\label{Proj.1}
\begin{aligned}
& P_1(t)=\frac{1}{2\pi\ii}\, \oint\limits_{|\lambda|=C_2(t)} R(\lambda,t)A(t)\, d\lambda,\qquad
 && Q_1(t)=\frac{1}{2\pi\ii}\, \oint\limits_{|\lambda|=C_2(t)} A(t)R(\lambda,t)\, d\lambda, \\
& P_2(t)=I_{\Rn}-P_1(t),\qquad
 && Q_2(t)=I_{\Rn}-Q_1(t)
\end{aligned}
 \end{equation}
and generate the decompositions of $\Rn$ into the direct sums of subspaces
\begin{equation}\label{Proj.2}
 \Rn=X_1(t)\dot{+}X_2(t),\quad  X_j(t)=P_j(t)\Rn, \quad
 \Rn=Y_1(t)\dot{+}Y_2(t),\quad  Y_j(t)=Q_j(t)\Rn, \quad j=1,2,
\end{equation}
such that the pairs of subspaces  $X_1(t)$, $Y_1(t)$ and $X_2(t)$, $Y_2(t)$ are invariant with respect to $A(t)$, $B(t)$ (i.e., $A(t),\, B(t) \colon  X_j(t)\to Y_j(t)$); the restricted operators
$$
{A_j(t)=\left. A(t)\right|_{X_j(t)}\colon  X_j(t)\to Y_j(t)},\quad
B_j(t)=\left. B(t)\right|_{X_j(t)}\colon X_j(t) \to Y_j(t), \quad j=1,2,
$$
are such that the inverse operators $A_1^{-1}(t)$ and $B_2^{-1}(t)$ exist (if $X_1(t)\not=\{0\}$ and $X_2(t)\not=\{0\}$, respectively).
If the regular pencil has index not higher than~1, i.e., satisfies \eqref{index1}, then ${A_2(t)=0}$ and the subspaces $X_j(t)$, $Y_j(t)$ are such that
$$
Y_1(t)=\range(A(t)),\quad X_2(t)=\ker A(t),\quad  Y_2(t)=B(t)X_2(t),\quad X_1(t)=R(\lambda,t)Y_1(t),\; |\lambda|\ge C_2(t).
$$
The spectral projectors \eqref{Proj.1} are real (because $A(t)$ and $B(t)$ are real) and
$$
A(t)P_1(t)=Q_1(t) A(t)=A(t),\quad  A(t)P_2(t)=Q_2(t) A(t)=0,\quad B(t)P_j(t)=Q_j(t)B(t), \quad  j=1,2.
$$

Using the spectral projectors, for each  $t\in[t_+,\infty)$ we obtain the auxiliary operator \cite{Vlasenko1,Vlasenko2000}
\begin{equation}\label{G(t)}
G(t)=A(t)+B(t)P_2(t)=A(t)+Q_2(t)B(t)\in \mathrm{L}(\Rn)
\end{equation}
such that ${G(t) X_j(t)= Y_j(t)}$, ${j=1,2}$; it has the inverse ${G^{-1}(t)\in \mathrm{L}(\Rn)}$ such that
$$
G^{-1}(t)A(t)P_1(t)=G^{-1}(t)A(t)=P_1(t),\qquad  G^{-1}(t)B(t)P_2(t)=P_2(t).
$$

The projectors $P_i(t)$, $Q_i(t)$, $i=1,2$, and the operators $G(t)$, $G^{-1}(t)$  as operator functions have the same degree of smoothness as the operator functions $A(t)$, $B(t)$ and the function $C_2(t)$ from \eqref{index1} \cite[Section~3.3]{Vlasenko1}, \cite{Vlasenko2000}.

In what follows, we suppose that $A,\, B\in C^1([t_+,\infty),\mathrm{L}(\Rn))$ and $C_2\in C^1([t_+,\infty),(0,\infty))$, then $P_i,\, Q_i\in C^1([t_+,\infty),\mathrm{L}(\Rn))$,  $i=1,2$, and $G,\, G^{-1}\in C^1([t_+,\infty),\mathrm{L}(\Rn))$.  Since the projectors $P_i(t)$, $Q_i(t)$ are continuous (moreover, they are continuously differentiable) as operator functions for $t\in[t_+,\infty)$, the dimensions of the subspaces $X_i(t)=P_i(t)\Rn$ and $Y_i(t)=Q_i(t)\Rn$ ($i=1,2$) are constant (cf. \cite[p.~34]{Kato-eng}), and we denote ${\dim X_2(t)=\dim Y_2(t)=d}$ ($d=\mathrm{const}\ge 0$) and, accordingly,  $\dim X_1(t)=\dim Y_1(t)=n-d$, $t\in [t_+,\infty)$.
This means that the pencil $\lambda A(t)+B(t)$ has either index 0 for each $t$ or index 1 for each $t$\, ($t\in[t_+,\infty)$).

\paragraph{\textbf{Method of the spectral projectors for the reduction of a time-varying semilinear DAE to the system of an explicit ODE and an algebraic equation.}}

As shown in \cite{Fil.DE-1}, using the projectors $P_i(t)$, $Q_i(t)$ and the operator $G^{-1}(t)$, one can reduce the DAE \eqref{DAE} to the equivalent system of the explicit ODE \eqref{DAEsys1.1} (with respect to $P_1(t)x(t)$) and the algebraic equation \eqref{DAEsys1.2}:
\begin{align}
 & [P_1(t)x(t)]'=\big[P'_1(t)-G^{-1}(t)\, Q_1(t)\, [A'(t)+B(t)]\,\big] P_1(t)x(t)+ G^{-1}(t)\, Q_1(t) f(t,x(t)), \label{DAEsys1.1} \\
 & G^{-1}(t)\, Q_2(t) \big[f(t,x(t))-A'(t)P_1(t)x(t)\big]-P_2(t) x(t)=0. \label{DAEsys1.2}
\end{align}

For each $t$, with respect to the decomposition $\Rn=X_1(t)\dot{+}X_2(t)$ (see \eqref{Proj.2}), every $x\in\Rn$ can be uniquely represented as
\begin{equation}\label{xdecomp}
x=P_1(t)x+P_2(t)x =x_{p_1}(t)+x_{p_2}(t),\quad x_{p_i}(t)=P_i(t)x\in X_i(t),\quad i=1,2.
\end{equation}
Using \eqref{xdecomp} (i.e., $x_{p_i}(t)=P_i(t)x(t)$, $x(t)=x_{p_1}(t)+x_{p_2}(t)$), we rewrite the system \eqref{DAEsys1.1}, \eqref{DAEsys1.2} as
\begin{align}
 & x_{p_1}'(t)= \big[P'_1(t)-G^{-1}(t)\, Q_1(t)\, [A'(t)+ B(t)]\, \big]x_{p_1}(t) + G^{-1}(t)\, Q_1(t)f\big(t,x_{p_1}(t)+x_{p_2}(t)\big), \label{DAEsys2.1} \\
 & G^{-1}(t)\, Q_2(t)\big[f\big(t,x_{p_1}(t)+x_{p_2}(t)\big)-A'(t)x_{p_1}(t)\big] -x_{p_2}(t)=0. \label{DAEsys2.2}
\end{align}
The system \eqref{DAEsys2.1}, \eqref{DAEsys2.2} or \eqref{DAEsys1.1}, \eqref{DAEsys1.2} is a \emph{nonautonomous} (\emph{time-varying}) \emph{semi-explicit DAE}. Generally, systems of the form $\dot{y}=f(t,y,z)$, $0=g(t,y,z)$ are referred to as nonautonomous (time-varying) semi-explicit DAEs.
The time-varying semilinear DAE \eqref{DAE1} is reduced to a semi-explicit form  in a similar way (see \cite[the system (1.18) or (1.19),~(1.20)]{Fil.DE-1}).

 \subsection{Existence, uniqueness and boundedness of global solutions of time-varying semilinear DAEs}\label{GlobSolv}

Below we give definitions \cite{Fil.DE-1,Fil.MPhAG} that will be needed to formulate further results.

A solution $x(t)$ of the IVP \eqref{DAE}, \eqref{ini} is called \emph{global} if it exists on the interval $[t_0,\infty)$.\;
A solution $x(t)$ of the IVP \eqref{DAE}, \eqref{ini} is called \emph{Lagrange stable} if it is global and bounded, i.e., $x(t)$ exists on $[t_0,\infty)$ and $\sup\limits_{t\in [t_0,\infty)} \|x(t)\|<\infty$.

A solution $x(t)$ of the IVP \eqref{DAE}, \eqref{ini} is called \emph{Lagrange unstable} (a solution \emph{has a finite escape time} or \emph{is blow-up in finite time}) if it exists on some finite interval $[t_0,T)$ and is unbounded, i.e., there exists $T>t_0$ ($T<\infty$) such that $\mathop{\lim}\limits_{t\to T-0} \|x(t)\|=\infty$.

Equation \eqref{DAE} is called \emph{Lagrange stable} (\emph{Lagrange unstable}) \emph{for the initial point $(t_0,x_0)$} if the solution of the IVP \eqref{DAE}, \eqref{ini} is Lagrange stable (Lagrange unstable) for this initial point.

\emph{Equation \eqref{DAE}} is called \emph{Lagrange stable} (\emph{Lagrange unstable}) if each solution of the IVP \eqref{DAE}, \eqref{ini} is Lagrange stable (Lagrange unstable), that is, the equation is Lagrange stable (Lagrange unstable) for each consistent initial point.

Similar definitions hold for the DAE \eqref{DAE1} (the IVP \eqref{DAE1}, \eqref{ini}).

 \smallskip
\paragraph{\textbf{The existence, uniqueness and boundedness of global solutions of DAEs \eqref{DAE} and \eqref{DAE1}.}}
In what follows, the following notation will be used:
 \begin{equation*}
U_R^c(0)=\{z\in\Rn \mid \|z\|\ge R\}.
 \end{equation*}

  \begin{theorem}[Global solvability of the DAE \eqref{DAE} {\cite[Theorem~2.1]{Fil.DE-1}}]\label{Th_GlobSolv}\;
Let $f\in C([t_+,\infty)\times \Rn, \Rn)$, $\dfrac{\partial f}{\partial x} \in C([t_+,\infty)\times \Rn, \mathrm{L}(\Rn))$, $A, B\in C^1([t_+,\infty),\mathrm{L}(\Rn))$,\; the pencil $\lambda A(t)+B(t)$ satisfy \eqref{index1}, where $C_2\in C^1([t_+,\infty),(0,\infty))$, and the following conditions hold:
\begin{enumerate}[1.]
\item\label{SoglRN1} For each $t\in [t_+,\infty)$ and each $x_{p_1}(t)\in X_1(t)$ there exists a unique $x_{p_2}(t)\in X_2(t)$ such that
  \begin{equation}\label{soglreg2}
 (t,x_{p_1}(t)+x_{p_2}(t))\in L_{t_+}.
  \end{equation}

\item\label{InvRN1}  For each (fixed) ${t_*\in[t_+,\infty)}$, ${x_{p_1}^*(t_*)\in X_1(t_*)}$, ${x_{p_2}^*(t_*)\in X_2(t_*)}$ such that ${(t_*,x_{p_1}^*(t_*)+x_{p_2}^*(t_*)) \in L_{t_+}}$, the operator $\Phi_{t_*,x^*_{p_1}(t_*),x^*_{p_2}(t_*)}$ defined by
 \begin{equation}\label{funcPhi}
\Phi_{t_*,x^*_{p_1}(t_*),x^*_{p_2}(t_*)}=\!\left[\frac{\partial}{\partial x} \big[Q_2(t_*)\, f\big(t_*,x^*_{p_1}(t_*) +x^*_{p_2}(t_*)\big)\,\big]\!-B(t_*)\right]\! P_2(t_*)\colon X_2(t_*)\to Y_2(t_*)
 \end{equation}
   is invertible.

\item\label{ExtensRN1}  There exist functions ${k\in C([t_+,\infty),{\mathbb R})}$, ${U\in C(0,\infty)}$, a number ${R>0}$ and a positive definite function $V\in C^1\big([t_+,\infty)\times U_R^c(0),{\mathbb R}\big)$ such that ${\int\limits_{{\textstyle v}_0}^{\infty}\big(U(v)\big)^{-1}\, dv =\infty}$ \textup{(}$v_0>0$ is some constant\textup{)} and it holds that

  \smallskip
  3.1.~~${V(t,z)\to\infty}$ uniformly in $t$ on each finite interval $[a,b)\subset [t_+,\infty)$ as ${\|z\|\to\infty}$;

  \smallskip
  3.2.~~For all $t\in[t_+,\infty)$, $x_{p_1}(t)\in X_1(t)$, $x_{p_2}(t)\in X_2(t)$ such that  $(t,x_{p_1}(t)+x_{p_2}(t))\in L_{t_+}$ and ${\|x_{p_1}(t)\|\ge R}$, the following inequality holds:
     \begin{equation}\label{LagrDAE}
    V'_{\eqref{DAEsys2.1}}(t,x_{p_1}(t))\le k(t)\, U\big(V(t,x_{p_1}(t))\big),
     \end{equation}
    where
    \begin{multline*}
    V'_{\eqref{DAEsys2.1}}(t,x_{p_1}(t))= \frac{\partial V}{\partial t}(t,x_{p_1}(t))+ \bigg(\frac{\partial V}{\partial z}(t,x_{p_1}(t)), \big[P'_1(t)   \\
    - G^{-1}(t)Q_1(t)[A'(t)+B(t)]\big]x_{p_1}(t)+G^{-1}(t)Q_1(t) f\big(t,x_{p_1}(t)+x_{p_2}(t)\big)\!\bigg)
    \end{multline*}
    is the derivative of $V$ along the trajectories of equation \eqref{DAEsys2.1} \textup{(}where ${x_{p_1}(t)\!=\!z(t)}$\textup{)}.
\end{enumerate}
Then for each initial point  $(t_0,x_0)\in L_{t_+}$ there exists a unique global solution of the IVP \eqref{DAE}, \eqref{ini}.
 \end{theorem}

Equation \eqref{DAEsys2.1} can be written as ${x_{p_1}'(t)=\Pi(t,x_{p_1}(t),x_{p_2}(t))}$, where $\Pi$ has the form \eqref{Pi} (see Section \ref{Num-meth}), i.e.,
$$
\Pi(t,x_{p_1}(t),x_{p_2}(t))=\big[P'_1(t)-G^{-1}(t)\, Q_1(t)\, [A'(t)+ B(t)]\, \big]x_{p_1}(t) + G^{-1}(t)\, Q_1(t)f\big(t,x_{p_1}(t)+x_{p_2}(t)\big).
$$

Consider the equation
 \begin{equation}\label{DAEsys2.1-cut}
x_{p_1}'(t)= \Pi_{\scriptscriptstyle T}(t,x_{p_1}(t),x_{p_2}(t)),\quad \Pi_{\scriptscriptstyle T}(t,x_{p_1}(t),x_{p_2}(t))=
\begin{cases}
\Pi(t,x_{p_1}(t),x_{p_2}(t)), & t\in [t_+,T], \\
\Pi(T,x_{p_1}(T),x_{p_2}(T)), & t>T, \end{cases}
 \end{equation}
where $T>t_+$ is a parameter and the function $\Pi_{\scriptscriptstyle T}$ is the truncation of $\Pi$ over $t$.
  \begin{proposition}\label{St_GlobSolvSrez}
Theorem \ref{Th_GlobSolv} remains valid if condition \ref{ExtensRN1} is replaced by the following:
\begin{enumerate}[1.]
\addtocounter{enumi}{2}
\item\label{ExtensRN1Srez}  There exists a positive definite function ${V\in C^1([t_+,\infty)\times U_R^c(0),\R)}$, where ${R>0}$ is some number, and a function ${U\in C(0,\infty)}$ satisfying the relation ${\int\limits_{{\textstyle v}_0}^{\infty}\big(U(v)\big)^{-1}\, dv =\infty}$ \textup{(}${v_0>0}$ is some constant\textup{)}, and for each ${T>0}$ there exists a number ${R_{\scriptscriptstyle T}\ge R}$ and a function ${k_{\scriptscriptstyle T}\in C([t_+,\infty),\R)}$ such that

   \smallskip
  3.1.~~${V(t,z)\to\infty}$ uniformly in $t$ on each finite interval $[a,b)\subset [t_+,\infty)$ as ${\|z\|\to\infty}$;

   \smallskip
  3.2.~~For all ${t\in[t_+,\infty)}$, ${x_{p_1}(t)\in X_1(t)}$, ${x_{p_2}(t)\in X_2(t)}$ such that $(t,x_{p_1}(t)+x_{p_2}(t))\in L_{t_+}$ and $\|x_{p_1}(t)\|\ge R_{\scriptscriptstyle T}$, the following inequality holds:
  $$
  V'_{\eqref{DAEsys2.1-cut}}(t,x_{p_1}(t))\le k_{\scriptscriptstyle T}(t)\, U\big(V(t,x_{p_1}(t))\big),
  $$
  where
  \begin{equation*}
  V'_{\eqref{DAEsys2.1-cut}}(t,x_{p_1}(t))=\frac{\partial V}{\partial t}(t,x_{p_1}(t)) + \bigg(\frac{\partial V}{\partial z}(t,x_{p_1}(t)),\, \Pi_{\scriptscriptstyle T}(t,x_{p_1}(t),x_{p_2}(t))\bigg)
  \end{equation*}
  is the derivative of $V$ along the trajectories of equation \eqref{DAEsys2.1-cut}.
\end{enumerate}
 \end{proposition}
 \begin{proof}
The proof of the above proposition is easily derived from the proof of Theorem \ref{Th_GlobSolv} (see \cite{Fil.DE-1}), since a solution of equation \eqref{DAEsys2.1-cut} with the truncation coincides with the solution of the original equation with the same initial values on the interval $[t_+,T]$ (where the right-hand sides of the equations coincide).
 \end{proof}

A system of $s$ pairwise disjoint projectors $\{\Theta_k\in \mathrm{L}(Z)\}_{k=1}^s$ (i.e., $\Theta_i\, \Theta_j =\delta_{ij}\, \Theta_i$) such that $\sum\limits_{k=1}^s\Theta _k=I_Z$, where $Z$ is an $s$-dimensional linear space and $I_Z$ is the identity operator in $Z$, is called an \emph{additive resolution of the identity} in $Z$ (cf. \cite{RF1}).
An operator function $\Phi\colon D\to \mathrm{L}(W,Z)$, where $W$, $Z$ are $s$-dimensional linear spaces and $D\subset W$,  is called \emph{basis invertible} on an interval $J\subset D$ if for some additive resolution of the identity $\{\Theta _k\}_{k=1}^s$ in $Z$ and for any set of elements $\{w^k\}_{k=1}^s\subset J$ the operator
 $$
\Lambda=\sum\limits_{k=1}^s\Theta_k\Phi(w^k)\in \mathrm{L}(W,Z)
 $$
has an inverse $\Lambda^{-1}\in \mathrm{L}(Z,W)$  (cf. \cite{RF1}).  The basis invertibility of $\Phi$ on $J$ implies its invertibility on $J$, i.e., the invertibility of the operator $\Phi (w^*)$ for each $w^*\in J$ . The converse is not true, except for the case when $W$, $Z$ are one-dimensional.
 \begin{theorem}[Global solvability of the DAE \eqref{DAE}, {\cite[Theorem~2.2]{Fil.DE-1}}]\label{Th_GlobSolvBInv}\;
Theorem~\ref{Th_GlobSolv} remains valid if conditions \ref{SoglRN1}, \ref{InvRN1} are replaced by the following:
\begin{enumerate}[1.]
\item\label{SoglRN2}  For each ${t\in [t_+,\infty)}$ and each ${x_{p_1}(t)\in X_1(t)}$ there exists ${x_{p_2}(t)\in X_2(t)}$ such that \eqref{soglreg2}.

\item\label{BasInvRN1} For each (fixed) ${t_*\in[t_+,\infty)}$, ${x_{p_1}^*(t_*)\in X_1(t_*)}$,
    ${x_{p_2}^i(t_*)\in X_2(t_*)}$ such that ${(t_*, x_{p_1}^*(t_*) + x_{p_2}^i(t_*))\in L_{t_+}}$, ${i=1,2}$,  the operator function  $\Phi_{t_*,x^*_{p_1}(t_*)}(x_{p_2}(t_*))$ defined by
 \begin{equation}\label{funcPhiBInv}
\begin{split}
\Phi_{t_*,x^*_{p_1}(t_*)}\colon & X_2(t_*)\to \mathrm{L}(X_2(t_*),Y_2(t_*)),  \\
& \Phi_{t_*,x^*_{p_1}(t_*)}(x_{p_2}(t_*))=\! \bigg[\frac{\partial}{\partial x}\! \big[Q_2(t_*)f\big(t_*,x^*_{p_1}(t_*)+ x_{p_2}(t_*)\big)\big] -B(t_*)\bigg]P_2(t_*),\!\!
\end{split}
\end{equation}
 is basis invertible on $[x_{p_2}^1(t_*),x_{p_2}^2(t_*)]$.
 \end{enumerate}
 \end{theorem}
 \begin{remark}\label{RemSmoothSol}
Theorems \ref{Th_GlobSolv} and \ref{Th_GlobSolvBInv} ensure the following smoothness for the components of a solution $x(t)$ of the DAE \eqref{DAE}:  $P_1(t)x(t)\in C^1([t_0,\infty),\Rn)$,  $P_2(t)x(t)\in C([t_0,\infty),\Rn)$.
If in these theorems $A,\, B\in C^{m+1}([t_+,\infty),\mathrm{L}(\Rn))$, the function $C_2\in  C^{m+1}([t_+,\infty),(0,\infty))$ in the condition \eqref{index1}, and $f\in C^m([t_+,\infty)\times\Rn,\Rn)$, where ${m\ge 1}$, then the solution $x(t)$ is such that $P_1(t)x(t)\in C^{m+1}([t_0,\infty),\Rn)$ and $P_2(t)x(t)\in C^m([t_0,\infty),\Rn)$.
\end{remark}

\begin{proposition}[{\cite[Assertion~2.1]{Fil.DE-1}}]\label{Th_GlobSolvContr}\;
Theorem~\ref{Th_GlobSolv} remains valid if conditions \ref{SoglRN1}, \ref{InvRN1} are replaced by the following condition: there exists a constant ${0\le\alpha<1}$ such that
\begin{equation}\label{GlobContr}
  \big\|G^{-1}(t)\, Q_2(t)\, f\big(t,x_{p_1}(t)+x_{p_2}^1(t)\big) - G^{-1}(t)\, Q_2(t)\, f\big(t,x_{p_1}(t)+x_{p_2}^2(t)\big)\big\|\le \alpha \big\|x_{p_2}^1(t)-x_{p_2}^2(t)\big\|
\end{equation}
for any $t\in[t_+,\infty)$, $x_{p_1}(t)\in X_1(t)$, $x_{p_2}^i(t)\in X_2(t)$, $i=1,2$.
\end{proposition}
A proposition similar to \ref{Th_GlobSolvContr} holds true for Theorem \ref{Th_GlobSolvBInv} and its conditions \ref{SoglRN2},~\ref{BasInvRN1}.  Note that if the conditions of Proposition~\ref{Th_GlobSolvContr} are satisfied, then the conditions of Theorems \ref{Th_GlobSolv}, \ref{Th_GlobSolvBInv} hold as well, and that these theorems impose weaker constraints on the functions in the DAE than Proposition~\ref{Th_GlobSolvContr}.

Recall that $\dim X_2(t)=\dim Y_2(t)=d=\mathrm{const}$, $\dim X_1(t)=\dim Y_1(t)=n-d$, $t\!\in\! [t_+,\infty)$.
 \begin{theorem}[Lagrange stability of the DAE \eqref{DAE} {\cite[Theorem~2.5]{Fil.DE-1}}]\label{Th_UstLagrDAE}\;
Let $f\in C([t_+,\infty)\times \Rn,\Rn)$, $\dfrac{\partial f}{\partial x}\in C([t_+,\infty)\times\Rn,\mathrm{L}(\Rn))$, $A, B\in C^1([t_+,\infty),\mathrm{L}(\Rn))$, and let the pencil $\lambda A(t)+B(t)$ satisfy \eqref{index1}, where $C_2\in C^1([t_+,\infty),(0,\infty))$. Assume that conditions \ref{SoglRN1},~\ref{InvRN1} of Theorem~\ref{Th_GlobSolv} or \ref{SoglRN2},~\ref{BasInvRN1} of Theorem~\ref{Th_GlobSolvBInv} are satisfied and the following conditions hold:
\begin{enumerate}[1.]
\addtocounter{enumi}{2}
\item\label{LagrV}  There exist functions ${k\in C([t_+,\infty),\R)}$, ${U\in C(0,\infty)}$, a positive definite function ${V\in C^1\big([t_+,\infty)\times U_R^c(0),\R\big)}$ and a number ${R>0}$  such that ${\int\limits_{t_+}^{\infty}k(t)\, dt<\infty}$,\, ${\int\limits_{{\textstyle v}_0}^{\infty}\big(U(v)\big)^{-1} dv =\infty}$ \textup{(}${v_0>0}$ is some constant\textup{)} and it holds that

    \smallskip
  3.1.~~${V(t,z)\to\infty}$ uniformly in $t$ on $[t_+,\infty)$ as ${\|z\|\to\infty}$;

  \smallskip
  3.2.~~For all $t\in[t_+,\infty)$, $x_{p_1}(t)\in X_1(t)$, $x_{p_2}(t)\in X_2(t)$ such that $(t,x_{p_1}(t)+x_{p_2}(t))\in L_{t_+}$ and $\|x_{p_1}(t)\|\ge R$ inequality \eqref{LagrDAE} holds.

  \smallskip
\item\label{LagrA}  Let one of the following conditions be satisfied:
 \begin{enumerate}[4a.]
 \item\label{LagrA1}  For all ${(t,x_{p_1}(t)+x_{p_2}(t))\in L_{t_+}}$, ${\|x_{p_1}(t)\|\le M<\infty}$ \textup{(}$M$ is an arbitrary constant\textup{)} it holds that
     $$
     \|G^{-1}(t) Q_2(t)\big[f(t,x_{p_1}(t)+ x_{p_2}(t))-A'(t)x_{p_1}(t)\big]\|\le K_M\!<\!\infty,
     $$
     where $K_M=K(M)$ is some constant.

 \item\label{LagrA1.2} For all $(t,x_{p_1}(t)+x_{p_2}(t))\in L_{t_+}$, $\|x_{p_1}(t)\|\le M<\infty$ \textup{(}$M$ is an arbitrary constant\textup{)}, it holds that
     $$
     \|x_{p_2}(t)\|\le K_M<\infty,
     $$
     where $K_M=K(M)$ is some constant.

 \item\label{LagrA1.3} For each ${t_*\in [t_+,\infty)}$, there exists ${\Tilde{x}_{p_2}(t_*)\in X_2(t_*)}$ such that for each ${x_{p_i}^*(t_*)\in X_i(t_*)}$, ${i=1,2}$,   which satisfy ${(t_*,x_{p_1}^*(t_*)+x_{p_2}^*(t_*))\in L_{t_+}}$, the operator function \eqref{funcPhiBInv} is basis invertible on ${(\Tilde{x}_{p_2}(t_*),x_{p_2}^*(t_*)]}$ and the corresponding inverse operator (i.e.,~$\Big[\sum\limits_{k=1}^d\tilde{\Theta}_k(t_*) \Phi_{t_*,x_{p_1}^*(t_*)}(x_{p_2,k}(t_*))\Big]^{-1}$ where $\{x_{p_2,k}(t_*)\}_{k=1}^d$ is an arbitrary set of the elements from $(\Tilde{x}_{p_2}(t_*),x_{p_2}^*(t_*)]$ and $\{\tilde{\Theta}_k(t_*)\}_{k=1}^d$ is some additive resolution of the identity in~$Y_2(t_*)$\,) is bounded uniformly in $t_*,\, x_{p_2}(t_*)$ (i.e., in $t_*,\,x_{p_2,k}(t_*)$, ${k=1,\dots,d}$)\,  on $[t_+,\infty)$, $(\Tilde{x}_{p_2}(t_*), x_{p_2}^*(t_*)]$. In addition, let   $\sup\limits_{t_*\in [t_+,\infty)}\|\Tilde{x}_{p_2}(t_*)\|<\infty$ and
     $$
     \sup\limits_{t_*\in [t_+,\infty),\, \|x_{p_1}^*(t_*)\|\le M} \|G^{-1}(t_*)Q_2(t_*) \big[f(t_*,x_{p_1}^*(t_*)+ \Tilde{x}_{p_2}(t_*))-A'(t_*)x_{p_1}^*(t_*)\big]\|<\infty
     $$
     ($M$ is an arbitrary constant).
   \end{enumerate}
\end{enumerate}
Then the DAE \eqref{DAE} is Lagrange stable.
 \end{theorem}
Note that if condition~\ref{LagrV} of Theorem~\ref{Th_UstLagrDAE} holds, then condition~\ref{ExtensRN1} of Theorem~\ref{Th_GlobSolv} holds.
From this remark we obtain the following corollary.
 \begin{corollary}\label{CorUstLagrGlobSolv}
If all the conditions of Theorem~\ref{Th_UstLagrDAE} except \ref{LagrA} are fulfilled, then the conditions of Theorem~\ref{Th_GlobSolv} or~\ref{Th_GlobSolvBInv} (depending on whether conditions~\ref{SoglRN1},~\ref{InvRN1} of Theorem~\ref{Th_GlobSolv} or conditions~\ref{SoglRN2},~\ref{BasInvRN1} of Theorem~\ref{Th_GlobSolvBInv} are fulfilled)\, are satisfied and, consequently, for each initial point $(t_0,x_0)\in L_{t_+}$ there exists a unique global solution of the IVP \eqref{DAE}, \eqref{ini}.
 \end{corollary}

The theorem \cite[\emph{Theorem~2.9}]{Fil.DE-1} \emph{on the Lagrange instability of the DAE \eqref{DAE} gives conditions under which the DAE is Lagrange unstable}, and therefore does not have global solutions, \emph{for consistent initial points $(t_0,x_0)$ such that the component $P_1(t_0)x_0$ belongs to a certain region}.

 \begin{remark}\label{RemSmoothSol1}
The theorems \cite[Theorems 2.3 and 2.4]{Fil.DE-1} on the global solvability of the DAE \eqref{DAE1} guarantee that its solution $x\in C^1([t_0,\infty),\Rn)$.  If in these theorems $A,\, B\in  C^m([t_+,\infty),\mathrm{L}(\Rn))$, the function $C_2\in   C^m([t_+,\infty),(0,\infty))$ in \eqref{index1} and $f\in  C^m([t_+,\infty)\times\Rn,\Rn)$, where ${m\ge 1}$, then the solution $x\in  C^m([t_0,\infty),\Rn)$.
 \end{remark}

\textbf{Remark on the choice of the function $V$ in the above theorems (see \cite[Section~4]{Fil.DE-2}).}\quad
First, note that {\cite[Section~2]{Fil.DE-1}} provides the theorems on the global solvability, Lagrange stability and instability and ultimate boundedness (dissipativity)  of DAEs \eqref{DAE} and \eqref{DAE1}. A scalar function $V(t,z)$ is called a \emph{Lyapunov type function} if it satisfies one of the theorems mentioned above. The function
 \begin{equation}\label{funcV}
V(t,z)=(H(t)z,z),
 \end{equation}
where $H\in C^1([t_+,\infty),\mathrm{L}(\Rn))$ is a positive definite self-adjoint operator function (see \cite[Definition~1.1]{Fil.DE-1}),
satisfies the conditions of Theorems \ref{Th_GlobSolv}--\ref{Th_UstLagrDAE} on the global solvability and Lagrange stability and the theorem \cite[Theorem~2.9]{Fil.DE-1} on the Lagrange instability, however, the conditions on the derivative $V'_{\eqref{DAEsys2.1}}(t,x_{p_1}(t))$ remain in the theorems and need to be verified. A similar statement holds for the corresponding theorems obtained for the DAE \eqref{DAE1} in \cite{Fil.DE-1}. \cite[Section~4]{Fil.DE-2} provides detailed comments on the application of the function \eqref{funcV} when checking the conditions of the theorems presented in \cite{Fil.DE-1,Fil.DE-2}.  For $V$ of the form  \eqref{funcV}, the derivative $V'_{\eqref{DAEsys2.1}}(t,x_{p_1}(t))$ (presented in \eqref{LagrDAE}) has the form
 \begin{equation}  \label{VderivDAE}
\begin{aligned}
V'_{\eqref{DAEsys2.1}}(t,x_{p_1}(t))= \Big(H'(t)x_{p_1}(t),x_{p_1}(t)\Big)+ 2 \Big(H(t)x_{p_1}(t),\, \big[P'_1(t)-G^{-1}(t)Q_1(t)\, [A'(t)  & \\
+ B(t)]\,\big]x_{p_1}(t) +G^{-1}(t)Q_1(t)\, f\big(t,x_{p_1}(t)+x_{p_2}(t)\big)\Big). &
\end{aligned}
 \end{equation}

 \section{Combined numerical methods: the construction, convergence and orders}\label{NumMs}

 \subsection{Calculation of the spectral projectors and preliminaries}\label{NumProj}

One of the advantages of the numerical methods proposed in this paper is the possibility to numerically find the spectral projectors $P_i(t)$, $Q_i(t)$,  which enables one to numerically solve the DAE in the original form \eqref{DAE}. To calculate the spectral projectors, residues can be used (see \eqref{ProjRes} below). Recall that, by assumption, the pencil $\lambda A(t)+B(t)$ is either a regular pencil of index~1 or a regular pencil of index~0.

The definition of the index of a regular pencil is given in Section~\ref{Sect-Index}.

Now, suppose that $\lambda A(t)+B(t)$ is a regular pencil of index $\nu$\, (${\nu\in{\mathbb N}\cup\{0\}}$). It follows from \eqref{Proj.1} that  ${P_1(t)=\frac{1}{2\pi\ii} \oint\limits_{|\mu|=1/C_2(t)}\frac{(A(t)+\mu B(t))^{-1}A(t)}{\mu}d\mu}$ and ${Q_1(t)=\frac{1}{2\pi\ii} \oint\limits_{|\mu|=1/C_2(t)}\frac{A(t)(A(t)+\mu B(t))^{-1}}{\mu}d\mu}$.
Thus, for each $t\in[t_+,\infty)$ the projectors \eqref{Proj.1} can be calculated by using residues:
\begin{equation}\label{ProjRes}
\begin{aligned}
& P_1(t)=\res\limits_{\mu =0}\left(\dfrac{(A(t)+\mu B(t))^{-1}A(t)}{\mu} \right),\qquad
 && Q_1(t)=\res\limits_{\mu =0}\left(\dfrac{A(t)(A(t)+\mu B(t))^{-1}}{\mu} \right), \\
& P_2(t)=I_{\Rn}-P_1(t),\qquad   && Q_2(t)=I_{\Rn}-Q_1(t).
\end{aligned}
\end{equation}
Denote
$$
{R_p(\mu,t):=\dfrac{(A(t)+\mu B(t))^{-1}A(t)}{\mu}},\qquad {R_q(\mu,t):=\dfrac{A(t)(A(t)+\mu B(t))^{-1}}{\mu}},
$$
then
$$
{P_1(t)=\res\limits_{\mu=0}R_p(\mu,t)}, \qquad {Q_1(t)=\res\limits_{\mu =0}R_q(\mu,t)},
$$
where $t$ is a parameter.
It is clear that the function $(R_p)_{ij}(\mu,t)$ ($i,j=1,...,n$), i.e., $(i,j)$-entry of the matrix $R_p(\mu,t)$, is a rational function in $\mu$, and if $\mu=0$ is its pole of order $k$, then $(R_p)_{ij}(\mu,t)=\varphi(\mu,t)\mu^{-k}$, where $\varphi(\mu,t)$ is a polynomial in $\mu$ such that $\varphi(0,t)\ne 0$ (similarly for $(R_q)_{ij}(\mu,t)$).

Thus, \emph{the main steps of the algorithm used for the calculation of the projectors \eqref{ProjRes} are as follows:}
 \begin{enumerate}[Step 1.]
 \setlength\itemsep{2mm}
\item\label{step1} For $i,j=1,...,n$: \\
      1.1.~determine the order $k$ of a pole of the function $(R_p)_{ij}(\mu,t)$ at the point $\mu=0$; to do this, we convert $(R_p)_{ij}(\mu,t)$ to a rational form (with respect to $\mu$) and determine the order $k$ of the zero of the denominator at the point $\mu=0$, at that, if the denominator does not have a zero at $\mu=0$, then the order $k=0$; \\
      1.2.~if the order $k\ge 1$, then
      $$
      {(P_1)_{ij}(t)=\res\limits_{\mu=0}(R_p)_{ij}(\mu,t)= \dfrac{1}{(k-1)!} \dfrac{d^{k-1}}{d\mu^{k-1}} \big[(R_p)_{ij}(\mu,t)\, \mu^k\big]\bigg|_{\mu=0}},
      $$
      and if the order $k=0$, then $(P_1)_{ij}(t)=0$.\quad Finally, $P_1(t)=\big((P_1)_{ij}(t)\big)_{1\le i,j\le n}$.
\item\label{step2} Perform the same as in Step \ref{step1}, replacing ${R_p(\mu,t)=\big((R_p)_{ij}(\mu,t)\big)_{1\le i,j\le n}}$ with ${R_q(\mu,t)=\big((R_q)_{ij}(\mu,t)\big)_{1\le i,j\le n}}$ and, accordingly, $P_1(t)$ with $Q_1(t)=\big((Q_1)_{ij}(t)\big)_{1\le i,j\le n}$.
\item\label{step3} Having calculated $P_1(t)$ and $Q_1(t)$, find $P_2(t)=I_{\Rn}-P_1(t)$ and $Q_2(t)=I_{\Rn}-Q_1(t)$. \smallskip
 \end{enumerate}
In the case when the index of the pencil is 0, we have ${P_1(t)\equiv I_{\Rn}}$, ${P_2(t)\equiv 0}$, and ${Q_1(t)\equiv I_{\Rn}}$, ${Q_2(t)\equiv 0}$. In this case, the operator $A(t)$ is nondegenerate (i.e., invertible) for each $t\in[t_+,\infty)$ and the DAE \eqref{DAE} (as well as \eqref{DAE1}) can be reduced to an ODE.
Obviously, in the case when $A(t)\equiv 0$ and $B(t)$ is nondegenerate for each $t$, which corresponds to a particular case of the pencil $\lambda A(t)+B(t)$ of index~1, we have $P_1(t)\equiv 0$, $P_2(t)\equiv I_{\Rn}$, and $Q_1(t)\equiv 0$, $Q_2(t)\equiv I_{\Rn}$, and, in this case, the DAE is a purely algebraic equation (i.e., do not contain the derivative).

Note that in Step \ref{step1} of the above algorithm one can use another (but similar) approach which involves the calculation of the entire projector at once and the \emph{second option of Step \ref{step1} is as follows}:
\begin{enumerate}
 \setlength\itemsep{1mm}
\item[ ]
   1.1.~determine the order $\nu$ of a pole of $R_p(\mu,t)$ at the point $\mu=0$ (to do this, we can determine the order $k_{ij}=k$ of a pole of  $(R_p)_{ij}(\mu,t)$ at $\mu=0$ as described in item~1.1 of Step \ref{step1} above, and set $\nu=\max\limits_{i,j=1,...,n}\{k_{ij}\}$); \\
   1.2.~if the order $\nu\ge 1$, then $$
   {P_1(t)=\res\limits_{\mu=0}R_p(\mu,t)=\dfrac{1}{(\nu-1)!} \dfrac{d^{\nu-1}}{d\mu^{\nu-1}} \big[R_p(\mu,t)\,\mu^\nu\big]\bigg|_{\mu=0}},
   $$
   and if $\nu=0$, then $P_1(t)=0$.
\end{enumerate}

After computing the projectors, the \emph{auxiliary operator $G(t)$ is calculated by the formula \eqref{G(t)}}.

For the special cases when $A(t)$ is nondegenerate or is zero (for all $t$), the results obtained herein remain valid, but since the purpose was to construct numerical methods for the DAEs, we carry out further proofs for the case when $A(t)$ is degenerate but not identically zero and $\lambda A(t)+B(t)$ is a regular pencil of index~1, without comments on the form of the conditions for the special cases.

 \begin{remark}\label{RemConsistInis}
(cf. \cite[Remark~1.2]{Fil.DE-1})
Introduce the manifold
\begin{equation} \label{L_t}
L_{t_+} = \{(t,x)\in [t_+,\infty)\times\Rn \mid Q_2(t)[A'(t)P_1(t)x+B(t)x-f(t,x)]=0\}
\end{equation}
(in \eqref{L_t} the number $t_+$ is a parameter). The \emph{consistency condition $(t_0,x_0)\in L_{t_+}$}   for the initial point $(t_0,x_0)$ is one of the necessary conditions for the existence of a solution of the IVP \eqref{DAE}, \eqref{ini}. An initial point $(t_0,x_0)$ satisfying this condition is called a \emph{consistent initial point} and the corresponding initial values  $t_0$,~$x_0$ are called \emph{consistent initial values}.
 \end{remark}

We discretize the IVP \eqref{DAE}, \eqref{ini} on the uniform mesh
\begin{equation}\label{mesh}
\omega_h=\{t_i=t_0+ih,\; i=0,...,N-1,\; t_N=t_0+Nh=T\}
\end{equation}
with the step $h=(T-t_0)/N$\, (\,$[t_0,T]\subset[t_0,\infty)$ is any given interval). The value of an approximate solution at the point $t_i$ is denoted by $x_i$ ($i=0,...,N$).

An initial value $t_0$ is given, and an initial value $x_0$ for the IVP \eqref{DAE}, \eqref{ini} is chosen so that the consistency condition
$$
Q_2(t_0)\big[A'(t_0)P_1(t_0)x_0+B(t_0)x_0-f(t_0,x_0)\big]=0,\qquad \text{i.e.,}\quad (t_0,x_0)\in L_{t_+},
$$
is fulfilled. Accordingly, the initial values $t_0$ and $z_0=P_1(t_0)x_0$, $u_0=P_2(t_0)x_0$ satisfy the condition $(t_0,z_0+u_0)\in L_{t_+}$ as well.  The consistency condition for the initial values $t_0$,~$x_0$ ensures the best choice of initial values for the developed combined methods (more precisely, for the methods applied to the ``algebraic part'' of the DAE, which are combined with those applied to the ``differential part'').

   \smallskip
Below, Theorems \ref{Th_GlobSolv}, \ref{Th_GlobSolvBInv} ensuring the existence of a unique exact solution of the IVP \eqref{DAE}, \eqref{ini} and other results which are presented in Section \ref{GlobSolv} are used.

   \subsection{Method 1 (the simple combined method)}\label{Num-meth}

 \begin{theorem}\label{ThNum-meth}
Let the conditions of Theorem~\ref{Th_GlobSolv} or~\ref{Th_GlobSolvBInv} be satisfied, and in addition let the operator
$$
{\Phi_{t_*,P_1(t_*)z_*,P_2(t_*)u_*}= \Phi_{t_*,P_1(t_*)z_*}(P_2(t_*)u_*)\colon X_2(t_*)\to Y_2(t_*)},
$$
which is defined by \eqref{funcPhi} or \eqref{funcPhiBInv} for each (fixed) $t_*$,\, ${x^*_{p_1}(t_*)=P_1(t_*)z_*}$,\, ${x^*_{p_2}(t_*)=P_2(t_*)u_*}$,\,  be invertible for each point ${(t_*,P_1(t_*)z_*+P_2(t_*)u_*)\in [t_0,T]\times\Rn}$.
Let $A,\, B\in C^2([t_0,T],\mathrm{L}(\Rn))$, $C_2\in C^2([t_0,T],(0,\infty))$ (recall that the function $C_2(t)$ was introduced in \eqref{index1}), $f\in C^1([t_0,T]\times \Rn,\Rn)$ and let an initial value $x_0$ be chosen so that the consistency condition $(t_0,x_0)\in L_{t_+}$ holds. Then the method
\begin{align}
& z_0 \!= P_1(t_0)x_0,\quad u_0 \!= P_2(t_0)x_0, \label{met1}\\
& z_{i+1} \!= \Big(I_{\Rn} +h\big[P'_1(t_i)- G^{-1}(t_i)Q_1(t_i)[A'(t_i)+ B(t_i)] \big]P_1(t_i)\Big) z_i+  h\, G^{-1}(t_i)Q_1(t_i) f\big(t_i,x_i\big), \label{met2}\\
 \begin{split}
& u_{i+1} \!= u_i- \! \bigg[I_{\Rn}- G^{-1}(t_{i+1})Q_2(t_{i+1})\dfrac{\partial f}{\partial x}\big(t_{i+1},P_1(t_{i+1})z_{i+1}+ P_2(t_{i+1})u_i\big)\, P_2(t_{i+1})\bigg]^{-1}\bigg[u_i \\
&\qquad  -G^{-1}(t_{i+1})Q_2(t_{i+1}) \Big[f\big(t_{i+1},P_1(t_{i+1})z_{i+1}+ P_2(t_{i+1})u_i\big) - A'(t_{i+1})P_1(t_{i+1})z_{i+1}\Big]\bigg],
 \end{split} \label{met3}\\
& x_{i+1} \!=  P_1(t_{i+1})z_{i+1}+P_2(t_{i+1})u_{i+1},\quad t_{i+1}\in \omega_h,\quad  i=0,...,N-1, \label{met4}
\end{align}
which approximates the IVP \eqref{DAE}, \eqref{ini} on $[t_0,T]$, is convergent of order 1, that is,
$$
\max\limits_{1\le i\le N}\|x(t_i)-x_i\|= O(h),\quad h\to 0,
$$
where $x(t)$ is an exact solution of the IVP \eqref{DAE}, \eqref{ini} ($x_i$ is the value of an approximate solution at~$t_i$).
\end{theorem}

 \begin{proof}
Take any initial point $(t_0,x_0)\in L_{t_+}$ (i.e., $Q_2(t_0)\big[A'(t_0)P_1(t_0)x_0+B(t_0)x_0-f(t_0,x_0)\big]=0$).  By virtue of the theorem conditions, taking into account Remark \ref{RemSmoothSol}, we obtain that for each initial point $(t_0,x_0)\in L_{t_+}$ there exists a unique global (exact) solution $x(t)$ of the IVP \eqref{DAE}, \eqref{ini} such that $z(t)=P_1(t)x(t)\in C^2([t_0,T],\Rn)$ and $u(t)=P_2(t)x(t)\in C^1([t_0,T],\Rn)$ ($z\in C^1([t_0,\infty),\Rn)$, $u\in  C([t_0,\infty),\Rn)$ and $z(t)\in X_1(t)$, $u(t)\in X_2(t)$).

The DAE \eqref{DAE} is equivalent to the system \eqref{DAEsys2.1}, \eqref{DAEsys2.2} which can be written in the form:
 \begin{align*}
x_{p_1}'(t) &= \big[P'_1(t)-G^{-1}(t)\, Q_1(t)\, [A'(t)+ B(t)]\, \big]x_{p_1}(t) + G^{-1}(t)\, Q_1(t)f\big(t,x_{p_1}(t)+x_{p_2}(t)\big),  \\
x_{p_2}(t) &=G^{-1}(t)\, Q_2(t)\big[f\big(t,x_{p_1}(t)+x_{p_2}(t)\big)-A'(t)x_{p_1}(t)\big].
 \end{align*}

Let us introduce the mappings $\Pi,\, W\colon [t_+,\infty)\times\Rn\times \Rn\to \Rn$ of the following form:
 \begin{align}
&\Pi(t,z,u):=\! \big[P'_1(t)-G^{-1}(t)Q_1(t)\,[A'(t)+B(t)]\, \big] P_1(t)z + G^{-1}(t)Q_1(t)\, f(t,P_1(t)z+P_2(t)u),  \label{Pi} \\
&W(t,z,u):= G^{-1}(t)Q_2(t)\big[f(t,P_1(t)z+P_2(t)u)- A'(t)P_1(t)z\big]. \label{W}
 \end{align}
These mappings are continuous in $(t,z,u)$ and have continuous partial derivatives with respect to $z$, $u$ on $[t_+,\infty)\times \Rn\times \Rn$ due to the conditions of Theorem~\ref{ThNum-meth}, as well as Proposition~\ref{remNum-meth} presented below, and, in addition, they have a continuous partial derivative with respect to $t$ on $[t_+,\infty)\times \Rn\times \Rn$ due to the conditions of the theorem.

Consider the system
\begin{align}
 & z'(t)=\Pi(t,z(t),u(t)), \label{num1} \\
 & u(t)=W(t,z(t),u(t)). \label{num2}
\end{align}

Below is the lemma that was proved in the paper \cite{Fil.DE-1} (in this paper, equation \eqref{num2} was written in the form $F(t,z(t),u(t))=0$, where $F(t,z(t),u(t))=W(t,z(t),u(t))-u(t)$).
 \begin{lemma}[{\cite[Lemma 2.1]{Fil.DE-1}}]\label{LemSolEquiv}\,
If a function $x(t)$ is a solution of the DAE \eqref{DAE} on $[t_0,t_1)$ and satisfies the initial condition \eqref{ini},  then the functions $z(t)=P_1(t)x(t)$, $u(t)=P_2(t)x(t)$ are a solution of the system \eqref{num1}, \eqref{num2} on $[t_0,t_1)$, satisfy the initial conditions $z(t_0)=P_1(t_0)x_0$, $u(t_0)=P_2(t_0)x_0$, and $z\in C^1([t_0,t_1),\Rn)$, $u\in C([t_0,t_1),\Rn)$.

Conversely, if functions $z\in C^1([t_0,t_1),\Rn)$, $u\in C([t_0,t_1),\Rn)$ are a solution of the system  \eqref{num1}, \eqref{num2} on $[t_0,t_1)$ and satisfy the initial conditions $z(t_0)=P_1(t_0)x_0$, $u(t_0)=P_2(t_0)x_0$, then $P_1(t)z(t)=z(t)$, $P_2(t)u(t)=u(t)$ and the function $x(t)=z(t)+u(t)$ is a solution of the DAE \eqref{DAE} on $[t_0,t_1)$ and satisfies the initial condition \eqref{ini}.
 \end{lemma}

Note that if $u(t)\in\Rn$ satisfies relation \eqref{num2}, then $u(t)\in X_2(t)$ (i.e., $u(t)=P_2(t)u(t)$).

Using equality \eqref{num2}, where $t$ is replaced by $t+h$, and the Taylor expansion
 \begin{equation}\label{Taylor2}
\begin{split}
W\big(t+h,z(t+h),u(t+h)\big) &= W\big(t+h,z(t+h),u(t)\big) \\
 & +\dfrac{\partial W}{\partial u}\big(t+h,z(t+h),u(t)\big)\big[u(t+h)-u(t)\big]+O(h),
\end{split}
 \end{equation}
\begin{equation}\label{dW_du}
\dfrac{\partial W}{\partial u}\big(t+h,z(t\!+\! h),u(t)\big)\!=\! G^{-1}(t\!+\! h)Q_2(t\!+\! h) \dfrac{\partial f}{\partial x}\big(t\!+\! h,P_1(t\!+\! h)z(t\!+\! h)\!+\! P_2(t\!+\! h)u(t)\big) P_2(t\!+\! h),
\end{equation}
we obtain the relation
 \begin{equation}\label{approx_u-1}
\begin{split}
u(t+h) & =\! \bigg[I_{\Rn}-\dfrac{\partial W}{\partial u}\big(t+h,z(t+h),u(t)\big) \bigg]^{-1} \bigg[W\big(t+h,z(t+h),u(t)\big) \\
 & -\dfrac{\partial W}{\partial u}\big(t+h,z(t+h),u(t)\big) u(t) +O(h)\bigg].
\end{split}
 \end{equation}
Relation \eqref{approx_u-1} can be rewritten as
 \begin{equation}\label{approx_u-2}
\begin{split}
u(t+h) &= u(t)- \bigg[I_{\Rn}- \dfrac{\partial W}{\partial u}\big(t+h,z(t+h),u(t)\big)\bigg]^{-1} \Big[u(t) \\
 & -G^{-1}(t+ h)Q_2(t+h) \Big(f\big(t+h,P_1(t+h)z(t+h)+P_2(t+h)u(t)\big) \\
 & -A'(t+h)P_1(t+h)z(t+h)\Big)\Big]+ O(h),
\end{split}
 \end{equation}
where $\dfrac{\partial W}{\partial u}\big(t+h,z(t+h),u(t)\big)$ is defined in \eqref{dW_du}.
As a result, for the algebraic equation (AE) \eqref{num2} we obtain a method similar to the Newton method with respect to the component $u$ of the phase variable $x=z+u$.
The existence of the inverse operator used in the relations \eqref{approx_u-1} and  \eqref{approx_u-2} follows from the following statement: From the invertibility of the operator $\Phi_{t,P_1(t)z,P_2(t)u}$ (if in the theorem conditions it is assumed that the requirements of Theorem~\ref{Th_GlobSolv} are satisfied)  and the basis invertibility of the operator function $\Phi_{t,P_1(t)z}(P_2(t)u)$ (if in the theorem conditions it is assumed that the requirements of Theorem~\ref{Th_GlobSolvBInv} are satisfied) for any fixed $t\in [t_0,\infty)$, $z\in\Rn$, $P_2(t)u\in X_2(t)$ such that $F(t,z,P_2(t)u)=0$ (i.e., $(t,P_1(t)z+P_2(t)u)\in L_{t_0}$) and the invertibility of $\Phi_{t,P_1(t)z,P_2(t)u}$ and $\Phi_{t,P_1(t)z}(P_2(t)u)$ for any fixed point  $(t,P_1(t)z+P_2(t)u)\in [t_0,T]\times \Rn$ it follows that there exists the inverse operator (if the requirements of Theorem~\ref{Th_GlobSolv} hold)
 \begin{equation}\label{InvNum}
\begin{split}
& \bigg[I_{\Rn} -\dfrac{\partial W}{\partial u}\big(t+h,z(t+h),u(t)\big)\bigg]^{-1}   \\
 & =\!\! \bigg[I_{\Rn} -G^{-1}(t+h)Q_2(t+h)\dfrac{\partial f}{\partial x}\big(t+h,P_1(t+h)z(t+h)+P_2(t+h)u(t)\big)\, P_2(t+h)\bigg]^{-1}    \\
  &\qquad  =P_1(t+h)-\left[\Phi_{t+h,P_1(t+h)z(t+h),P_2(t+h)u(t)}\right]^{-1} G(t+h)P_2(t+h)\in \mathrm{L}(\Rn),
\end{split}
 \end{equation}
where $\Phi_{t,P_1(t)z,P_2(t)u}$ is the operator \eqref{funcPhi}, and the inverse operator (if the requirements of Theorem~\ref{Th_GlobSolvBInv} hold)
 \begin{equation}\label{BasInvNum}
\begin{split}
\bigg[I_{\Rn} &-\dfrac{\partial W}{\partial u}\big(t+h,z(t+h),u(t)\big)\bigg]^{-1} \\
&=P_1(t+h)-\left[\Phi_{t+h,P_1(t+h)z(t+h)}(P_2(t+h)u(t))\right]^{-1} G(t+h)P_2(t+h)\in \mathrm{L}(\Rn),
\end{split}
 \end{equation}
where $\Phi_{t,P_1(t)z}(P_2(t)u)$ is the operator \eqref{funcPhiBInv}, for the points $(t+h,P_1(t+h)z(t+h)+P_2(t+h)u(t))$ belonging to $L_{t_0}$ and $[t_0,T]\times \Rn$.
For clarity, we note that the inverse operators on the left sides of \eqref{InvNum} and \eqref{BasInvNum} are the same, but the formulas for them are written through $\Phi_{t,P_1(t)z,P_2(t)u}$ and $\Phi_{t,P_1(t)z}(P_2(t)u)$ respectively.

Using the representation
 \begin{equation}\label{Taylor1}
\frac{dz}{dt}(t)=\frac{z(t+h)-z(t)}{h}+O(h),\quad h\to 0,
 \end{equation}
we obtain (an analog of the explicit Euler method for the DE \eqref{num1})
 \begin{equation}\label{approx_z}
\begin{split}
 z(t+h)=z(t) &+h\, \Pi(t,z(t),u(t))+O(h^2)= \Big(I_{\Rn} +h\big[P'_1(t)-G^{-1}(t)Q_1(t) [A'(t) \\
 &+B(t)] \big]P_1(t)\Big) z(t)+h\, G^{-1}(t)Q_1(t) f\big(t,P_1(t)z(t)+P_2(t)u(t)\big)+O(h^2).
\end{split}
 \end{equation}

Taking into account the obtained equalities \eqref{approx_z}, \eqref{approx_u-2} and Lemma \ref{LemSolEquiv}, we write the IVP \eqref{DAE}, \eqref{ini} at the points $\{t_i\}_{i=0,...,N}$ of the introduced mesh $\omega_h$ in the form:
 \begin{align}
& z(t_0) \!=\! P_1(t_0)x_0,\quad u(t_0)\!=\! P_2(t_0)x_0,  \nonumber\\
 \begin{split}
& z(t_{i+1}) \!=\! z(t_i)+h\, \Pi(t_i,z(t_i),u(t_i))+O(h^2)\!=\!  \Big(I_{\Rn} +h\big[P'_1(t_i)   \\
 &\qquad\quad -G^{-1}(t_i)Q_1(t_i)[A'(t_i)+B(t_i)] \big]P_1(t_i)\Big) z(t_i)+h\, G^{-1}(t_i)Q_1(t_i) f\big(t_i,x(t_i)\big)+O(h^2),
 \end{split} \label{exact2}\\
 \begin{split}
& u(t_{i+1}) \!=\! u(t_i)\!-\! \bigg[I_{\Rn}\!-\! \dfrac{\partial W}{\partial u}\big(t_{i+1},z(t_{i+1}),u(t_i)\big)\bigg]^{-1} \Big[u(t_i)\!-\! W\big(t_{i+1},z(t_{i+1}),u(t_i)\big)\Big]\!+\! O(h)\!    \\
 &\;\; =u(t_i) \!-\! \bigg[I_{\Rn}\!-\! G^{-1}(t_{i+1})Q_2(t_{i+1})\dfrac{\partial f}{\partial x}\big(t_{i+1},P_1(t_{i+1})z(t_{i+1})\!+\! P_2(t_{i+1})u(t_i)\big) P_2(t_{i+1})\bigg]^{-1}    \\
 &\qquad\qquad\qquad\qquad
 \times \bigg[u(t_i)\!-\! G^{-1}(t_{i+1})Q_2(t_{i+1}) \Big[f\big(t_{i+1},P_1(t_{i+1})z(t_{i+1})\!+\! P_2(t_{i+1})u(t_i)\big)   \\
 &\qquad\qquad\qquad\qquad\qquad\qquad\qquad\qquad\qquad\qquad\qquad\quad
 -A'(t_{i+1})P_1(t_{i+1})z(t_{i+1})\Big]\bigg]\!+\! O(h),
 \end{split} \label{exact3}\\
& x(t_{i+1})= z(t_{i+1})+u(t_{i+1})=  P_1(t_{i+1})z(t_{i+1})+P_2(t_{i+1})u(t_{i+1}),\quad i= 0,...,N-1. \label{exact4}
\end{align}
Then the numerical method for solving the IVP \eqref{DAE}, \eqref{ini} on $[t_0,T]$ takes the form \eqref{met1}--\eqref{met4}, where $z_i$, $u_i$ ($i=0,...,N$) are values of the approximate solution of the system \eqref{num1}, \eqref{num2} at the point $t_i$, which satisfies the initial conditions $z(t_0)=P_1(t_0)x_0$ and $u(t_0)=P_2(t_0)x_0$, and $x_i$ ($i=0,...,N$) is a value of the approximate solution of the IVP \eqref{DAE}, \eqref{ini} at $t_i$.

Denote
\begin{equation}\label{const}
 \begin{split}
 &p_i=\sup\limits_{t\in [t_0,T]}\|P_i(t)\|,\quad q_i=\sup\limits_{t\in [t_0,T]}\|Q_i(t)\|,\quad i=1,2,\quad p=\max\{p_1,p_2\},            \\
 &\tilde{p}_1=\sup\limits_{t\in [t_0,T]}\|P'_1(t)\|,\quad  \tilde{a}=\sup\limits_{t\in [t_0,T]}\|A'(t)\|,\quad b=\sup\limits_{t\in [t_0,T]}\|B(t)\|,\quad \hat{g}=\sup\limits_{t\in [t_0,T]}\|G^{-1}(t)\|.
 \end{split}
\end{equation}
Since the partial derivative of $f(t,x)$ with respect to $x$ is continuous on $[t_+,\infty)\times \Rn$, then, using the finite increment formula, we obtain (for $i=1,...,N$):
\begin{equation}\label{FKP1}
\big\|f\big(t_i,P_1(t_i)z(t_i)+P_2(t_i)u(t_i)\big)- f(t_i,P_1(t_i)z_i+P_2(t_i)u_i)\big\|\le M p \big(\|z(t_i)-z_i\|+\|u(t_i)-u_i\|\big),
\end{equation}
where\;  $M=\!\max\limits_{1\le i\le N}\sup\limits_{\theta_i\in (0,1)} \bigg\|\dfrac{\partial f}{\partial x} \Big(t_i,P_1(t_{i})z_{i}+P_2(t_{i})u_{i}+ \theta_i \big(P_1(t_{i})[z(t_{i})-z_{i}]+ P_2(t_{i})[u(t_{i})-u_{i}]\big)\Big)\bigg\|$.

Denote
$$
\varepsilon^z_i=\|z(t_i)-z_i\|,\quad \varepsilon^u_i=\|u(t_i)-u_i\|.
$$
It follows from the initial condition that $\varepsilon^z_0=0$, $\varepsilon^u_0=0$, and from the formulas \eqref{met2}, \eqref{exact2} and \eqref{FKP1} we have:
$\varepsilon^z_1= O(h^2)$,
 \begin{equation}\label{num3}
\varepsilon^z_{i+1}\le \big(1+h\, [\tilde{p}_1+\hat{g}q_1(p_1(\tilde{a}+b)+Mp)]\,\big) \varepsilon^z_i+ h\hat{g}q_1Mp\, \varepsilon^u_i+ O(h^2).
 \end{equation}
Denote $\varkappa= \tilde{p}_1+\hat{g}q_1(p_1(\tilde{a}+b)+Mp)$,\,  $r(h)=1+\varkappa\, h$ and $\tilde{M}=\hat{g}q_1Mp$,  then \eqref{num3} takes the form
 \begin{equation}\label{num6}
\varepsilon^z_{i+1}\le r(h)\varepsilon^z_i+h \tilde{M}\varepsilon^u_i +O(h^2).
 \end{equation}
Using \eqref{num6} recursively, we obtain that\;
\begin{equation}
\varepsilon^z_{i+1}\le h\,\tilde{M}\sum\limits_{j=0}^i r^{i-j}(h) \varepsilon^u_j+ O(h^2)\sum\limits_{j=0}^i r^j(h),\quad i=0,...,N-1.
\end{equation}
Since $r^j(h)\le \left(1+\frac{(T-t_0)\varkappa}{N}\right)^N\le e^{(T-t_0)\varkappa}$, $j=1,...,N$, then
\begin{equation}\label{num10}
 \varepsilon^z_{i+1}\le O(h)\sum\limits_{j=1}^i \varepsilon^u_j+ O(h),\quad i=1,...,N-1.
\end{equation}

Further, using the formula
\begin{multline*}
u(t_{i+1})=G^{-1}(t_{i+1})Q_2(t_{i+1}) \big[f\big(t_{i+1},P_1(t_{i+1})z(t_{i+1})+P_2(t_{i+1})u(t_i)\big) -A'(t_{i+1})P_1(t_{i+1})z(t_{i+1})\big]   \\
+ G^{-1}({t_{i+1}})Q_2(t_{i+1}) \dfrac{\partial f}{\partial x}\big(t_{i+1},P_1(t_{i+1})z(t_{i+1})+P_2(t_{i+1})u(t_i)\big)\, P_2(t_{i+1}) \big[u(t_{i+1})-u(t_i)\big]+O(h)
\end{multline*}
and the corresponding formula for finding the approximate value $u_{i+1}$, that is,
\begin{multline*}
u_{i+1}=G^{-1}(t_{i+1})Q_2(t_{i+1}) \big[f\big(t_{i+1},P_1(t_{i+1})z_{i+1}+ P_2(t_{i+1})u_i\big)- A'(t_{i+1})P_1(t_{i+1})z_{i+1}\big]   \\
+G^{-1}({t_{i+1}})Q_2(t_{i+1}) \dfrac{\partial f}{\partial x}\big(t_{i+1},P_1(t_{i+1})z_{i+1}+P_2(t_{i+1})u_i\big) P_2(t_{i+1})\big[u_{i+1}-u_i\big],
\end{multline*}
we obtain the relation
\begin{multline*}
u(t_{i+1})-u_{i+1} = \bigg[I_{\Rn}-G^{-1}(t_{i+1})Q_2(t_{i+1})\dfrac{\partial f}{\partial x}\big(t_{i+1},P_1(t_{i+1})z_{i+1}+P_2(t_{i+1})u_i\big)\, P_2(t_{i+1})\bigg]^{-1}   \\
 \times\! \bigg[G^{-1}(t_{i+1})Q_2(t_{i+1}) \bigg( f\big(t_{i+1},P_1(t_{i+1})z(t_{i+1})+P_2(t_{i+1})u(t_i)\big)- f\big(t_{i+1},P_1(t_{i+1})z_{i+1}+P_2(t_{i+1})u_i\big)   \\ -A'(t_{i+1})P_1(t_{i+1})\big[z(t_{i+1})-z_{i+1}\big]+ \dfrac{\partial f}{\partial x}\big(t_{i+1},P_1(t_{i+1})z(t_{i+1})+P_2(t_{i+1})u(t_i)\big)\, P_2(t_{i+1}) \big[u(t_{i+1})   \\
 -u(t_i)\big]- \dfrac{\partial f}{\partial x}\big(t_{i+1},P_1(t_{i+1})z_{i+1}+P_2(t_{i+1})u_i\big)\, P_2(t_{i+1})\big[u(t_{i+1})-u(t_i)+u(t_i)-u_i\big]\bigg)+O(h)\bigg].
\end{multline*}

By the finite increment formula, we have that (for $i=0,...,N-1$):
 \begin{multline*}
\big\|f\big(t_{i+1},P_1(t_{i+1})z(t_{i+1})+P_2(t_{i+1})u(t_i)\big)- f\big(t_{i+1},P_1(t_{i+1})z_{i+1}+P_2(t_{i+1})u_i\big)\big\|    \\
\le \hat{M} p \big(\|z(t_{i+1})-z_{i+1}\|+\|u(t_i)-u_i\|\big),
 \end{multline*}
where $p$ is defined in \eqref{const} and
\begin{multline*}
\hat{M}=\max\limits_{0\le i\le N-1}\sup\limits_{\hat{\theta}_i\in (0,1)} \bigg\|\dfrac{\partial f}{\partial x}\Big(t_{i+1},P_1(t_{i+1})z_{i+1}+P_2(t_{i+1})u_i+\hat{\theta}_i \big(P_1(t_{i+1})[z(t_{i+1})-z_{i+1}]   \\ +P_2(t_{i+1})[u(t_{i})-u_{i}]\big)\Big)\bigg\|.
\end{multline*}
 Denote
$$
C_1=\sup\limits_{0\le i\le N-1} \left\|\dfrac{\partial f}{\partial x}\big(t_{i+1},P_1(t_{i+1})z(t_{i+1})+P_2(t_{i+1})u(t_i)\big)\right\|,
$$
 \begin{equation}\label{C2}
C_2=\sup\limits_{0\le i\le N-1} \left\|\dfrac{\partial f}{\partial x}\big(t_{i+1},P_1(t_{i+1})z_{i+1}+P_2(t_{i+1})u_i\big)\right\|,
 \end{equation}
 \begin{equation}\label{K}
K\!=\! \sup\limits_{0\le i\le N-1}\bigg\|\Big[I_{\Rn}\!-\! G^{-1}(t_{i+1})Q_2(t_{i+1}) \dfrac{\partial f}{\partial x}\big(t_{i+1},P_1(t_{i+1})z_{i+1}\!+\! P_2(t_{i+1})u_i\big)\, P_2(t_{i+1})\Big]^{-1}\bigg\|.
 \end{equation}
Then
\begin{multline*}
\varepsilon^u_{i+1}\le K\hat{g}q_2 \big[\hat{M}p (\varepsilon^z_{i+1}+\varepsilon^u_i)+ \tilde{a}p_1\varepsilon^z_{i+1}+ C_1p_2 O(h)+ C_2p_2 (O(h)+\varepsilon^u_i)\big] +O(h)   \\
=K\hat{g}q_2(\hat{M}p+\tilde{a}p_1)\varepsilon^z_{i+1}+ K\hat{g}q_2(\hat{M}p+C_2p_2)\varepsilon^u_i+O(h), \quad i=0,...,N-1.
\end{multline*}
Consequently, there exist the constants $\alpha=K\hat{g}q_2(\hat{M}p+\tilde{a}p_1)$ and $\beta=K\hat{g}q_2(\hat{M}p+C_2p_2)$ such that
\begin{equation}\label{num5}
\varepsilon^u_{i+1}\le \alpha\varepsilon^z_{i+1}+\beta\varepsilon^u_i +O(h),\quad i=0,...,N-1.
\end{equation}
From \eqref{num5}, \eqref{num10} and the relation $\varepsilon^z_1=O(h^2)$ we obtain
$${\varepsilon^u_{i+1}\le O(h) \sum\limits_{j=1}^i \varepsilon^u_j+ \beta\varepsilon^u_i +O(h)},\; {i=0,...,N-1}.
$$
Further, using the method of mathematical induction, we find that $\varepsilon^u_{i+1}=O(h)$, ${i=0,...,N-1}$, and given \eqref{num10} we
obtain $\varepsilon^z_{i+1}=O(h)$, $i=1,...,N-1$.
Consequently,
$$
\max\limits_{1\le i\le N}\varepsilon^u_i=\max\limits_{1\le i\le N}\|u(t_i)-u_i\|=O(h),\quad \max\limits_{1\le i\le N}\varepsilon^z_i= \max\limits_{1\le i\le N}\|z(t_i)-z_i\|=O(h),\quad  h\to 0,
$$
and hence $\max\limits_{1\le i\le N}\|x(t_i)-x_i\|=O(h)$, $h\to 0$ (recall that $\varepsilon^z_0=0$, $\varepsilon^u_0=0$ and $\|x(t_0)-x_0\|=0$).
Thus, the method \eqref{met1}--\eqref{met4} converges and has the first order.
 \end{proof}

 \begin{proposition}\label{remNum-meth}
If in Theorem~\ref{ThNum-meth} we do not require the additional smoothness for  $f$, $A, B$ and $C_2$, i.e., we assume that $f\in C([t_+,\infty)\times\Rn,\Rn)$, $\dfrac{\partial f}{\partial x}\in C([t_+,\infty)\times\Rn, \mathrm{L}(\Rn))$, $A, B\in C^1([t_+,\infty),\mathrm{L}(\Rn))$ and $C_2\in C^1([t_+,\infty),(0,\infty))$ (these restrictions are specified in Theorems \ref{Th_GlobSolv} and~\ref{Th_GlobSolvBInv}), then the method \eqref{met1}--\eqref{met4} is convergent, but may not have the first order, that is,
$$
{\max\limits_{1\le i\le N}\|x(t_i)-x_i\|=o(1)},\; {h\to 0}\; (\text{i.e.,}\; \max\limits_{1\le i\le N}\|x(t_i)-x_i\|\to 0,\; {h\to 0}\,),\; \|x(t_0)-x_0\|=0.
$$
 \end{proposition}
  \begin{proof}
The proof is carried out in the same way as the proof of Theorem~\ref{ThNum-meth}, where instead of  \eqref{Taylor2}, \eqref{Taylor1} we use the representations
 \begin{equation}\label{Taylor2-1_o-small}
\begin{split}
W\big(t\!+\! h,z(t\!+\! h),u(t\!+\! h)\big)= W\big(t\!+\! h,z(t\!+\! h),u(t)\big)\!+\!  \dfrac{\partial W}{\partial u}\big(t\!+\! h,z(t\!+\! h),u(t)\big) \big[u(t\!+\! h)  &\\
 -u(t)\big]+o(1),\quad \frac{dz}{dt}(t)=\frac{z(t\!+\! h)-z(t)}{h}+o(1),\quad h\to0. &
\end{split}
 \end{equation}
\end{proof}

Proposition \ref{remNum-meth} states that, in general, the conditions of Theorem \ref{Th_GlobSolv} or~\ref{Th_GlobSolvBInv} on the global solvability of the DAE are sufficient for the convergence of the methods, and only the invertibility of the operator $\Phi_{t_*,P_1(t_*)z_*,P_2(t_*)u_*}$ for every fixed ${(t_*,P_1(t_*)z_*+P_2(t_*)u_*)\in [t_0,T]\times\Rn}$ which does not belong to $L_{t_+}$ is additionally needed.

 \subsection{Method 2 (the combined method with recalculation)}\label{Num-Mod_meth}

\begin{theorem}\label{ThNum-meth2}
Let the conditions of Theorem~\ref{Th_GlobSolv} or~\ref{Th_GlobSolvBInv} be satisfied, and in addition let the operator
$$
{\Phi_{t_*,P_1(t_*)z_*,P_2(t_*)u_*}=\Phi_{t_*,P_1(t_*)z_*}(P_2(t_*)u_*)\colon X_2(t_*)\to Y_2(t_*)},
$$
which is defined by \eqref{funcPhi} or \eqref{funcPhiBInv} for each (fixed) $t_*$,\, ${x^*_{p_1}(t_*)=P_1(t_*)z_*}$,\, ${x^*_{p_2}(t_*)=P_2(t_*)u_*}$,  be invertible for each point ${(t_*,P_1(t_*)z_*+P_2(t_*)u_*)\in [t_0,T]\times\Rn}$.
In addition, let ${A,\, B\in C^3([t_0,T],\mathrm{L}(\Rn))}$, ${C_2\in C^3([t_0,T],(0,\infty))}$, ${f\in C^2([t_0,T]\times \Rn,\Rn)}$, and let an initial value $x_0$ be chosen so that the consistency condition ${(t_0,x_0)\in L_{t_+}}$ is satisfied. Then the method
 \begin{align}
z_0 &= P_1(t_0)x_0,\quad u_0= P_2(t_0)x_0,  \label{NImpMet1}\\
\widetilde{z}_{i+1} & =\!  \Big[I_{\Rn}+ h\Big(P'_1(t_i) - G^{-1}(t_i)Q_1(t_i)[A'(t_i)\!+\! B(t_i)]\Big) P_1(t_i)\Big] z_i\!+\! h\, G^{-1}(t_i)Q_1(t_i) f\big(t_i,x_i\big),  \label{NImpMet2}\\
\begin{split}
\widetilde{u}_{i+1} &=\! u_i \!-\! \bigg[I_{\Rn}\!-\!G^{-1}(t_{i+1})Q_2(t_{i+1})\dfrac{\partial f}{\partial x}\big(t_{i+1},P_1(t_{i+1})\widetilde{z}_{i+1}+ P_2(t_{i+1})u_i\big)\, P_2(t_{i+1})\bigg]^{-1}\! \bigg[u_i    \\
 & -G^{-1}(t_{i+1})Q_2(t_{i+1}) \Big[f\big(t_{i+1},P_1(t_{i+1})\widetilde{z}_{i+1}\!+\! P_2(t_{i+1})u_i\big) -A'(t_{i+1})P_1(t_{i+1})\widetilde{z}_{i+1}\Big]\bigg]\!,
\end{split}     \label{NImpMet3}\\
\begin{split}
z_{i+1} &= \! \bigg[I_{\Rn} \!+\! \dfrac{h}{2}\big(P'_1(t_i)- G^{-1}(t_i)Q_1(t_i)[A'(t_i)\!+\! B(t_i)]\big) P_1(t_i)\bigg] z_i \\
 & +\dfrac{h}{2}\big(P'_1(t_{i+1})- G^{-1}(t_{i+1})Q_1(t_{i+1})[A'(t_{i+1})\!+\! B(t_{i+1})] \big)P_1(t_{i+1})\widetilde{z}_{i+1}    \\
 & +\dfrac{h}{2}\bigg[G^{-1}(t_i)Q_1(t_i) f\big(t_i,x_i\big)\!+\!  G^{-1}(t_{i+1})Q_1(t_{i+1}) f\big(t_{i+1},P_1(t_{i+1})\widetilde{z}_{i+1}\!+\! P_2(t_{i+1})\widetilde{u}_{i+1}\big)\bigg]\!,
\end{split}  \label{NImpMet4}\\
\begin{split}
u_{i+1} &= \! u_i \!-\! \bigg[I_{\Rn}-G^{-1}(t_{i+1})Q_2(t_{i+1})\dfrac{\partial f}{\partial x}\big(t_{i+1},P_1(t_{i+1})z_{i+1}\!+\! P_2(t_{i+1})u_i\big)\, P_2(t_{i+1})\bigg]^{-1}\! \bigg[u_i    \\
 & -G^{-1}(t_{i+1})Q_2(t_{i+1}) \Big[f\big(t_{i+1},P_1(t_{i+1})z_{i+1}\!+\! P_2(t_{i+1})u_i\big) -A'(t_{i+1})P_1(t_{i+1})z_{i+1}\Big]\bigg]\!,
\end{split}  \label{NImpMet5}\\
x_{i+1} &= P_1(t_{i+1})z_{i+1}+P_2(t_{i+1})u_{i+1},\quad t_{i+1}\in \omega_h,\quad  i=0,...,N-1, \label{NImpMet6}
 \end{align}
which approximates the IVP \eqref{DAE}, \eqref{ini} on $[t_0,T]$, is convergent of order 2, that is,
$$
\max\limits_{1\le i\le N}\|x(t_i)-x_i\|=O(h^2),\quad h\to 0,
$$
where $x(t)$ is an exact solution of the IVP \eqref{DAE},\eqref{ini} ($x_i$ is the value of an approximate solution at~$t_i$).
\end{theorem}

 \begin{proof}
Take any initial point $(t_0,x_0)\in L_{t_+}$. By virtue of the theorem conditions, for each initial point $(t_0,x_0)\in L_{t_+}$ there exists a unique global (exact) solution $x(t)$ of the IVP \eqref{DAE}, \eqref{ini} such that $z(t)=P_1(t)x(t)\in C^3([t_0,T],\Rn)$ and $u(t)=P_2(t)x(t)\in C^2([t_0,T],\Rn)$  ($z\in C^1([t_0,\infty),\Rn)$, $u\in C([t_0,\infty),\Rn)$ and $z(t)\in X_1(t)$, $u(t)\in X_2(t)$).

As in the proof of the previous theorem, we consider the system \eqref{num1}, \eqref{num2}, where the mappings $\Pi(t,z,u)$ and $W(t,z,u)$ have the form \eqref{Pi} and  \eqref{W}. Lemma \ref{LemSolEquiv} remains valid.

Using equality \eqref{num2} where $t$ is replaced by $t+h$ and the Taylor expansion of the form
$$
W\big(t+h,z(t+h),u(t+h)\big)= W\big(t+h,z(t+h),u(t)\big)+\dfrac{\partial W}{\partial u}\big(t+h,z(t+h),u(t)\big)\big[u(t+h)-u(t)\big]+O(h^2),
$$
where $\dfrac{\partial W}{\partial u}\big(t+h,z(t+h),u(t)\big)$ has the form \eqref{dW_du},
we obtain the relation of the form \eqref{approx_u-1} where $O(h)$ is replaced by $O(h^2)$.
This relation can be written as
 \begin{equation}\label{Napprox_u-2}
\begin{split}
u(t+h)\!=\! u(t)\!-\! \bigg[I_{\Rn}\!-\! \dfrac{\partial W}{\partial u}\big(t\!+\! h,z(t\!+\! h),u(t)\big)\bigg]^{-1}\! \Big[u(t)\!-\! W\big(t\!+\! h,z(t\!+\! h),u(t)\big)\Big]\!+\! O(h^2)    &\\
 = u(t)\!-\!\bigg[I_{\Rn}\!-\! G^{-1}(t\!+\! h)Q_2(t\!+\! h)\dfrac{\partial f}{\partial x}\big(t\!+\! h,P_1(t\!+\! h)z(t\!+\! h)\!+\! P_2(t\!+\! h)u(t)\big)\, P_2(t\!+\! h)\bigg]^{-1}   &\\
  \times \bigg[u(t)\!-\! G^{-1}(t\!+\! h)Q_2(t\!+\! h) \Big(f\big(t\!+\! h,P_1(t\!+\! h)z(t\!+\! h)+P_2(t\!+\! h)u(t)\big)    &\\
 -A'(t\!+\! h)P_1(t\!+\! h)z(t\!+\! h)\Big)\bigg]\!+ O(h^2).   &
\end{split}
 \end{equation}
There exist the inverse operators \eqref{InvNum} and \eqref{BasInvNum} (when the requirements of Theorems~\ref{Th_GlobSolv} and~\ref{Th_GlobSolvBInv}, respectively, are fulfilled) for the points $(t+h,P_1(t+h)z(t+h)+P_2(t+h)u(t))\!\in\! L_{t_0}$ and $(t+h,P_1(t+h)z(t+h)+P_2(t+h)u(t))\!\in\! [t_0,T]\times \Rn$ (see the explanation in the proof of Theorem~\ref{ThNum-meth}).

As above, we denote by $z_i$, $u_i$ and $x_i$ ($i=0,...,N$) the values, at the points $t_i$, of an approximate solution of the system \eqref{num1}, \eqref{num2} that satisfies the initial conditions  $z(t_0)=P_1(t_0)x_0$ and $u(t_0)=P_2(t_0)x_0$
and of an approximate solution of the IVP \eqref{DAE}, \eqref{ini}, respectively.

To approximate the DE \eqref{num1}, we will use the Euler scheme with recalculation (such schemes are also called implicit and ``predictor-corrector'' schemes).

The preliminary value of $z(t)$ at the point $t_{i+1}$ is calculated using the explicit Euler method (as in method 1), i.e., the DE \eqref{num1} is approximated by the scheme $z(t+h)=z(t)+h\, \Pi(t,z(t),u(t))+O(h^2)$,  and the approximate value for $z(t_{i+1})$, which will be denoted by $\widetilde{z}_{i+1}$, is calculated by the formula \eqref{NImpMet2}, where $x_i=P_1(t_i)z_i+P_2(t_i)u_i$, or
$\widetilde{z}_{i+1}=z_i+h\, \Pi(t_i,z_i,u_i)$. Denote
 \begin{equation}\label{NExact2}
\begin{split}
\widetilde{z}(t_{i+1}) = & \, z(t_i)+h\, \Pi(t_i,z(t_i),u(t_i))=\Big(I_{\Rn} +h\big[P'_1(t_i)-G^{-1}(t_i)Q_1(t_i)[A'(t_i)   \\
& +B(t_i)]\big]P_1(t_i)\Big) z(t_i)+ h\, G^{-1}(t_i)Q_1(t_i) f\big(t_i,P_1(t_i)z(t_i)+P_2(t_i)u(t_i)\big).
\end{split}
 \end{equation}
Using the formula \eqref{Napprox_u-2} and substituting ${z(t_i+h)}=z(t_{i+1}):=\widetilde{z}(t_{i+1})$, we find the preliminary value of $u(t)$ at the point $t_{i+1}$, which will be denoted by $\widetilde{u}(t_{i+1})$. As a result, we have
 \begin{equation}
\begin{split}
 \widetilde{u}(t_{i+1})\!=\!u(t_i)\!-\!\! \bigg[I_{\Rn}\!-\! \dfrac{\partial W}{\partial u}\big(t_{i+1},\widetilde{z}(t_{i+1}),u(t_i)\big)\bigg]^{-1}\! \Big[u(t_i)\!-\! W\big(t_{i+1},\widetilde{z}(t_{i+1}),u(t_i)\big)\Big] \!\!+\!O(h^2)\!     &\\
=u(t_i)\!-\! \bigg[I_{\Rn}\!-\! G^{-1}(t_{i+1})Q_2(t_{i+1})\dfrac{\partial f}{\partial x}\big(t_{i+1},P_1(t_{i+1})\widetilde{z}(t_{i+1})\!+\! P_2(t_{i+1})u(t_i)\big)\, P_2(t_{i+1})\bigg]^{-1}   &\\
\times \bigg[u(t_i)-G^{-1}(t_{i+1})Q_2(t_{i+1}) \Big[f\big(t_{i+1},P_1(t_{i+1})\widetilde{z}(t_{i+1})\!+\! P_2(t_{i+1})u(t_i)\big)      &\\
 -A'(t_{i+1})P_1(t_{i+1})\widetilde{z}(t_{i+1})\Big]\bigg]\!+\! O(h^2).  &
\end{split} \label{NExact3}
 \end{equation}
The corresponding approximate value which is denoted by $\widetilde{u}_{i+1}$ takes the form \eqref{NImpMet3} or
$$
\widetilde{u}_{i+1}=u_i-\bigg[I_{\Rn}-\dfrac{\partial W}{\partial u}\big(t_{i+1},\widetilde{z}_{i+1},u_i\big)\bigg]^{-1} \Big[u_i-W\big(t_{i+1},\widetilde{z}_{i+1},u_i\big)\Big].
$$

Now let us perform the recalculation using the formula \eqref{NImpMet4}, i.e., the approximate value found for $z(t_{i+1})$ by the formula \eqref{NImpMet2} is refined using the expression
 \begin{equation*}
z_{i+1}=z_i+\dfrac{h}{2}\bigg[\Pi(t_i,z_i,u_i)+ \Pi(t_{i+1},\widetilde{z}_{i+1},\widetilde{u}_{i+1})\bigg],
 \end{equation*}
where $\widetilde{z}_{i+1}$ and $\widetilde{u}_{i+1}$ have the form \eqref{NImpMet2} and \eqref{NImpMet3}.

Substitute the values $z(t_i)$, $u(t_i)$ of the exact solution into \eqref{NImpMet4} and write the expression for finding the residual (approximation error):
\begin{equation}\label{Residual}
\psi_i(h)= -\dfrac{z(t_{i+1})-z(t_i)}{h}+ \dfrac{1}{2} \bigg[\Pi\big(t_i,z(t_i),u(t_i)\big)+ \Pi\big(t_{i+1},z(t_i)\!+\! h\Pi(t_i,z(t_i),u(t_i)),\widetilde{u}(t_{i+1})\big)\bigg]\!,\!
\end{equation}
where $\widetilde{u}(t_{i+1})$ is defined by \eqref{NExact3}.

Using the Taylor formula, we obtain the following expansions (as $h\to 0$):
 \begin{equation}\label{appr1_psi}
\dfrac{z(t_{i+1})-z(t_i)}{h}=z'(t_i)+\dfrac{h}{2}z''(t_i)+O(h^2),
\end{equation}
\begin{equation}\label{appr2_psi}
 \begin{split}
\Pi\big(t_{i+1},z(t_i)+h\, \Pi(t_i,z(t_i),u(t_i)),u(t_{i+1})\big)= \Pi\big(t_i,z(t_i),u(t_i)\big)+ h\bigg[\dfrac{\partial\Pi}{\partial t}\big(t_i,z(t_i),u(t_i)\big)   &\\
+ \dfrac{\partial\Pi}{\partial x}\big(t_i,z(t_i),u(t_i)\big) \bigg(P_1(t_i)\Pi\big(t_i,z(t_i),u(t_i)\big)+ P_2(t_i)\dfrac{du}{dt}(t_i)\bigg)\bigg]+ O(h^2). &
\end{split}
\end{equation}
It follows from \eqref{num1}, \eqref{Residual}, \eqref{appr1_psi}, \eqref{appr2_psi} and the equality
$$
z''(t_i)=  \dfrac{\partial\Pi}{\partial t}\big(t_i,z(t_i),u(t_i)\big)+ \dfrac{\partial\Pi}{\partial x}\big(t_i,z(t_i),u(t_i)\big) \bigg[P_1(t_i) \Pi\big(t_i,z(t_i),u(t_i)\big)+P_2(t_i)\dfrac{du}{dt}(t_i)\bigg]
$$
that
$$
\psi_i(h)= O(h^2).
$$
Thus, the value of $z(t)$ at the point $t_{i+1}$ is finally calculated by the following formula (where $\widetilde{u}(t_{i+1})$ has the form \eqref{NExact3}):
 \begin{equation}\label{NExact4}
z(t_{i+1})=z(t_i)+\dfrac{h}{2}\bigg[\Pi(t_i,z(t_i),u(t_i))+ \Pi\big(t_{i+1},z(t_i)+h\Pi(t_i,z(t_i),u(t_i)),\widetilde{u}(t_{i+1})\big)\bigg]\!+ O(h^3).
 \end{equation}

Further, we carry out the recalculation of the value of $u(t_{i+1})$, using the same formula as before, but with the value of $z(t_{i+1})$ refined by the formula \eqref{NExact4}:
 \begin{equation*}
\begin{split}
u(t_{i+1})\!=\!u(t_i)\!-\!\! \bigg[I_{\Rn}\!-\! G^{-1}(t_{i+1})Q_2(t_{i+1})\dfrac{\partial f}{\partial x}\big(t_{i+1},P_1(t_{i+1})z(t_{i+1})\!+\! P_2(t_{i+1})u(t_i)\big) P_2(t_{i+1})\bigg]^{-1}\! \bigg[u(t_i)  &\\
{}-G^{-1}(t_{i+1})Q_2(t_{i+1}) \Big[f\big(t_{i+1},P_1(t_{i+1})z(t_{i+1})+P_2(t_{i+1})u(t_i)\big) -A'(t_{i+1})P_1(t_{i+1})z(t_{i+1})\Big]\bigg]\!+O(h^2).  &
\end{split}
 \end{equation*}

A solution of the IVP \eqref{DAE}, \eqref{ini} at the points of the introduced mesh \eqref{mesh} is calculated by the formula  \eqref{exact4}.

Denote
$$
\varepsilon^z_i=\|z(t_i)-z_i\|,\; \varepsilon^u_i=\|u(t_i)-u_i\|,\; i=0,...,N;\;\; \widetilde{\varepsilon}^z_i=\|\widetilde{z}(t_i)-\widetilde{z}_i\|,\; \widetilde{\varepsilon}^u_i=\|\widetilde{u}(t_i)-\widetilde{u}_i\|,\; i=1,...,N.
$$
It follows from the initial condition that $\varepsilon^z_0=\varepsilon^u_0=0$.
From the above, we obtain the inequality ${\varepsilon^z_{i+1}\le \varepsilon^z_i+ O(h^3)+h\|\varphi_i(h)\|}$, where
 \begin{equation*}
\begin{split}
\varphi_i(h) =& -0.5\Big[\Pi\big(t_i,z(t_i),u(t_i)\big)-\Pi(t_i,z_i,u_i)  \\
 &{} +\Pi\big(t_{i+1},z(t_i)+ h\Pi\big(t_i,z(t_i),u(t_i)\big),\widetilde{u}(t_{i+1})\big)- \Pi(t_{i+1},z_i+h\Pi(t_i,z_i,u_i),\widetilde{u}_{i+1})\Big].
\end{split}
 \end{equation*}
Introduce the estimates \eqref{const}, \eqref{FKP1}.
Then
 \begin{equation}\label{Pi1}
\|\Pi\big(t_i,z(t_i),u(t_i)\big)-\Pi(t_i,z_i,u_i)\|\le k p_1 \varepsilon^z_i + \tilde{M}(\varepsilon^z_i +\varepsilon^u_i)= O(\varepsilon^z_i)+O(\varepsilon^u_i),
 \end{equation}
where ${\tilde{M}=\hat{g}q_1Mp}$ and ${k=\tilde{p}_1+\hat{g}q_1[\tilde{a}+b]}$.
Similarly, using the finite increment formula and the estimates \eqref{const} and \eqref{Pi1}, we obtain
 \begin{multline*}
\|\Pi\big(t_{i+1},z(t_i)+h\Pi\big(t_i,z(t_i),u(t_i)\big),\widetilde{u}(t_{i+1})\big)- \Pi\big(t_{i+1},z_i+h\Pi(t_i,z_i,u_i),\widetilde{u}_{i+1}\big)\|   \\
 \le \big[k p_1+\widehat{M}+O(h)\big] \varepsilon^z_i+ O(h)\varepsilon^u_i+ \widehat{M}\,\widetilde{\varepsilon}^u_{i+1},
 \end{multline*}
\begin{flalign*}
&\text{where}\quad & &\widehat{M}=\hat{g}q_1 \tilde{\tilde{M}}p, \qquad
\tilde{\tilde{M}}=\!\max\limits_{0\le i\le N-1}\sup\limits_{\theta_{i+1}\in (0,1)} \bigg\|\dfrac{\partial f}{\partial x} \Big(t_{i+1},\widetilde{x}_{i+1}+ \theta_{i+1} \big[\widetilde{x}(t_{i+1})-\widetilde{x}_{i+1}\big]\Big)\bigg\|,  \\
& & &\widetilde{x}(t_{i+1})= P_1(t_{i+1})\big[z(t_i)+h\Pi\big(t_i,z(t_i),u(t_i)\big)\big]+ P_2(t_{i+1})\widetilde{u}(t_{i+1}),  \\
& & &\widetilde{x}_{i+1}= P_1(t_{i+1})\big[z_i+h\Pi(t_i,z_i,u_i)\big]+ P_2(t_{i+1})\widetilde{u}_{i+1}.
\end{flalign*}
Thus,
 $$
\|\varphi_i(h)\|\le (O(1)+O(h))[\varepsilon^z_i+\varepsilon^u_i]+O(1)\widetilde{\varepsilon}^u_{i+1},
 $$
and hence
 \begin{equation}\label{Nnum6}
\begin{split}
\varepsilon^z_{i+1}\le [1+O(h)+O(h^2)]\varepsilon^z_i+ [O(h)+O(h^2)]\varepsilon^u_i+ & O(h)\widetilde{\varepsilon}^u_{i+1}+O(h^3)   \\
 & =\hat{r}(h)\varepsilon^z_i+O(h) [\varepsilon^u_i+\widetilde{\varepsilon}^u_{i+1}]+ O(h^3),
\end{split}
 \end{equation}
 \begin{equation}\label{TildeEps_z_1}
\widetilde{\varepsilon}^z_{i+1}\le \varepsilon^z_i +h\|\Pi\big(t_i,z(t_i),u(t_i)\big)-\Pi(t_i,z_i,u_i)\|= \hat{r}(h)\varepsilon^z_i+O(h)\varepsilon^u_i,\;\; i=0,...,N-1,
 \end{equation}
where $\hat{r}(h)=1+O(h)$.

Further, the expression
\begin{multline*}
u(t_{i+1})=G^{-1}(t_{i+1})Q_2(t_{i+1}) \big[f\big(t_{i+1},P_1(t_{i+1})z(t_{i+1})+P_2(t_{i+1})u(t_i)\big) -A'(t_{i+1})P_1(t_{i+1})z(t_{i+1})\big]  \\
 +G^{-1}({t_{i+1}})Q_2(t_{i+1}) \dfrac{\partial f}{\partial x}\big(t_{i+1},P_1(t_{i+1})z(t_{i+1})+ P_2(t_{i+1})u(t_i)\big)\, P_2(t_{i+1}) \big[u(t_{i+1})-u(t_i)\big]+O(h^2)
\end{multline*}
and the corresponding expression for finding the approximate value $u_{i+1}$, that is,
\begin{multline*}
u_{i+1}=G^{-1}(t_{i+1})Q_2(t_{i+1}) \big[f\big(t_{i+1},P_1(t_{i+1})z_{i+1}+P_2(t_{i+1})u_i\big) -A'(t_{i+1})P_1(t_{i+1})z_{i+1}\big]  \\
 +G^{-1}({t_{i+1}})Q_2(t_{i+1}) \dfrac{\partial f}{\partial x}\big(t_{i+1},P_1(t_{i+1})z_{i+1}+ P_2(t_{i+1})u_i\big)\, P_2(t_{i+1})\big[u_{i+1}-u_i\big],
\end{multline*}
yield
 \begin{multline*}
u(t_{i+1})-u_{i+1}=  \bigg[I_{\Rn}-G^{-1}(t_{i+1})Q_2(t_{i+1})\dfrac{\partial f}{\partial x} \big(t_{i+1},P_1(t_{i+1})z_{i+1}+P_2(t_{i+1})u_i\big) P_2(t_{i+1})\bigg]^{-1}   \\
 \times\bigg(G^{-1}(t_{i+1})Q_2(t_{i+1}) \bigg[\dfrac{\partial f}{\partial x} \big(t_{i+1},P_1(t_{i+1})z_{i+1}+P_2(t_{i+1})u_i\big)-A'(t_{i+1})\bigg]\, P_1(t_{i+1})[z(t_{i+1})-z_{i+1}]   \\ +O\big([\|z(t_{i+1})-z_{i+1}\|+\|u(t_i)-u_i\|]^2\big) +O\big([\|z(t_{i+1})-z_{i+1}\|+\|u(t_i)-u_i\|]\, \|u(t_{i+1})-u(t_i)]\|\big) + O(h^2)\bigg).                                                              \end{multline*}
As in method 1, we denote by $C_2$ and $K$ the constants \eqref{C2} and \eqref{K}.
Then
$$
\varepsilon^u_{i+1}\le K \Big(\hat{g}q_2 [C_2+ \tilde{a}] p_1 \varepsilon^z_{i+1} +O([\varepsilon^z_{i+1}+\varepsilon^u_i]^2)+ O(\varepsilon^z_{i+1}+\varepsilon^u_i)O(h)+O(h^2)\Big),
$$
and consequently
 \begin{equation}\label{Eps_u_1}
\varepsilon^u_{i+1} = O(\varepsilon^z_{i+1})+ O\big((\varepsilon^z_{i+1})^2+(\varepsilon^u_i)^2\big)+O(h^2),\quad  i=0,...,N-1.
 \end{equation}
Similarly, we find that
 \begin{equation}\label{TildeEps_u_1}
\widetilde{\varepsilon}^u_{i+1}= O(\widetilde{\varepsilon}^z_{i+1})+ O\big((\widetilde{\varepsilon}^z_{i+1})^2+(\varepsilon^u_i)^2\big)+O(h^2),\quad  i=0,...,N-1.
 \end{equation}
Substituting \eqref{TildeEps_z_1} into \eqref{TildeEps_u_1} and carrying out certain manipulations, we obtain
 \begin{equation}\label{TildeEps_u_1-2}
\widetilde{\varepsilon}^u_{i+1} = [O(1)+O(h)]\big(\varepsilon^z_i+ (\varepsilon^z_i)^2+(\varepsilon^u_i)^2\big)+O(h)\varepsilon^u_i+O(h^2),\quad  i=0,...,N-1.
 \end{equation}
Substituting \eqref{TildeEps_u_1-2} into \eqref{Nnum6}, we get
 \begin{equation}\label{Eps_z1-2}
\varepsilon^z_{i+1}\le \hat{r}(h)\varepsilon^z_i+O(h)\varepsilon^u_i+ O(h)\big[(\varepsilon^z_i)^2+(\varepsilon^u_i)^2\big]+ O(h^3).
 \end{equation}
From the above, we obtain $\varepsilon^z_0=\varepsilon^u_0=\widetilde{\varepsilon}^z_1=0$,  $\widetilde{\varepsilon}^u_1=O(h^2)$, $\varepsilon^u_1=O(h^2)$, $\varepsilon^z_1=O(h^3)$, $\widetilde{\varepsilon}^u_2=O(h^2)$, $\varepsilon^u_2=O(h^2)$ and $\varepsilon^z_2=O(h^3)$.
Using \eqref{Eps_z1-2}, we find recurrently the estimate
 \begin{equation}\label{Eps_z1-3}
\varepsilon^z_{i+1}\le O(h) \bigg( \sum\limits_{j=0}^i\big[\varepsilon^u_j+(\varepsilon^u_j)^2\big]+  \sum\limits_{k=1}^{i-1}\sum\limits_{j=0}^{i-k} \big[\varepsilon^u_j+(\varepsilon^u_j)^2\big]^{2k} \bigg)+ O(h^2), \quad  i=2,...,N-1.
 \end{equation}
Substituting \eqref{Eps_z1-3} into \eqref{Eps_u_1}, we get
 \begin{equation}\label{Eps_u_1-2}
\begin{split}
\varepsilon^u_{i+1}\le\, &  O(h)\sum\limits_{j=0}^i \big[\varepsilon^u_j+(\varepsilon^u_j)^2+(\varepsilon^u_j)^4\big]+  O((\varepsilon^u_i)^2)+ O(h)\sum\limits_{k=1}^{i-1}\sum\limits_{j=0}^{i-k} \big[\varepsilon^u_j+(\varepsilon^u_j)^2\big]^{2k}    \\
& + \sum\limits_{k=1}^{i-1}\sum\limits_{j=0}^{i-k} \big[\varepsilon^u_j+(\varepsilon^u_j)^2\big]^{4k} +O(h^2), \quad  i=2,...,N-1.
\end{split}
 \end{equation}
Further, using the method of mathematical induction, we find that $\varepsilon^u_{i+1}=O(h^2)$, $i=0,...,N-1$. Then it follows from \eqref{Eps_z1-3} that $\varepsilon^z_{i+1}=O(h^2)$, $i=0,...,N-1$. Hence,
$$
\max\limits_{1\le i\le N}\varepsilon^u_i=O(h^2),\quad \max\limits_{1\le i\le N}\varepsilon^z_i=O(h^2),\qquad h\to 0,
$$
and therefore $\max\limits_{0\le i\le N}\|x(t_i)-x_i\|=O(h^2)$, $h\to 0$. Thus, the method \eqref{NImpMet1}--\eqref{NImpMet6} converges and has the second order.
\end{proof}
 \begin{proposition}\label{remModNum-meth}
If in Theorem~\ref{ThNum-meth2} we do not require the additional smoothness for $f$, $A, B$ and $C_2$, i.e., we assume that $f\in C([t_+,\infty)\times\Rn,\Rn)$, $\dfrac{\partial f}{\partial x}\in C([t_+,\infty)\times\Rn, \mathrm{L}(\Rn))$, $A, B\in C^1([t_+,\infty),\mathrm{L}(\Rn))$ and $C_2\in C^1([t_+,\infty),(0,\infty))$ (these restrictions are specified in Theorems \ref{Th_GlobSolv} and~\ref{Th_GlobSolvBInv}), then the method \eqref{NImpMet1}--\eqref{NImpMet6} is convergent, but may not have the second order.
 \end{proposition}
 \begin{proof}
The proof is carried out in the same way as the proof of Theorem~\ref{ThNum-meth2}, but with the use of the representations \eqref{Taylor2-1_o-small}.
 \end{proof}

Note that if condition~\ref{InvRN1} of Theorem~\ref{Th_GlobSolv} is fulfilled (i.e., the operator \eqref{funcPhi} is invertible)   for each $t_*\!\in\![t_+,\infty)$, $x_{p_1}^*(t_*)\!\in\! X_1(t_*)$, $x_{p_2}^*(t_*)\!\in\! X_2(t_*)$, and not only for those that  $(t_*,x_{p_1}^*(t_*)+x_{p_2}^*(t_*)) \!\in\! L_{t_+}$,  or if condition~\ref{BasInvRN1} of Theorem \ref{Th_GlobSolvBInv} is fulfilled (i.e., the operator function \eqref{funcPhiBInv} is basis invertible on $[x_{p_2}^1(t_*),x_{p_2}^2(t_*)]$\,)\,   for each $t_*\!\in\![t_+,\infty)$, $x_{p_1}^*(t_*)\!\in\! X_1(t_*)$, $x_{p_2}^i(t_*)\!\in\! X_2(t_*)$, ${i=1,2}$, then in Theorems~\ref{ThNum-meth} and~\ref{ThNum-meth2} it is not necessary to check the fulfillment of the additional condition of the invertibility of the operator $\Phi_{t_*,P_1(t_*)z_*,P_2(t_*)u_*}= \Phi_{t_*,P_1(t_*)z_*}(P_2(t_*)u_*)$  for each $(t_*,P_1(t_*)z_*+P_2(t_*)u_*)\in [t_0,T]\!\times\!\Rn$.

 \begin{remark}\label{Num-methDAE1}
Since it is assumed that the operator function $A(t)$ is continuously differentiable, we can write the DAE \eqref{DAE1}, i.e., $A(t)\frac{d}{dt}x(t)+B(t)x(t)=f(t,x(t))$, in the form
\begin{equation}\label{DAE1form2}
\dfrac{d}{dt}[A(t)x(t)]+\widetilde{B}(t)x(t)=f(t,x(t)),\quad \text{where}\quad \widetilde{B}(t)=B(t)-A'(t),
\end{equation}
and use the numerical methods obtained for the DAE of the form \eqref{DAE}.

For the IVP \eqref{DAE1form2}, \eqref{ini} as well as for the IVP  \eqref{DAE1}, \eqref{ini}, the \emph{consistency condition} for the initial values $t_0$,~$x_0$ takes the form
$$
(t_0,x_0)\in \widehat{L}_{t_+} = \{(t,x)\in [t_+,\infty)\times\Rn \mid Q_2(t)[B(t)x-f(t,x)]=0\}.
$$
 \end{remark}

  \section{Numerical experiments}\label{NumExp}

\vspace*{-1mm}

In Sections \ref{ApplPIMM}, \ref{ApplDE} we carry out the theoretical and numerical analyses of mathematical models of the dynamics of electric circuits, which demonstrate the application of the developed methods and obtained theorems to a real physical problems and show that the theoretical and numerical results are consistent.

In Section~\ref{CompareMeth}, the comparative analysis of the obtained methods is carried out and numerical examples illustrating the proved convergence are presented.

All computations were done using Matlab.

  \subsection{Example 1: Analysis of a mathematical model of the electrical circuit dynamics}\label{ApplPIMM}

 \subsubsection{Theoretical analysis of the mathematical model of the electrical circuit dynamics}\label{TheorAnalPIMM}

\vspace*{-1mm}

Consider the simple electrical circuit with a time-varying inductance $L(t)$, time-varying linear resistances $R(t)$, $R_L(t)$ and nonlinear resistances  $\varphi_L (I_L)$, $\varphi (I_{\varphi})$, whose dynamics is described by the DAE \eqref{DAE} (where we omit the dependence on $t$ in the notation of the variable $x(t)$) with
 \begin{equation}\label{NestCoefDAE}
x \!=\!\!\begin{pmatrix} x_1 \\ x_2 \\ x_3 \end{pmatrix}\!\!,\,
A(t) \!=\!\!\begin{pmatrix} L(t) & 0 & 0 \\ 0 & 0 & 0 \\ 0 & 0 & 0 \end{pmatrix}\!\!,\,
B(t) \!=\!\!\begin{pmatrix} R_L(t) & -1 & 0\\ 1 & 0 & 1\\ 0 & 1 & -R(t)\end{pmatrix}\!\!,\,
f(t,x) \!=\!\!\begin{pmatrix} -\varphi_L(x_1)\\ I(t)\\ U(t)+\varphi(x_3) \end{pmatrix}\!\!,
 \end{equation}
where $I(t)$ and $U(t)$ are a given (input) current and a given voltage,   ${x_1=I_L}$, ${x_3=I_{\varphi}}$ and ${x_2=U_L}$ are unknown currents and an unknown voltage.  The remaining currents and voltages in the circuit are uniquely expressed via the desired and given ones.

Using the formulas \eqref{ProjRes} and, accordingly,  the algorithm given in Section \ref{NumProj}, we compute
 \begin{equation}\label{Ex_1-Proj}
\begin{split}
P_1(t) \!=\!\!\begin{pmatrix} 1 & 0 & 0 \\ -R(t) & 0 & 0 \\ -1 & 0 & 0 \end{pmatrix}\!\!,
P_2(t) \!=\!\!\begin{pmatrix} 0 & 0 & 0 \\ R(t) & 1 & 0 \\ 1 & 0 & 1 \end{pmatrix}\!\!,  \\
Q_1(t)\!=\!\begin{pmatrix} 1 & R(t) & 1 \\ 0 & 0 & 0 \\ 0 & 0 & 0 \end{pmatrix}\!\!,
Q_2(t)\!=\!\!\begin{pmatrix} 0 & -R(t) & -1 \\ 0 & 1 & 0 \\ 0 & 0 & 1 \end{pmatrix}\!\!.
\end{split}
 \end{equation}
Then, the vector $x$ has the projections
$$
x_{p_1}(t)=P_1(t)x=(x_1,-R(t)x_1,-x_1)^T,\quad x_{p_2}(t)=P_2(t)x=(0,{R(t)x_1+x_2},{x_1+x_3})^{\T}.
$$
\paragraph{\textbf{Global solvability and Lagrange stability of the mathematical model \eqref{DAE}, \eqref{NestCoefDAE} of the electrical circuit dynamics.}}
Below, the definitions and theorems given in Section \ref{GlobSolv} are used. Recall that a solution of the initial value problem \eqref{DAE}, \eqref{NestCoefDAE},  \eqref{ini} is global if it exists on $[t_0,\infty)$.

\emph{By Theorem~\ref{Th_GlobSolv} as well as by Theorem~\ref{Th_GlobSolvBInv}, for each initial point $(t_0,x_0)\in [t_+,\infty)\times \R^3$, where $x_0 =(x_{0,1},x_{0,2},x_{0,3})^{\T}$, which satisfies the consistency condition $(t_0,x_0)\in L_{t_+}$, that is,}
$$
x_{0,1}+x_{0,3}=I(t_0),\qquad x_{0,2}-R(t_0)x_{0,3}=U(t_0)+\varphi(x_{0,3}),
$$
\emph{there exists a unique global solution of the DAE \eqref{DAE}, \eqref{NestCoefDAE} with the initial condition \eqref{ini} if the
following conditions hold:}
 \begin{equation}
\begin{split}
& \text{\emph{$L,\, R,\, R_L\!\in\! C^1([t_+,\infty),\R)$,\, $I,\, U\!\in\! C([t_+,\infty),\R)$,\,  $\varphi,\, \varphi_L\!\in\! C^1(\R)$,\, $L(t)\ge L_0>0$ and}}  \\
& \hspace{2cm} \text{\emph{$R(t)\ne 0$\, ($R(t)\!>\!0$ from physical considerations) for all $t\!\in\! [t_+,\infty)$, and}}  \\
& \text{\emph{$\lambda L(t)+R_L(t)+R(t)\ne 0$ for sufficiently large $|\lambda|$ such that $|\lambda|\ge L_0^{-1}$ and all $t\!\in\![t_+,\infty)$;}} \end{split}  \label{GlobSolvPIMM1}
 \end{equation}
 \begin{equation}
\begin{split}
& \text{\emph{there exists a number $R>0$ such that $[\varphi_L(x_1)-\varphi(I(t)-x_1)-R(t)I(t)-U(t)]x_1+$}}  \\
& \text{\emph{$+[L'(t)/2+R_L(t)+R(t)]x_1^2\ge0$  for all $t\!\in\! [t_+,\infty)$, $\|x_{p_1}(t)\|=|x_1|\|(1,-R(t),-1)^T\|\!\ge\! R$.}}
\end{split} \label{GlobSolvPIMM3}
 \end{equation}

  \smallskip
The condition \eqref{GlobSolvPIMM3} can be weakened by using Proposition \ref{St_GlobSolvSrez} presented in Section \ref{GlobSolv}.

 \smallskip
\emph{If the conditions \eqref{GlobSolvPIMM1}, \eqref{GlobSolvPIMM3} hold and $\sup\limits_{t\in [t_+,\infty)}\! |I(t)|<\infty$,  $\sup\limits_{t\in [t_+,\infty)}\! |U(t)|<\infty$, $\sup\limits_{t\in [t_+,\infty)}\! |R(t)|<\infty$, then by Theorem~\ref{Th_UstLagrDAE} the DAE \eqref{DAE}, \eqref{NestCoefDAE} is Lagrange stable}, i.e., for each consistent initial point $(t_0,x_0)$ a global solution of the IVP \eqref{DAE}, \eqref{NestCoefDAE}, \eqref{ini} exists and is bounded.

 \subsubsection{Numerical analysis of the mathematical model}\label{NumAnalPIMM}

Let us seek numerical solutions of the DAE \eqref{DAE}, \eqref{NestCoefDAE} provided that there exist corresponding exact global solutions, i.e., that the global solvability conditions presented in Section \ref{TheorAnalPIMM} are satisfied. The functions used in numerical experiments satisfy the conditions of the proved theorems or propositions on the convergence of the methods. This enables one to compute a numerical solution on any given time interval.

Choose
 \begin{equation} \label{Ex2_param}
\begin{split}
&U(t)=2\sin(2t+\pi),\quad {I(t)=\sin(2t-\pi)},\quad {L(t)=10^{-1}+(t+1)^{-1}}, \\
&{R_L(t)=3+0.5\sin(2t)},\quad {R(t)=1+0.5\sin(2t)},
\end{split}
 \end{equation}
\begin{equation}
\varphi(x_3)=a\, x_3^{2k-1},\quad \varphi_L(x_1)=b\, x_1^{2m-1},  \label{StepenFunc2}
\end{equation}
where $a=b=1$, $k=m=2$.  For the chosen functions the DAE \eqref{DAE}, \eqref{NestCoefDAE} is Lagrange stable since the conditions of the Lagrange stability, specified in Section \ref{TheorAnalPIMM}, hold.  The components of the solution $x(t)=(x_1(t),x_2(t),x_3(t))^\T$ computed (by method 1) for the consistent initial values $t_0=0$, $x_0=(0,0,0)^\T$ are displayed in Fig.~\ref{Ex2_h_001}.
 \begin{figure}[!h]%
 \centering
\includegraphics[width=4.8cm]{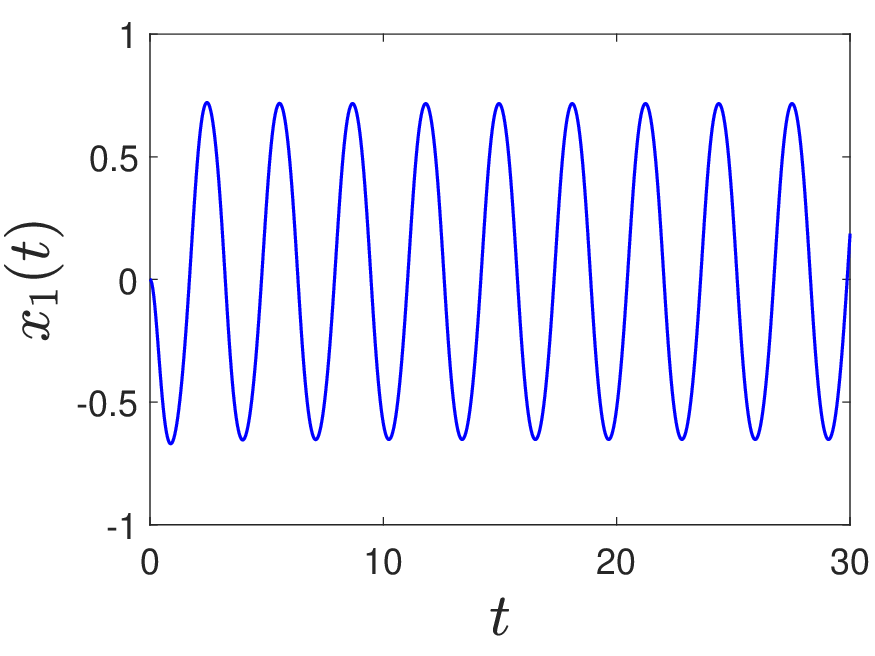}
\includegraphics[width=4.8cm]{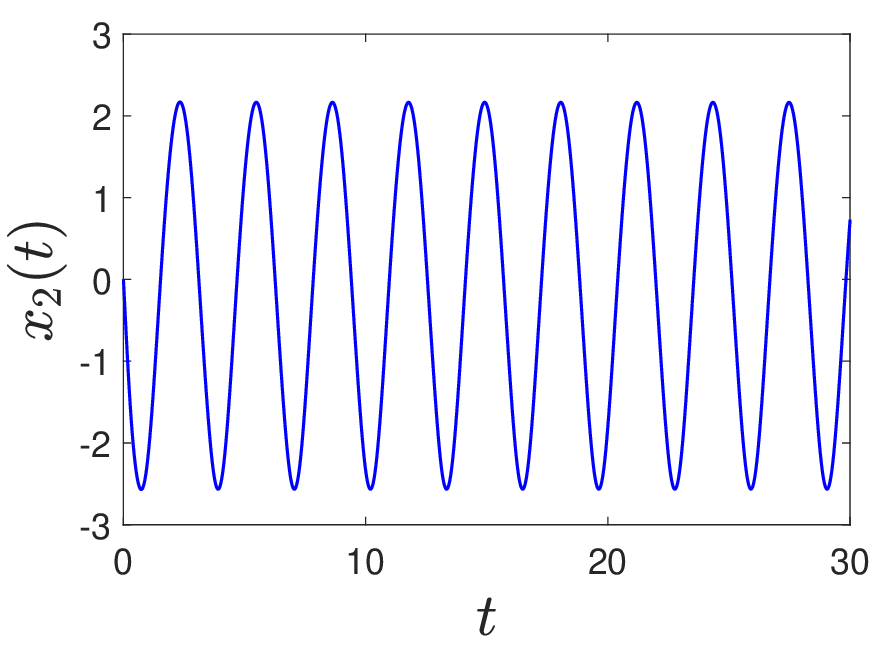}
\includegraphics[width=4.8cm]{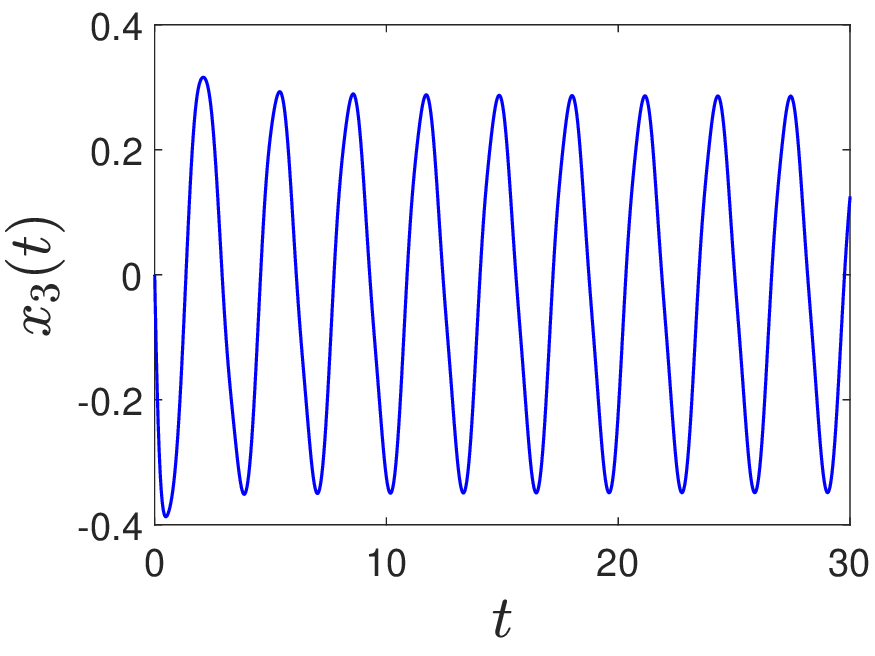}

\vspace*{-3mm}

\caption{The example of a Lagrange-stable solution: The plots of the components $x_1(t)$, $x_2(t)$, $x_3(t)$ of the approximating solution $x(t)=(x_1(t),x_2(t),x_3(t))^\T$ computed for the DAE \eqref{DAE}, \eqref{NestCoefDAE} with the functions \eqref{Ex2_param}, \eqref{StepenFunc2}, where $a=b=1$, $k=m=2$, and the initial values $t_0=0$, $x_0=(0,0,0)^\T$. The analysis of the presented graphs shows that the solution exists on the given interval and its norm does not increase with increasing time. When the interval is increased by a factor of 10 and more, the qualitative picture of the behavior of the numerical solution does not change (therefore, the corresponding graphs were not presented here).
Thus, the results of the numerical experiment are consistent with the conclusion about the Lagrange stability of the DAE, which was obtained using the corresponding theorem.}\label{Ex2_h_001}
 \end{figure}

For the functions
\begin{equation}\label{Ex1_1-3_param}
U(t)=t+1,\quad I(t)=3(t+1)^{-1},\quad L(t)=10^{-1}+(t+1)^{-1},\quad
R_L(t)=e^{-t},\quad R(t)=2+\cos t
\end{equation}
and $\varphi$, $\varphi_L$ of the form \eqref{StepenFunc2} where $a=b=1$ and $k=m=2$, and for the consistent initial values $t_0=0$ and $x_0=(0,37,3)^\T$, the components of the numerical solution are plotted in Fig.~\ref{Ex1_1-3_h_001}.
 \begin{figure}[!h]%
 \centering
\includegraphics[width=4.8cm]{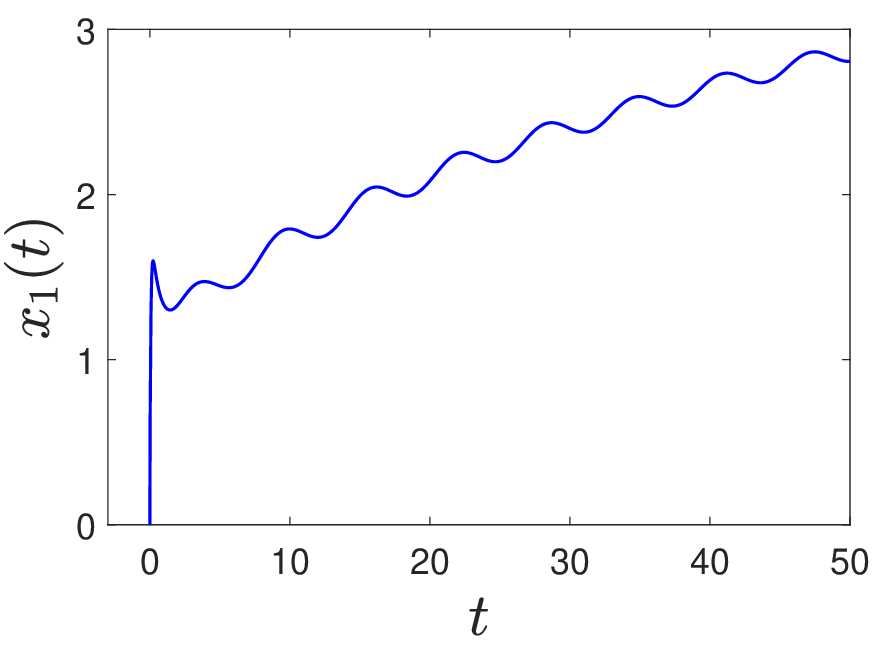}
\includegraphics[width=4.8cm]{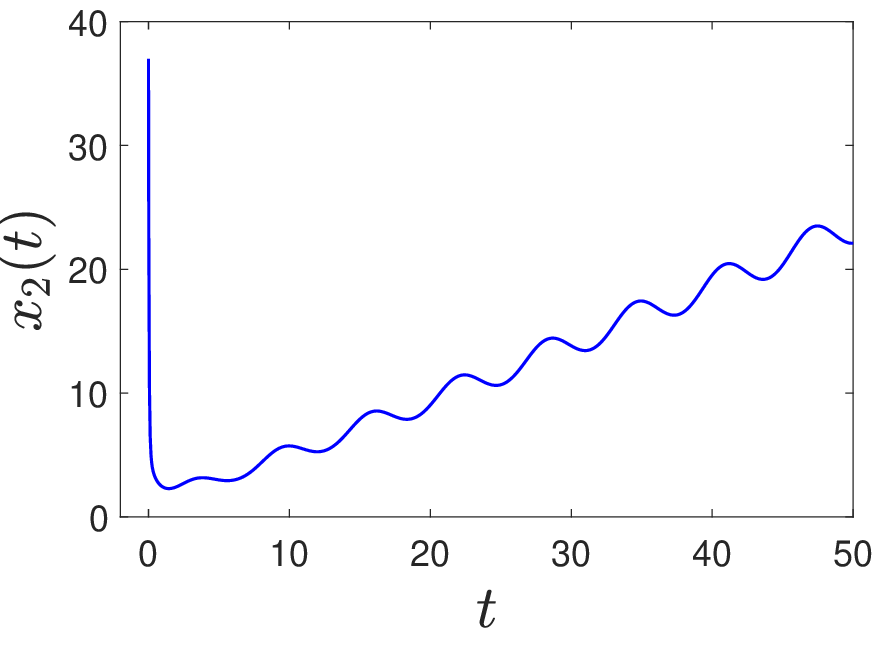}
\includegraphics[width=4.8cm]{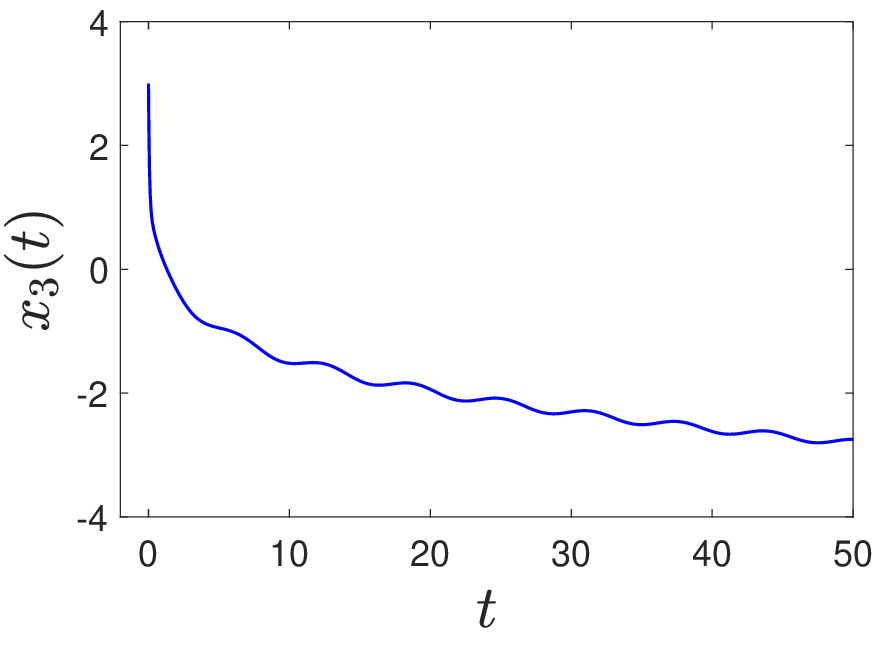}

\vspace*{-3mm}

\caption{The example of a global solution:  The plots of the components $x_1(t)$, $x_2(t)$ and $x_3(t)$ of the approximating solution $x(t)=(x_1(t),x_2(t),x_3(t))^\T$ computed (by method 1) for the DAE \eqref{DAE}, \eqref{NestCoefDAE} with the functions \eqref{Ex1_1-3_param} and \eqref{StepenFunc2}, where $a=b=1$, $k=m=2$, and the initial values $t_0=0$ and $x_0=(0,37,3)^\T$.
In this case, the exact solution is global, but it can be unbounded on $[t_0,\infty)$. This is because the conditions for the global solvability of the DAE \eqref{DAE}, \eqref{NestCoefDAE}, specified in Section \ref{TheorAnalPIMM}, hold, but the additional conditions for the Lagrange stability are not fulfilled.  The presented graphs demonstrate the same behavior pattern of the solution. When the interval is increased by a factor of 10, the qualitative picture of the behavior of the solution does not change.}\label{Ex1_1-3_h_001}
 \end{figure}

Further, consider the case when the function $U(t)$ is continuous, but not differentiable. Let the voltage $U(t)$ have the sawtooth shape (see Fig.~\ref{Upila})
\begin{equation}\label{Upila-func}
U(t)=\!\begin{cases} t-15\,i, & t\in [15\,i, 10+ 15\,i],\quad i\in\{0\}\cup{\mathbb N}, \\ 30(i+1)-2t, & t\in [10+ 15\,i, 15+15\,i],\quad i\in\{0\}\cup{\mathbb N}. \end{cases}
\end{equation}
Also, let
\begin{equation}\label{ExPila_param}
{I(t)=\sin(2t-\pi)},\,\; {L(t)=10^{-1}+(t+1)^{-1}},\,\; {R_L(t)=3+0.5\sin(2t)},\,\;  R(t)=1+0.5\sin(2t)
\end{equation}
 \begin{wrapfigure}[6]{r}{0.3\linewidth}
\includegraphics[width=\linewidth]{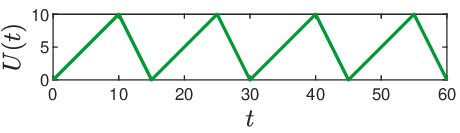}

\vspace*{-3mm}

\caption{The plot of $U(t)$}\label{Upila}
 \end{wrapfigure}
and $\varphi,\,\varphi_L$ have the form \eqref{StepenFunc2} where $a=3$, $b=4$ and ${k=m=2}$. In this case, the DAE \eqref{DAE}, \eqref{NestCoefDAE} is Lagrange stable since the conditions for the Lagrange stability, given in  Section \ref{TheorAnalPIMM}, hold.  The numerical solution for this case was obtained by methods 1 and 2 for the consistent initial values $t_0=0$, ${x_0=(0,0,0)^\T}$. Its components obtained by method 1 are displayed in Fig.~\ref{ExPila_h_001}.  When the calculation interval increases, the qualitative picture of the solution behavior does not change. The analysis of the numerical solution shows that the results of the numerical experiment are consistent with the conclusion about the Lagrange stability of the exact solution.
\begin{figure}[H]%
 \centering
\includegraphics[width=4.8cm]{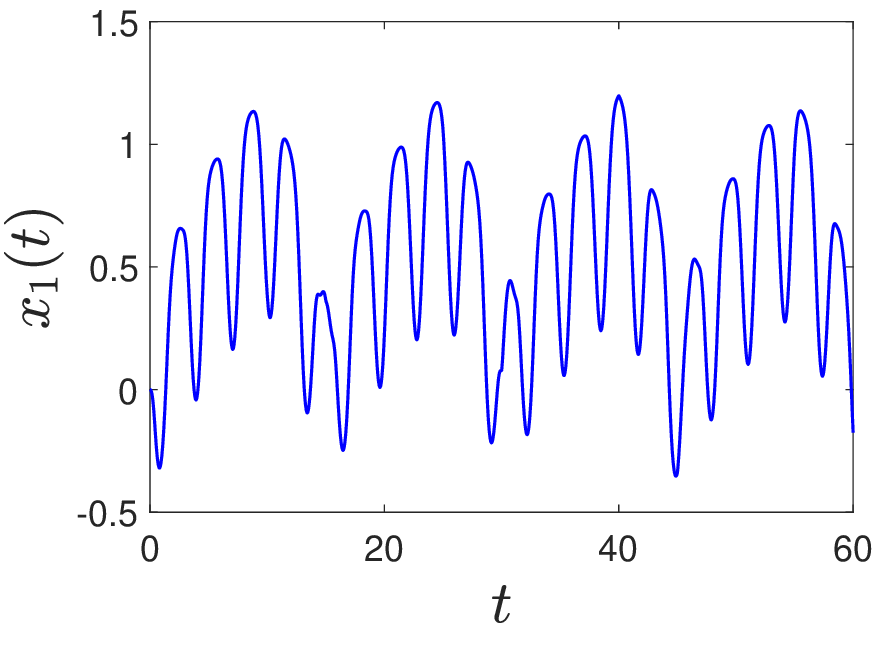}
\includegraphics[width=4.8cm]{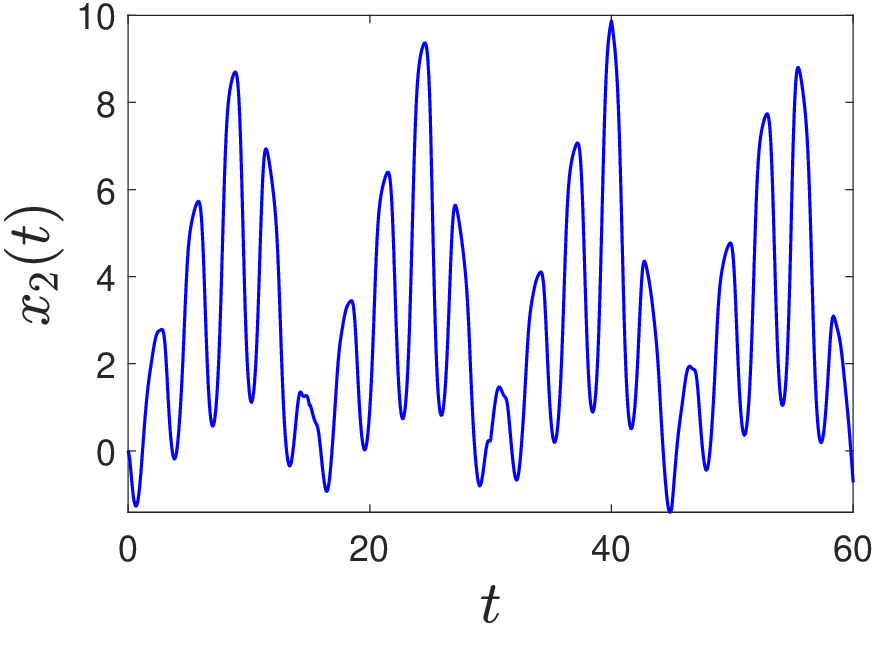}
\includegraphics[width=4.8cm]{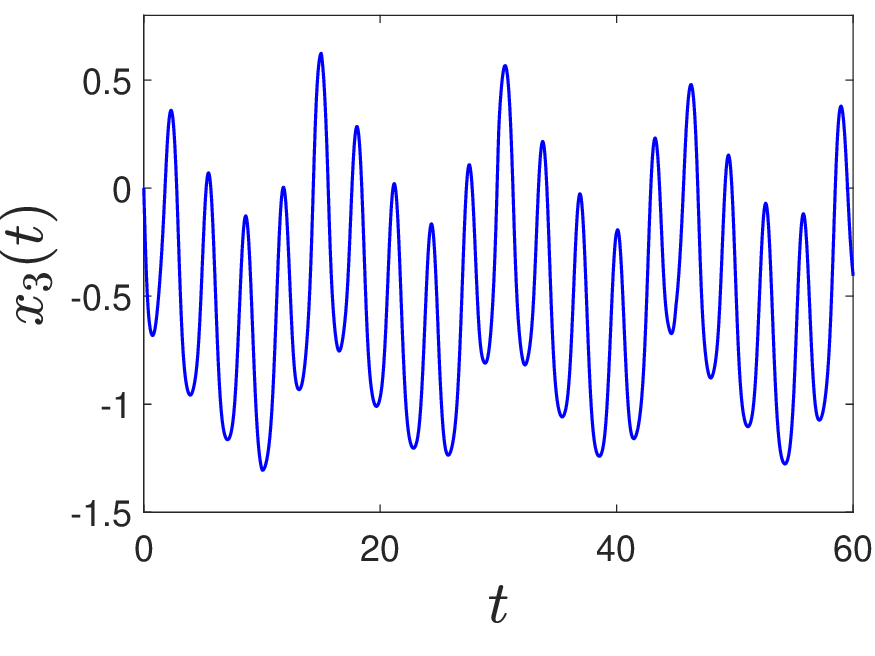}

\vspace*{-3mm}

\caption{The example of a solution for the case when the function $U(t)$ is continuous, but not differentiable: The plots of the components $x_1(t)$, $x_2(t)$, $x_3(t)$ of the approximate solution $x(t)=(x_1(t),x_2(t),x_3(t))^\T$ computed for the DAE \eqref{DAE}, \eqref{NestCoefDAE} with the functions \eqref{Upila-func}, \eqref{ExPila_param} and \eqref{StepenFunc2}, where $a=3$, $b=4$, $k=2$, $m=2$, and the initial values $t_0=0$, $x_0=(0,0,0)^\T$. The theoretical analysis shows that the exact solution is Lagrange stable and the presented plots demonstrate the same behavior pattern of the approximate solution.}\label{ExPila_h_001}
\end{figure}

The analysis of the obtained numerical solutions shows that the results of the numerical experiments are consistent with the results of the theoretical analysis of the DAE \eqref{DAE}, \eqref{NestCoefDAE}.

   \subsection{Example 2: Analysis of a mathematical model of the electrical circuit dynamics}\label{ApplDE}

 \subsubsection{Theoretical analysis of the mathematical model of the electrical circuit dynamics}\label{NestElCirc-TheorAnal}

Consider an electrical circuit whose diagram is given in Fig.~\ref{NestElCirc} (reference directions for currents and voltages across the circuit elements coincide).  The global solvability of the mathematical model \eqref{DAE}, \eqref{NestCoef2DAE} (see below) describing the circuit dynamics has been studied in \cite[Section 5]{Fil.DE-2}.
In the present section, we provide the conditions for the existence, uniqueness and boundedness of a global solution of the IVP \eqref{DAE}, \eqref{NestCoef2DAE}, \eqref{ini}  both in the general case and in the particular cases for which approximate solutions are found using the obtained numerical methods (see Section \ref{NestElCirc-NumAnal}).
\begin{figure}[!h]%
 \centering
   \includegraphics[width=7.6cm]{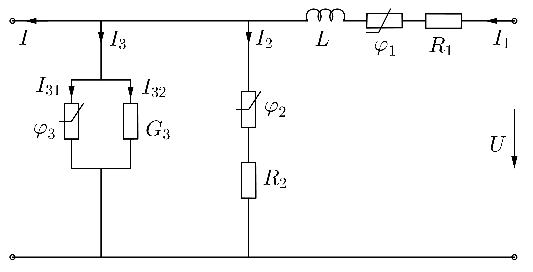}

\vspace*{-2mm}

  \caption{The electric circuit diagram}\label{NestElCirc}
\end{figure}

An inductance $L(t)$, a conductance $G_3(t)$ and resistances $R_1(t)$, $R_2(t)$,  $\varphi_1 (I_1)$,  $\varphi_2 (I_2)$ and $\varphi_3(I_{31})$ are given for the circuit.  Inductance, resistance and conductance are given in henries (H), ohm ($\Omega$) and siemens (S), respectively.

We denote the unknown currents by $x_1(t)=I_1(t)$, $x_2(t)=I_{31}(t)$ and $x_3(t)=I_2(t)$ and in the sequel, for brevity, omit the dependence on $t$ in the notation for $x_j(t)$ ($j=1,2,3$). The mathematical model of the electrical circuit dynamics has the form of the system
\begin{align}
 & \frac{d}{dt}[L(t)x_1]+R_1(t)x_1= U(t)-\varphi_1(x_1)-\varphi_3(x_2), \label{NestElSyst1x} \\
 & x_1-x_2-x_3= I(t)+G_3(t)\varphi_3(x_2),  \label{NestElSyst2x} \\
 & R_2(t)x_3= \varphi_3(x_2)-\varphi_2(x_3), \label{NestElSyst3x}
\end{align}
which describes a transient process in the electrical circuit. The current $I(t)$ and voltage $U(t)$ are given. Having solved the obtained system, we find the currents  $I_1(t)$, $I_{31}(t)$, $I_2(t)$. The remaining currents and voltages in the circuit are uniquely expressed via the desired and given ones. The mathematical model \eqref{NestElSyst1x}--\eqref{NestElSyst3x} can be represented as the DAE \eqref{DAE} where
 \begin{equation}\label{NestCoef2DAE}
x\!=\!\!\begin{pmatrix}\! x_1 \\ x_2 \\ x_3 \!\end{pmatrix}\!\!,
A(t)\!=\!\!\begin{pmatrix}L(t) & 0 & 0 \\ 0 & 0 & 0 \\ 0 & 0 & 0 \end{pmatrix}\!\!,
B(t)\!=\!\!\begin{pmatrix} R_1(t) & 0 & 0\\ 1 & -1 & -1\\ 0 & 0 & R_2(t) \!\end{pmatrix}\!\!,
f(t,x)\!=\!\!\begin{pmatrix} U(t)-\varphi_1(x_1)-\varphi_3(x_2) \\ I(t)+G_3(t)\varphi_3(x_2) \\ \varphi_3(x_2)-\varphi_2(x_3) \!\end{pmatrix}\!\!.
 \end{equation}
We assume that $L,\, R_1,\, R_2\in C^1([t_+,\infty),\R)$, $\varphi_j\in  C^1(\R)$, $j=1,2,3$,\,  $I,\, U,\, G_3\in C([t_+,\infty),\R)$, and ${L(t),\, R_1(t),\, R_2(t),\, G_3(t)>0}$ for all ${t\in [t_+,\infty)}$. Then $A, B\in C^1([t_+,\infty),\mathrm{L}(\R^3))$, $f\in C([t_+,\infty)\times \R^3,\R^3)$, $\partial f/\partial x\in C([t_+,\infty)\times \R^3,\mathrm{L}(\R^3))$, for each $t$ the pencil $\lambda A(t)+B(t)$ is regular, and the condition \eqref{index1}, where ${C_1(t)=\sqrt{2}\big(1+R_2^{-1} (t)\big)+1}$ and ${C_2(t)=L^{-1}(t)\big(1+R_1(t)\big)+1}$, holds for all ${t\in [t_+,\infty)}$.

Using the formulas \eqref{ProjRes}, we obtain the projection matrices $P_j(t)$, $Q_j(t)$ \eqref{Proj.1} (the algorithm for computing the projection matrices \eqref{Proj.1} is given in Section \ref{NumProj}):
$$
P_1(t)\!=\!\begin{pmatrix} 1 & 0 & 0 \\ 1 & 0 & 0 \\ 0 & 0 & 0 \end{pmatrix}\!,\;
P_2(t)\!=\!\begin{pmatrix} 0 & 0 & 0 \\ -1 & 1 & 0 \\ 0 & 0 & 1 \end{pmatrix}\!,\;
Q_1(t)\!=\!\begin{pmatrix} 1 & 0 & 0 \\ 0 & 0 & 0 \\ 0 & 0 & 0 \end{pmatrix}\!,\;
Q_2(t)\!=\!\begin{pmatrix} 0 & 0 & 0 \\ 0 & 1 & 0 \\ 0 & 0 & 1 \end{pmatrix}\!.
$$
Then we obtain the matrix $G(t)$ by \eqref{G(t)}.
The vector $x$ has the projections (components)
$$
x_{p_1}(t)=P_1(t)x=(x_1,x_1,0)^\T=:x_{p_1},\quad  x_{p_2}(t)=P_2(t)x=(0,x_2-x_1,x_3)^\T=:x_{p_2}.
$$
Denote ${z=x_1}$, ${u=x_2-x_1}$, ${w=x_3}$, then
$$
x_{p_1}=(z,z,0)^\T,\quad x_{p_2}=(0,u,w)^\T.
$$

The consistency condition $(t,x)\in L_{t_+}$ (see Remark \ref{RemConsistInis}) holds if $t$, $x_i$, $i=1,2,3$, satisfy the algebraic equations \eqref{NestElSyst2x}, \eqref{NestElSyst3x}.  Using the above notation, we can rewrite the system \eqref{NestElSyst2x},~\eqref{NestElSyst3x} as
 \begin{align*}
& u=-I(t)-G_3(t)\,\varphi_3(u+z)-w, \\
& w=R_2^{-1}(t)\,[\varphi_3(u+z)-\varphi_2(w)]
 \end{align*}
(see \cite[p.~205]{Fil.DE-2}) and transform it to the form
\begin{align}
 & w=-I(t)-u-G_3(t)\, \varphi_3(u+z), \label{NestElSyst2*} \\
 & u=\psi(t,z,u),\qquad  \text{where} \label{NestElSyst3*}
\end{align}
\begin{equation}\label{Nest_psi}
\psi(t,z,u)=-I(t)-\left(G_3(t)+R_2^{-1}(t)\right) \varphi_3(u+z) +R_2^{-1}(t)\, \varphi_2\big(-I(t)-u-G_3(t)\, \varphi_3(u+z)\big).
\end{equation}

Below, the theorems and propositions from Section \ref{GlobSolv} are used. The derivation of the constraints on the functions in the DAE \eqref{DAE}, \eqref{NestCoef2DAE}, under which the conditions of Theorems~\ref{Th_GlobSolv} and \ref{Th_GlobSolvBInv} are satisfied, is described in detail in \cite[Section 5]{Fil.DE-2}. The below conditions for the existence and uniqueness of a global solution of the IVP \eqref{DAE}, \eqref{NestCoef2DAE}, \eqref{ini} were obtained based on these results.

 \smallskip
\paragraph{\textbf{Global solvability of the mathematical model \eqref{DAE}, \eqref{NestCoef2DAE}.}}\;
\emph{By Theorem~\ref{Th_GlobSolv}, for each initial point $(t_0,x_0)\in [t_+,\infty)\times \R^3$, where $x_0 =(x_{0,1},x_{0,2},x_{0,3})^\T$, for which equalities  \eqref{NestElSyst2x}, \eqref{NestElSyst3x} hold (i.e., the consistency condition $(t_0,x_0)\in L_{t_+}$ holds), there exists a unique global solution of the DAE \eqref{DAE}, \eqref{NestCoef2DAE} satisfying the initial condition \eqref{ini}\, if the following conditions are fulfilled:}
 \begin{align}
\begin{split}
& \text{\emph{$L,\, R_1,\, R_2\in C^1([t_+,\infty),\R)$,\; $I,\, U,\, G_3\in C([t_+,\infty),\R)$,\; $\varphi_j\in C^1(\R)$, $j=1,2,3$, and}}  \\
 & \hspace{5.5cm} \text{\emph{${L(t),\, R_1(t),\, R_2(t),\, G_3(t)>0}$ for all $t\in [t_+,\infty)$;}}
 \end{split}   \label{GlobSolvNestEl1}\\  \smallskip
& \text{\emph{for each $t\!\in\! [t_+,\infty)$ and each $z\!\in\! \R$ there exists a unique $u\in\R$ such that \eqref{NestElSyst3*} holds;}} \label{NestElCond1}\\ \smallskip
\begin{split}
& \text{\emph{for each $t_*\in [t_+,\infty)$ and each $z_*,\, u_*,\, w_*\in\R$ which satisfy \eqref{NestElSyst2*}, \eqref{NestElSyst3*}, the relation}}    \\
 & \hspace{2.8cm} \text{\emph{$\varphi'_3(u_*+z_*)+ \big[\varphi'_2(w_*)+R_2(t_*)\big]\big[1+G_3(t_*)\, \varphi'_3(u_*+z_*)\big]\ne 0$\, holds;}}
 \end{split}   \label{NestElCond2} \\ \smallskip
\begin{split}
& \text{\emph{there exists a number $R>0$ such that\;    $-\big(\varphi_1(z)+\varphi_3(u+z)\big)z\le R_1(t)z^2$}}   \\
 & \hspace{4cm} \text{\emph{for any $t\in [t_+,\infty)$, $z,\, u\in\R$ satisfying \eqref{NestElSyst3*} and $|z|\ge R$.}}
 \end{split}   \label{GlobSolvNestEl3}
 \end{align}

\emph{By Theorem~\ref{Th_GlobSolvBInv}, a similar statement holds if the above conditions are satisfied with the following changes: the condition \eqref{NestElCond1} does not contain the requirement that $u$ be unique;\;  the condition \eqref{NestElCond2} is replaced by the following:}
 \begin{equation}
\begin{split}
& \text{\emph{for each $t_*\in [t_+,\infty)$, $z_*\in\R$ and each $u_*^j,\, w_*^j\in\R$, $j=1,2$, satisfying \eqref{NestElSyst2*}, \eqref{NestElSyst3*}}} \\
 & \hspace{1cm} \text{\emph{the relation $\varphi'_3(u_2+z_*)+ \big[\varphi'_2(w_2)+R_2(t_*)\big]\big[1+G_3(t_*)\, \varphi'_3(u_1+z_*)\big]\ne 0$ holds}}     \\
& \hspace{6.3cm} \text{\emph{for any $u_k\in [u_*^1,u_*^2]$, $w_k\in [w_*^1,w_*^2]$, $k=1,2$}}
\end{split}   \label{NestElCond2:2}
 \end{equation}
(obviously, this condition is satisfied in the case if the relation present in it holds for each $t_*\in [t_+,\infty)$, each $z_*\in\R$ and each $u_k,\, w_k\in\R$, $k=1,2$).

\emph{The global solvability conditions mentioned above can be weakened by using Proposition} \ref{St_GlobSolvSrez}.

Below, examples of the functions that satisfy the presented conditions are considered and certain changes of these conditions are discussed.

The conditions \eqref{NestElCond1}, \eqref{NestElCond2}, as well as  \eqref{NestElCond2:2}, hold if the functions $\varphi_2$, $\varphi_3$ are increasing (nondecreasing) on~$\R$, for example:
 \begin{equation}\label{PowerFunct}
\varphi_2(y)=a\, y^{2k-1},\, \varphi_3(y)=b\, y^{2m-1}\, \text{ or }\, \varphi_2(y)=a\, y^{\frac{1}{2k-1}},\,  \varphi_3(y)=b\, y^{\frac{1}{2m-1}},\,  a, b>0,\, k, m\!\in\!\mathbb N,
 \end{equation}
or if they have the form \eqref{Sin} and inequality \eqref{SinCond1} is satisfied:
 \begin{gather}
\varphi_2(y)=a\sin y,\; \varphi_3(y)=b\sin y\quad \text{ or }\quad \varphi_2(y)=a\cos y,\; \varphi_3(y)=b\cos y,\quad  a,\, b\in \R, \label{Sin} \\
G_3(t)|b|+R_2^{-1}(t)(|a|+|b|+G_3(t)|a|\, |b|)<1,\quad t\in [t_+,\infty). \label{SinCond1}
 \end{gather}

Note that if $\varphi_2$, $\varphi_3$ have the form \eqref{PowerFunct}, then the mapping $\psi(t,z,u)$ \eqref{Nest_psi} is not globally contractive with respect to $u$ (in general, it does not satisfy the global Lipschitz condition in $u$ and $z$) for $k,m\ge 2$ and any $G_3(t)$, $R_2(t)$, $a$ and $b$, and hence the condition \eqref{GlobContr} (see Section \ref{GlobSolv}) is not fulfilled. Obviously, if $\psi(t,z,u)$ is globally contractive with respect to $u$ for any $t$,~$z$, i.e., there exists a constant $\alpha<1$ such that
\begin{equation}\label{ContrPsi}
\big|\psi(t,z,u_1)-\psi(t,z,u_2)\big|\le\alpha |u_1-u_2|,\;\; u_1,\, u_2\in\R,
\end{equation}
for each $t\in [t_+,\infty)$ and each $z\in\R$, then the condition \eqref{NestElCond1} holds.

If we take into account that $t_*$, $z_*$, $u_*$, $w_*$ satisfy  \eqref{NestElSyst2*}, i.e., $w_*=-I(t_*)-u_*-G_3(t_*) \varphi_3(u_*+z_*)$, but  disregard equality \eqref{NestElSyst3*}, then the condition \eqref{NestElCond2}  takes the following form:
  \par{\centering
 for each $t_*\in [t_+,\infty)$ and each $z_*,\, u_*\in\R$ the relation $\dfrac{\partial\psi}{\partial u}(t_*,z_*,u_*)\ne -1$ holds. \par\smallskip}

The conditions \eqref{NestElCond1}, \eqref{NestElCond2} ensure the fulfillment of conditions \ref{SoglRN1}, \ref{InvRN1} of Theorem~\ref{Th_GlobSolv};\, \eqref{NestElCond1} without the requirement for $u$ to be unique and \eqref{NestElCond2:2} ensure the fulfillment of conditions \ref{SoglRN2},~\ref{BasInvRN1} of Theorem~\ref{Th_GlobSolvBInv}.
Instead of conditions \ref{SoglRN1}, \ref{InvRN1} of Theorem~\ref{Th_GlobSolv} or Theorem~\ref{Th_GlobSolvBInv} one can use the condition \eqref{GlobContr} of Proposition~\ref{Th_GlobSolvContr} which is satisfied if there exists a constant $\alpha<1$  such that
 \begin{equation}
 \begin{split}
G_3(t)\, \big|\varphi_3(u_1+z)-\varphi_3(u_2+z)\big|+ R_2^{-1}(t)\big|\varphi_3(u_1+z)-\varphi_3(u_2+z)- \varphi_2(w_1)+\varphi_2(w_2)\big|  & \\
 \le \alpha \sqrt{|u_1-u_2|^2+|w_1-w_2|^2} &
 \end{split}   \label{GlobContr1}
 \end{equation}
for any $t\in [t_+,\infty)$, $z\in \R$ and $u_i,\, w_i\in\R$, $i=1,2$, that is, the nonlinear function in the ``algebraic part'' of the DAE is a globally contractive with respect to $x_{p_2}$ for any $t$, $x_{p_1}$.  However, this condition is more restrictive.
If we take into account that the graph of a solution $x(t)$ must lie in the manifold $L_{t_+}$ and, therefore, $t$, $z$, $u$, $w$ are related by equalities \eqref{NestElSyst2*}, \eqref{NestElSyst3*}, then, using these equalities, we can transform inequality  \eqref{GlobContr1} so that it will be similar to \eqref{ContrPsi}.

To derive the condition \eqref{GlobSolvNestEl3}, the function $V(t,x_{p_1}(t))$ of the form \eqref{funcV} with a time-invariant operator $H$, i.e., $V(t,x_{p_1}(t))\equiv \big(H x_{p_1}(t),x_{p_1}(t)\big)$, where $H=0.5\, I_{\R^3}$, was chosen.
Then $V'_{\eqref{DAEsys2.1}}(t,x_{p_1}(t))$ has the form \eqref{VderivDAE} where $H(t)\equiv H$, and  condition~\ref{ExtensRN1} of Theorem~\ref{Th_GlobSolv} (the same condition is present in Theorem \ref{Th_GlobSolvBInv}) is satisfied if there exist functions ${\tilde{U}\in C(0,\infty)}$, ${k\in C([t_+,\infty),\R)}$ such that ${\int\limits_{{\textstyle v}_0}^{\infty}\big(\tilde{U}(v)\big)^{-1} dv =\infty}$ (${v_0>0}$) and for some ${R>0}$ the inequality
 \begin{equation}\label{NestElCond3}
2 L^{-1}(t) \big[-(L'(t)+R_1(t))z^2+U(t)z- \big(\varphi_1(z)+\varphi_3(u+z)\big)z\big]\le k(t)\, \tilde{U}(z^2)
 \end{equation}
holds for all $t\in [t_+,\infty)$, $z,u\in \R$ satisfying \eqref{NestElSyst3*} and $|z|\!\ge\! R$. It is readily verified that \eqref{NestElCond3}, where
$$
k(t)= 2L^{-1}(t)\big(|L'(t)|+|U(t)|\big),\qquad  {\tilde{U}(v)=v},
$$
is satisfied if \eqref{GlobSolvNestEl3} holds. The specified functions $k(t)$, $\tilde{U}(v)$ are also used to obtain conditions for the Lagrange stability of the DAE \eqref{DAE}, \eqref{NestCoef2DAE}.

 \smallskip
\paragraph{\textbf{Global solvability of the mathematical model \eqref{DAE}, \eqref{NestCoef2DAE} in some particular cases.}}\quad

 \smallskip
\textbf{I.}  Consider the functions
 \begin{equation}\label{PowerFunct_1-2-3}
\varphi_1(y)=c\, y^{2l-1},\quad  \varphi_2(y)=a\, y^{2k-1},\quad \varphi_3(y)=b\, y^{2m-1},\quad  a,\, b,\, c>0,\quad  k,\, m,\, l\in\mathbb N,
 \end{equation}
where $\varphi_2$, $\varphi_3$ from \eqref{PowerFunct}.\,  The functions  \eqref{PowerFunct_1-2-3} satisfy \eqref{GlobSolvNestEl3} if
$$
\sup\limits_{t\in [t_+,\infty)}\! |I(t)|<\infty,\quad \inf\limits_{t\in [t_+,\infty)} R_2(t)=K_0>0 \quad \text{($K_0$ is some constant) and ${m\le l}$}.
$$

Thus, \emph{if $\varphi_j$, ${j=1,2,3}$, have the form \eqref{PowerFunct_1-2-3}, where ${m\le l}$, and, in addition, ${L, R_1, R_2\in C^1([t_+,\infty),\R)}$,\, ${I, U, G_3\in C([t_+,\infty),\R)}$,\, ${L(t), R_1(t), G_3(t)>0}$ for ${t\in [t_+,\infty)}$, ${\sup\limits_{t\in [t_+,\infty)} |I(t)|<\infty}$ and ${\inf\limits_{t\in [t_+,\infty)} R_2(t)=K_0>0}$, then for each initial point ${(t_0,x_0)\in [t_+,\infty)\times \R^3}$ satisfying \eqref{NestElSyst2x}, \eqref{NestElSyst3x} there exists a unique global solution of the IVP \eqref{DAE}, \eqref{NestCoef2DAE}, \eqref{ini}.}

  \smallskip
   \textbf{II.}  Now consider the functions
 \begin{equation}\label{Sin_1-2-3}
\varphi_1(y)=c\sin y,\quad  \varphi_2(y)=a\sin y,\quad \varphi_3(y)=b\sin y,\quad  a,\, b,\, c \in\R,
 \end{equation}
where $\varphi_2$, $\varphi_3$ from \eqref{Sin} and we can replace sines by cosines in \eqref{Sin_1-2-3}.\,  For the functions \eqref{Sin_1-2-3} the condition \eqref{GlobSolvNestEl3} holds if  ${\inf\limits_{t\in [t_+,\infty)} R_1(t)=R_*>0}$. Notice that for the functions \eqref{Sin} condition~\ref{SoglRN2} of Theorem \ref{Th_GlobSolvBInv} is always satisfied.

Thus, \emph{if $\varphi_j$, $j=1,2,3$, have the form \eqref{Sin_1-2-3}, and, in addition, $L, R_1, R_2\in C^1([t_+,\infty),\R)$,\, $I, U, G_3\in C([t_+,\infty),\R)$,\, $L(t), R_2(t), G_3(t)>0$ for $t\in [t_+,\infty)$,  the functions $\varphi_2$, $\varphi_3$, $G_3$, $R_2$ satisfy the condition \eqref{SinCond1}, and $\inf\limits_{t\in [t_+,\infty)} R_1(t)=R_*>0$, then for each initial point $(t_0,x_0)\!\in\! [t_+,\infty)\!\times\! \R^3$  satisfying \eqref{NestElSyst2x}, \eqref{NestElSyst3x} there exists a unique global solution of the IVP}  \eqref{DAE}, \eqref{NestCoef2DAE}, \eqref{ini}.

  \smallskip
\paragraph{\textbf{Lagrange stability of the mathematical model \eqref{DAE}, \eqref{NestCoef2DAE}.}}\, \emph{By Theorem~\ref{Th_UstLagrDAE}, the DAE \eqref{DAE},~\eqref{NestCoef2DAE} is Lagrange stable if the above conditions \eqref{GlobSolvNestEl1}--\eqref{GlobSolvNestEl3} are fulfilled and in addition ${\int\limits_{t_+}^{\infty}L^{-1}(t)\big(|L'(t)|+|U(t)|\big)dt<\infty}$}
(this integral converges if ${\int\limits_{t_+}^{\infty} L^{-1}(t)|U(t)|dt<\infty}$ and ${\lim\limits_{t\to \infty}L(t) =\tilde{L} <\infty}$, ${\tilde{L}\ne 0}$) \emph{and condition \ref{LagrA1}, or \ref{LagrA1.2}, or \ref{LagrA1.3} from Theorem \ref{Th_UstLagrDAE} holds.}
Notice that condition \ref{LagrA1} of Theorem~\ref{Th_UstLagrDAE}  is a consequence of condition \ref{LagrA1.2}.

Condition \ref{LagrA1} of Theorem \ref{Th_UstLagrDAE}, as well as condition \ref{LagrA1.2}, holds if
$$
[I(t)+G_3(t)\varphi_3(x_2)+R_2^{-1}(t)(\varphi_3(x_2)-\varphi_2(x_3))]^2+  [R_2^{-1}(t)(\varphi_3(x_2)-\varphi_2(x_3))]^2\le K_{M_*}=\mathrm{const}
$$
for all $t\!\in\! [t_+,\infty)$, $x_2,\,x_3\!\in\!\R$ and $|x_1|\le M_*$ (where $M_*$ is an arbitrary constant) satisfying equalities \eqref{NestElSyst2x}, \eqref{NestElSyst3x}. These conditions are fulfilled, for example, if
$$
{\sup\limits_{t\in[t_+,\infty)}\! |I(t)|\!<\!\infty},\; \sup\limits_{t\in[t_+,\infty)}\! G_3(t)\!<\!\infty,\;  \sup\limits_{t\in[t_+,\infty)}\! R_2^{-1}(t)\!<\!\infty,\;
\sup\limits_{x_2\in\R}\! |\varphi_3(x_2)|\!<\!\infty,\;
\sup\limits_{x_3\in\R}\! |\varphi_2(x_3)|\!<\!\infty.
$$

Choose $\Tilde{x}_{p_2}(t_*)=\Tilde{x}_{p_2}= (0,\tilde{x}_2-x_1^*,\tilde{x}_3)^{\T}=(0,\tilde{u},\tilde{w})^{\T}=0$. Then it is easily verified that condition \ref{LagrA1.3} of Theorem \ref{Th_UstLagrDAE} is satisfied if, for example, the following conditions are satisfied:
\begin{itemize}
\leftmargin=0pt
\item for each ${t_*\in\! [t_+,\infty)}$ and each ${z_*,\, u_*,\, w_*\in\!\R}$ satisfying \eqref{NestElSyst2*}, \eqref{NestElSyst3*} and for any ${\lambda_1,\lambda_2\in\! (0,1]}$ the following relation holds:
    $${\varphi'_3(\lambda_2 u_*+z_*)+ \big[\varphi'_2(\lambda_2 w_*)+ R_2(t_*)\big]\big[1+G_3(t_*)\, \varphi'_3(\lambda_1 u_* +z_*)\big]\ne 0}
    $$
    (i.e., the relation from the condition \eqref{NestElCond2:2}, where ${u_i=\lambda_i u_*}$, ${i=1,2}$, and ${w_2=\lambda_2 w_*}$, holds);
\item for all ${t_*\in [t_+,\infty)}$, ${z_*, u_*, w_*\in \R}$ satisfying \eqref{NestElSyst2*}, \eqref{NestElSyst3*} it holds that ${|I(t_*)|<\infty}$,  ${G_3(t_*)<\infty}$, ${R_2^{-1}(t_*)<\infty}$, ${|\varphi_2(w_*)|<\infty}$ and ${|\varphi_3(u_*+z_*)|\le K_1(z_*)<\infty}$, where ${K_1(z_*)=K_1^*}$ is some constant for each fixed $z_*$.
\end{itemize}

\subsubsection{Numerical analysis of the mathematical model of the electrical circuit dynamics}\label{NestElCirc-NumAnal}

In this section, we present the plots of numerical solutions of the DAE \eqref{DAE}, \eqref{NestCoef2DAE} describing the electrical circuit dynamics (see Section \ref{NestElCirc-TheorAnal}) for such parameters of the electric circuit (i.e., the functions $I(t)$, $U(t)$, $G_3(t)$, $L(t)$, $R_1(t)$, $R_2(t)$, $\varphi_1(x_1)$, $\varphi_2(x_3)$ and $\varphi_3(x_2)$\,) for which there exists a unique global solution of the IVP \eqref{DAE}, \eqref{NestCoef2DAE}, \eqref{ini}, as well as the conditions of Theorems \ref{ThNum-meth}, \ref{ThNum-meth2} or Propositions~\ref{remNum-meth},~\ref{remModNum-meth} on the convergence of the methods hold.

Consider the case when $\varphi_i$, $i=1,2,3$, have the form \eqref{PowerFunct_1-2-3}, where $k=m=l=2$, i.e.,
 \begin{equation}\label{StepenFunc}
\varphi_1(y)=c\, y^3,\quad \varphi_2(y)=a\, y^3,\quad  \varphi_3(y)=b\, y^3,\quad a,b,c>0,\quad y\in\R.
 \end{equation}
Let $L,\, R_1,\, R_2\in C^1([t_+,\infty),\R)$,\, $I,\, U,\, G_3\in C([t_+,\infty),\R)$,\, $L(t),\, R_1(t),\, G_3(t)>0$ for all $t\!\in\! [t_+,\infty)$,\, ${\sup\limits_{t\in[t_+,\infty)} |I(t)|<\infty}$ and $\inf\limits_{t\in[t_+,\infty)} R_2(t)=K_0>0$. Then, as shown in Section \ref{NestElCirc-TheorAnal}, for each initial point $(t_0,x_0)\!\in\! [t_+,\infty)\times \R^3$ satisfying equalities \eqref{NestElSyst2x}, \eqref{NestElSyst3x} there exists a unique global solution of the IVP for the DAE \eqref{DAE}, \eqref{NestCoef2DAE} with the functions \eqref{StepenFunc} and the initial condition \eqref{ini}.

Note that equalities \eqref{NestElSyst2x}, \eqref{NestElSyst3x} can be transformed into the following form:
 \begin{gather}
x_3=x_1-x_2-I(t)-G_3(t)\, \varphi_3(x_2),  \label{x_3} \\
x_2=x_1-I(t)-\left(G_3(t)+R_2^{-1}(t)\right) \varphi_3(x_2) +R_2^{-1}(t)\, \varphi_2\big(x_1-x_2-I(t)-G_3(t)\, \varphi_3(x_2)\big)  \label{x_2}
 \end{gather}
(recall that if $(t,x)$ satisfy \eqref{NestElSyst2x}, \eqref{NestElSyst3x}, then $(t,x)\in L_{t_+}$),  and the condition \eqref{NestElCond1} can be rewritten as follows: for each $t\in [t_+,\infty)$ and each $x_1\in\R$ there exists a unique $x_2\in\R$ such that \eqref{x_2} holds.  Consequently, by setting arbitrary initial values $t_0\in [t_+,\infty)$ and $x_{0,1}\in\R$, one can always find a unique $x_{0,2}$ by the formula \eqref{x_2} and then find a unique $x_{0,3}$ by the formula \eqref{x_3} such that the initial point $(t_0,x_0)$, where $x_0=(x_{0,1},x_{0,2},x_{0,3})^{\T}$, will be consistent.\,  For example, in the particular case when $\varphi_i$ have the form \eqref{StepenFunc}, if $t_+=t_0=0$, $x_{0,1}=0$ and $I(t)$ is such that $I(0)=0$, then $t_0=0$, $x_0=(0,0,0)^{\T}$ are consistent initial values.

Recall that the components of a solution $x(t)=(x_1(t),x_2(t),x_3(t))^{\T}$ denote the functions of the currents, namely,
$$
x_1(t)=I_1(t),\quad  x_2(t)=I_{31}(t),\quad  x_3(t)=I_2(t).
$$

Consider the case when
 \begin{equation}\label{Example8_param}
\begin{split}
& I(t)=(t+1)^{-1}-1,\quad U(t)=t+1,\quad G_3(t)=(t+1)^2,\quad L(t)=500\,(t+1)^{-1},   \\
& R_1(t)=1+(t+1)^{-1},\quad  R_2(t)=t(t+1)^{-1}
\end{split}
 \end{equation}
and $\varphi_i$, $i=1,2,3$, have the form \eqref{StepenFunc} where $a=b=c=1$, and take the consistent initial values $t_0=0$, $x_0=(0,0,0)^\T$.  As mentioned above, an exact solution of the DAE is global, i.e., exists on $[t_0,\infty)$, in all cases considered in this section. However, in this case, the solution can be unbounded on $[t_0,\infty)$, since the conditions for the Lagrange stability, specified in Section \ref{NestElCirc-TheorAnal}, are not fulfilled.
The components of the numerical solution (obtained by method 2) are plotted in Fig.~\ref{Example8_1-3}.
 \begin{figure}[!h]%
 \centering
\includegraphics[width=4.8cm]{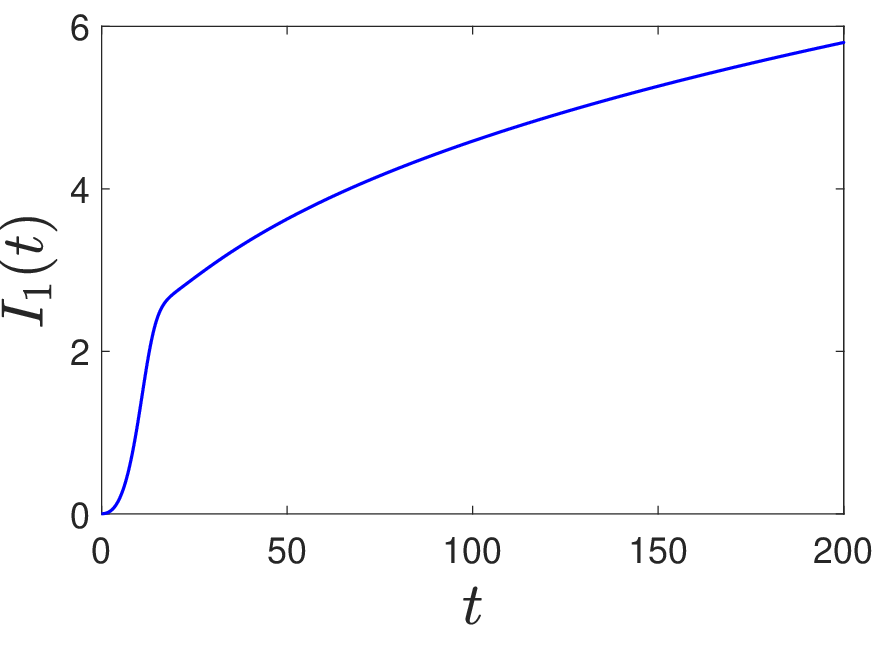}
\includegraphics[width=4.8cm]{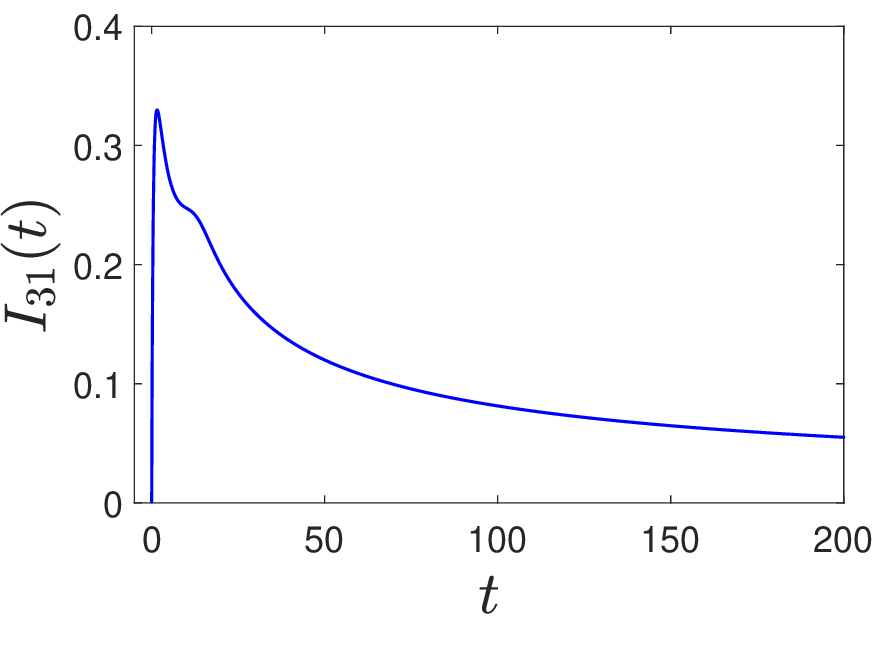}
\includegraphics[width=4.8cm]{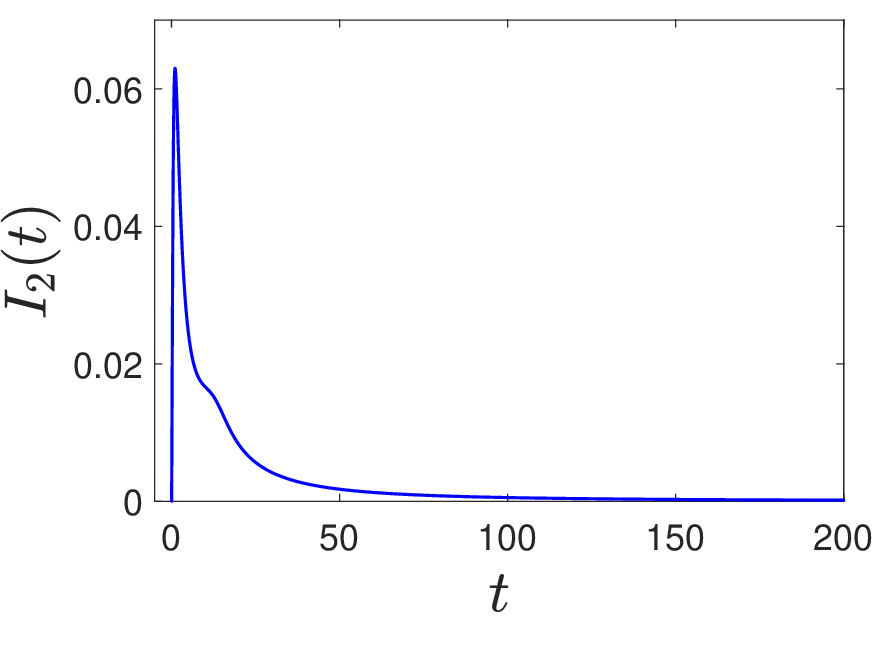}

\vspace*{-3mm}

\caption{The example of a global solution:  The plots of the components ${x_1(t)=I_1(t)}$, ${x_2(t)=I_{31}(t)}$ and ${x_3(t)=I_2(t)}$ of the numerical solution of the DAE \eqref{DAE}, \eqref{NestCoef2DAE} with the functions \eqref{StepenFunc}, where $a=b=c=1$, and \eqref{Example8_param}, and with the initial values $t_0=0$, $x_0=(0,0,0)^\T$.  
The presented graphs demonstrate that the qualitative behavior of the numerical solution is consistent with the theoretical conclusion about the existence of the global exact solution which, however, can be unbounded on $[0,\infty)$.}
\label{Example8_1-3}
 \end{figure}

In realistic problems of electrical engineering the inductance $L(t)$ can be very small, therefore, we take ${L(t)=10^{-3}}$. Choose the remaining parameters of the circuit in the form \eqref{StepenFunc}, where $a=b=c=1$, and
\begin{equation}\label{Example5_param}
R_1(t)=e^{-t},\quad R_2(t)=5+e^{-t},\quad I(t)=\sin t,\quad U(t)=(t+1)^{-1},\quad G_3(t)=(t+1)^{-1}.
\end{equation}
Take the consistent initial values  $t_0=0$, $x_0=(0,0,0)^\T$. As proved above, the exact solution is global.
The components of the computed (by method 2) solution are plotted in Fig. \ref{Example5_1-3_2h}.
 \begin{figure}[!h]%
 \centering
\vspace*{-1mm}
\includegraphics[width=4.6cm]{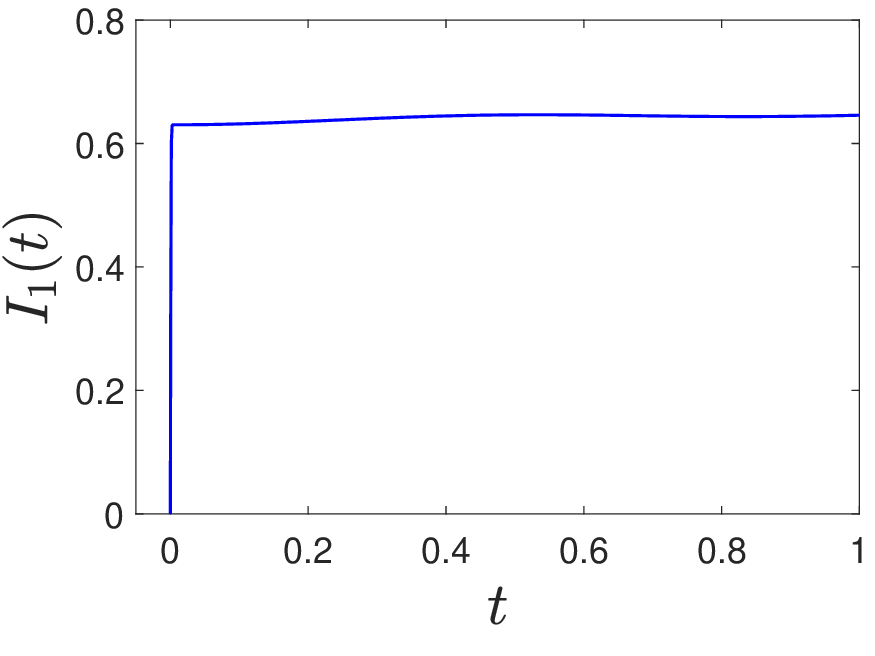}
\includegraphics[width=4.6cm]{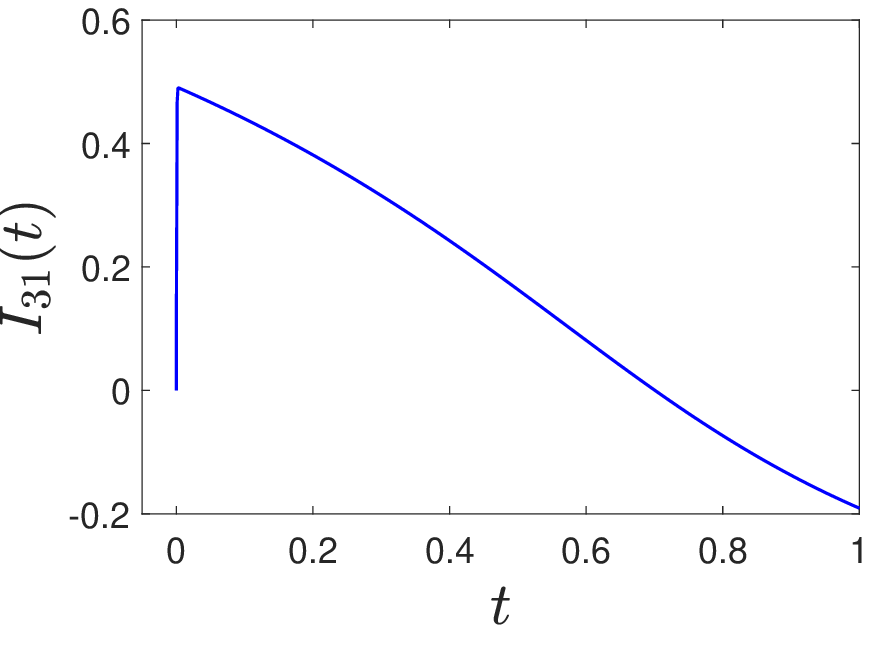}
\includegraphics[width=4.7cm]{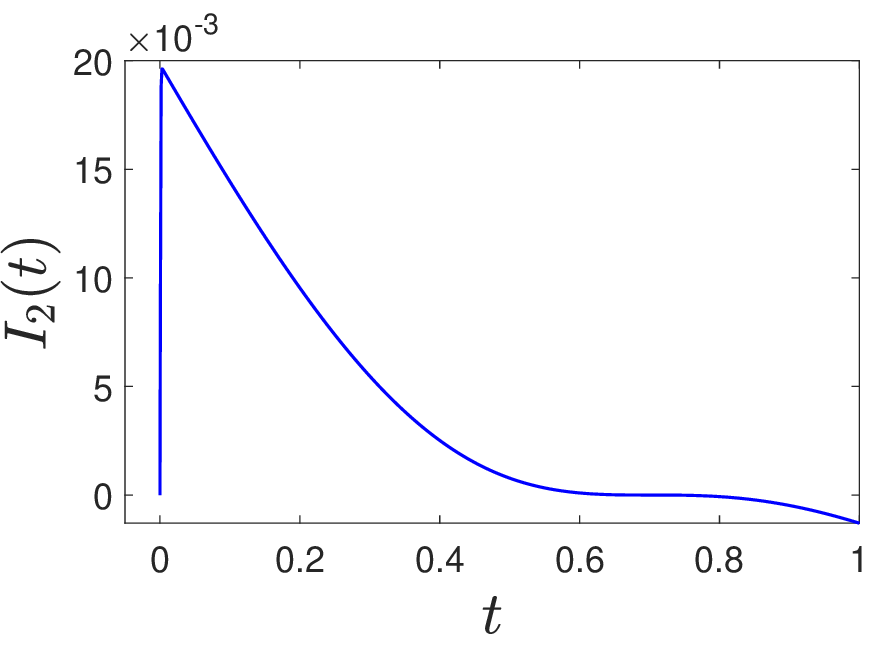}

\vspace*{-3mm}

\caption{The example of the solution for the case usually encountered in practice when the inductance $L(t)$ is small:
The plots of the components  of the numerical solution of the DAE \eqref{DAE}, \eqref{NestCoef2DAE} with ${L(t)=10^{-3}}$, the functions \eqref{StepenFunc}, where $a=b=c=1$, and \eqref{Example5_param}, and with the initial values $t_0=0$, $x_0=(0,0,0)^\T$. The theoretical analysis shows that the exact solution is global, and the behavior of the presented numerical solution is consistent with the theoretical conclusion.}\label{Example5_1-3_2h}
 \end{figure}

Consider the case when the function $U(t)$ is not continuously differentiable, but only continuous. Take the voltage of the triangular shape (see Fig.~\ref{UTriang}):
\begin{equation}\label{UTriang-func}
{U(t)=10-|t-10-20\,k|},\quad t\in [20\,k, 20+ 20\,k],\quad k\in{\mathbb N}\cup\{0\}.
\end{equation}
In this case we use Propositions~\ref{remNum-meth} and~\ref{remModNum-meth}.  Also, take the functions
\begin{equation}\label{Example6_param}
I(t)=(t+1)^{-1}-1,\; G_3(t)=(t+1)^{-1},\; L(t)=10^{-1}+(t+1)^{-1},\; R_1(t)=e^{-t},\; R_2(t)=2+e^{-t},
\end{equation}
 \begin{wrapfigure}[5]{r}{0.3\linewidth}
\vspace*{-2mm}
\includegraphics[width=\linewidth]{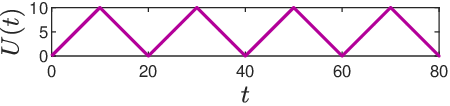}

\vspace*{-3mm}

\caption{The plot of $U(t)$}\label{UTriang}
 \end{wrapfigure}
and $\varphi_i$ ($i=1,2,3$) of the form \eqref{StepenFunc} where $a=b=c=1$, and the consistent initial values $t_0=0$, $x_0=(0,0,0)^\T$.  As proved above, the exact solution of the DAE is global. The numerical solution for this case was obtained by both method 1 and method 2. The plots of its components obtained by method 2 are presented in Fig.~\ref{Example6_1-3}.
\begin{figure}[H]%
 \centering
\includegraphics[width=4.8cm]{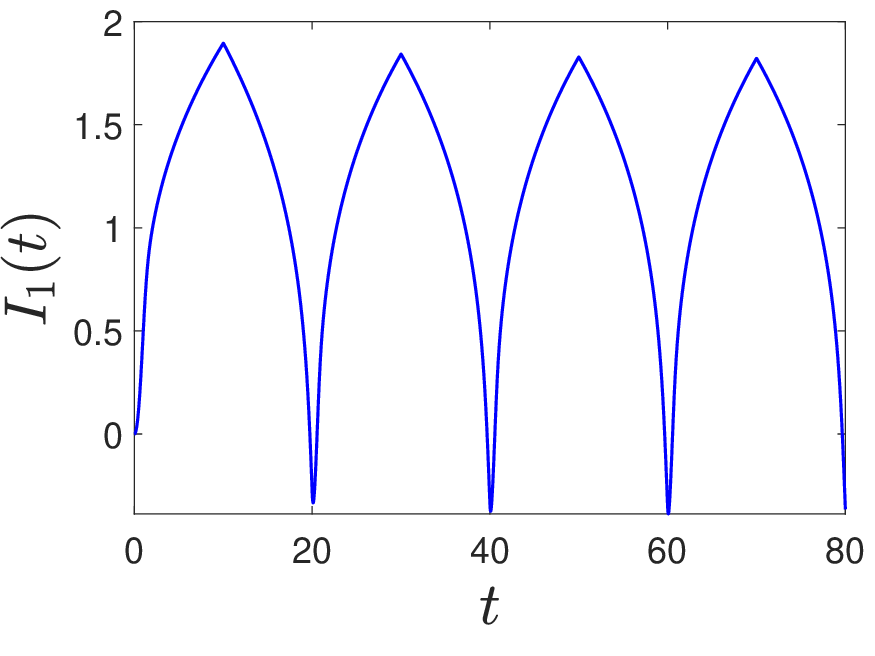}
\includegraphics[width=4.8cm]{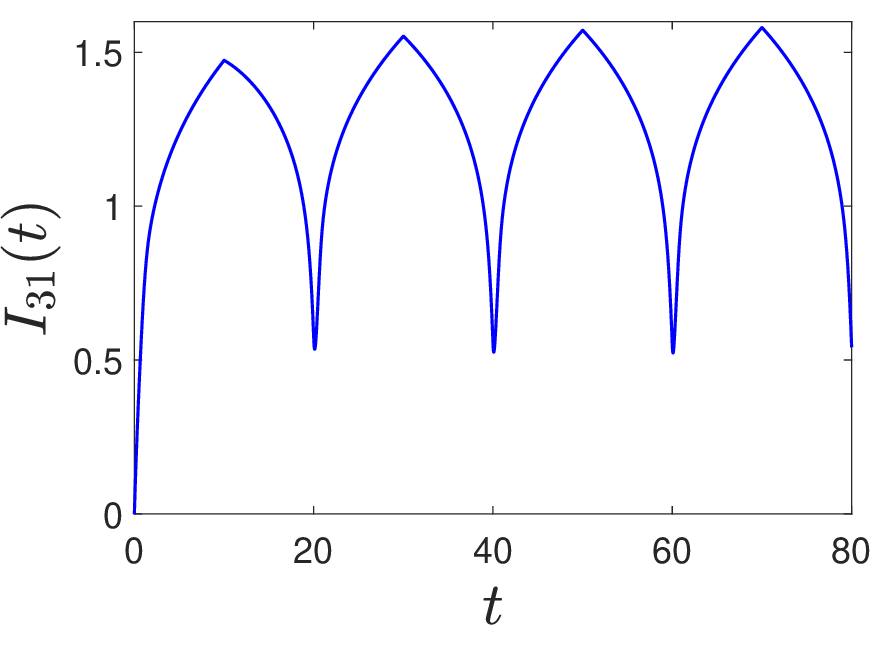}
\includegraphics[width=4.8cm]{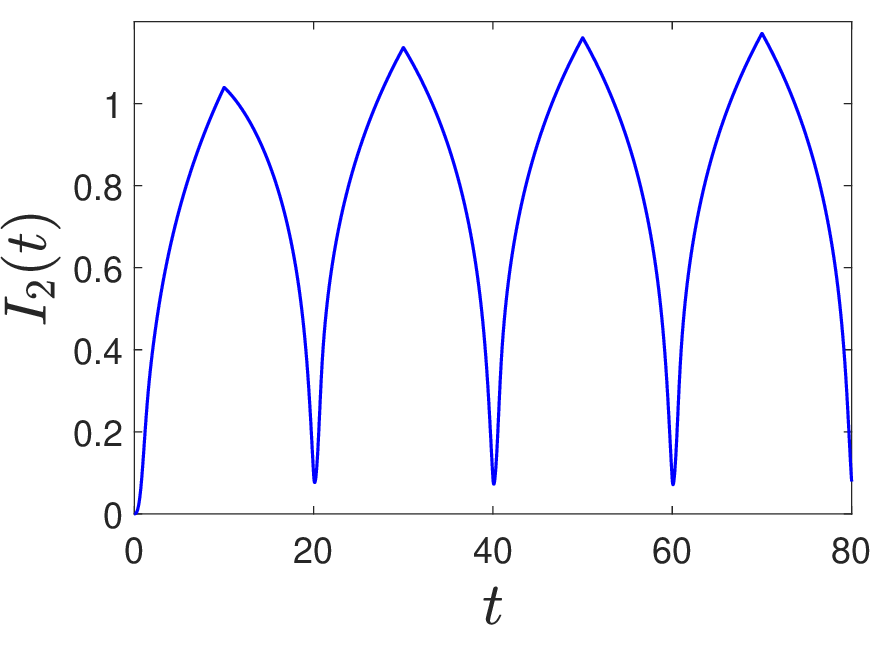}

\vspace*{-3mm}

\caption{The example of a solution for the case when the function $U(t)$ is continuous, but not differentiable: The plots of the components
$x_1(t)\!=\! I_1(t)$, $x_2(t)\!=\! I_{31}(t)$, $x_3(t)\!=\! I_2(t)$ of the numerical solution of the DAE \eqref{DAE}, \eqref{NestCoef2DAE} with the functions \eqref{UTriang-func}, \eqref{Example6_param}, \eqref{StepenFunc}, where $a\!=\! b\!=\! c\!=\! 1$, and the initial values $t_0\!=\! 0$, $x_0\!=\! (0,0,0)^\T$. The theoretical analysis shows that the exact solution is global and the plots demonstrate the same behavior pattern of the numerical solution.} \label{Example6_1-3}
\end{figure}

Now, consider the case when
\begin{align}\label{Example7_param}
& \varphi_1(x_1)\!=\! x_1^5,\, \varphi_2(x_3)\!=\! \frac{1}{3}\cos x_3,\, \varphi_3(x_2)\!=\! \frac{1}{3}\cos x_2,\, R_1(t)\!=\! 1+\frac{1}{2}\sin t,\, R_2(t)\!=\! 3+\frac{1}{2}\sin t,  \\
& L(t)\!=\! (t+10)^{-1/2}+10^{-2},\, I(t)\!=\! (\ln(t+1)+1)^{-1},\,  U(t)\!=\! 100(t+1)^{-2},\, G_3(t)\!=\! (t+1)^{-1}.
\end{align}
In this case the DAE \eqref{DAE}, \eqref{NestCoef2DAE}  is Lagrange stable since the conditions for the Lagrange stability, specified in Section \ref{NestElCirc-TheorAnal}, hold.  Take the consistent initial values $t_0=0$, $x_0=(4/3,0,0)^\T$.  The plots of the components of the numerical solution (computed by method 2) are given in Fig.~\ref{Example7_1-3}.
\begin{figure}[!h]%
 \centering
\includegraphics[width=4.8cm]{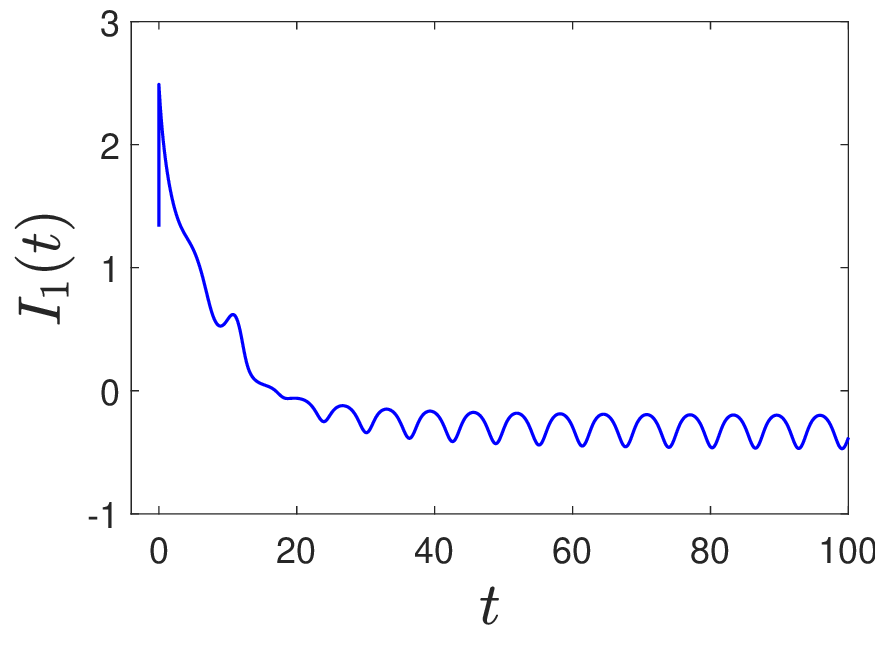}
\includegraphics[width=4.8cm]{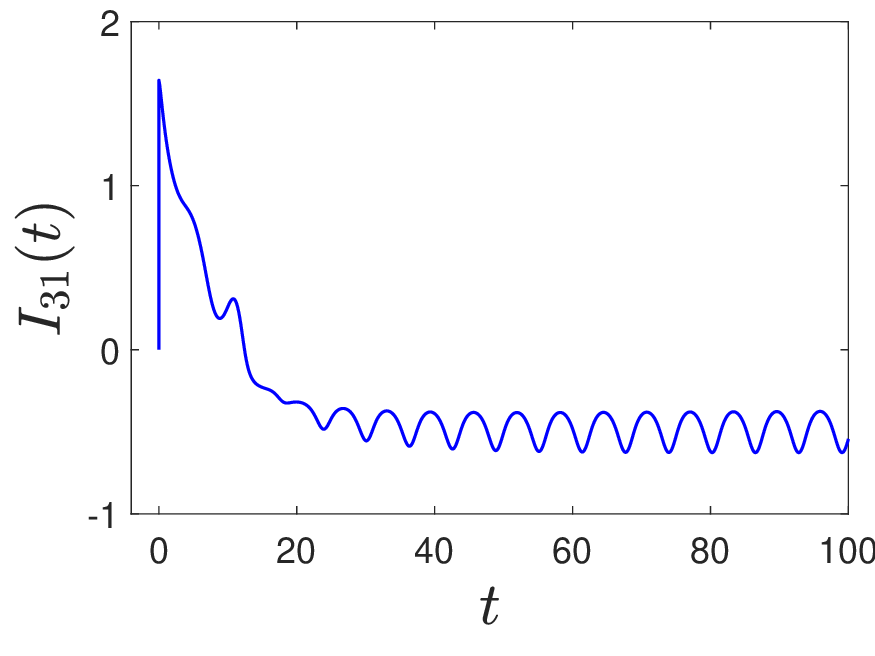}
\includegraphics[width=4.8cm]{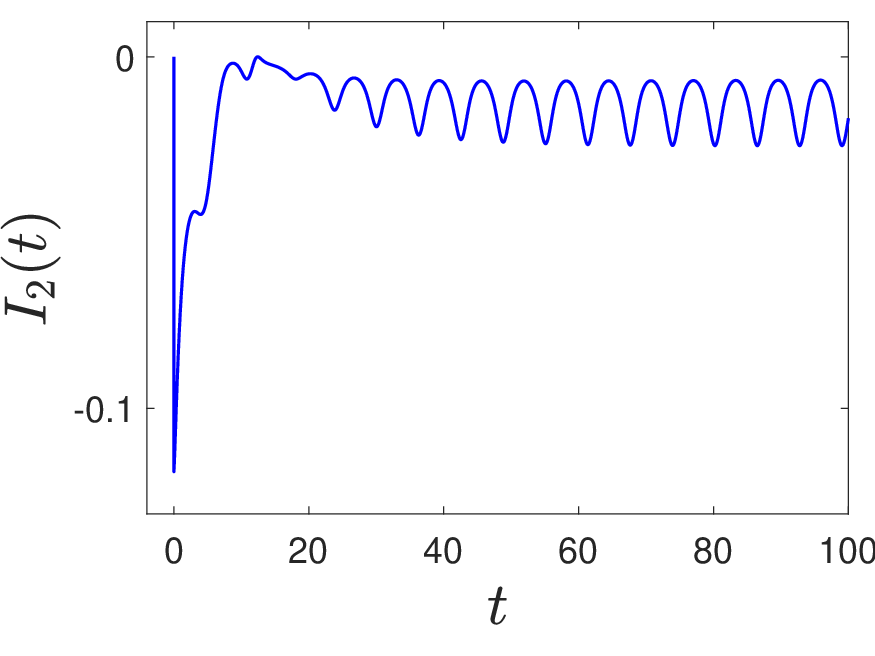}

\vspace*{-3mm}

\caption{The example of a Lagrange-stable solution: The plots of the components $x_1(t)=I_1(t)$, $x_2(t)=I_{31}(t)$, $x_3(t)=I_2(t)$ of the numerical solution of the DAE \eqref{DAE}, \eqref{NestCoef2DAE} with the functions \eqref{Example7_param} and the initial values $t_0=0$ and $x_0=(4/3,0,0)^\T$.
The presented graphs show the solution exists on the given interval and its norm does not increase with increasing time. When the interval is increased by a factor of 10, the qualitative picture of the solution behavior does not change.  Thus, the results of the numerical experiment are consistent with the theoretical conclusion about the Lagrange stability of the DAE. }
\label{Example7_1-3}
\end{figure}

Further, consider the particular case when $\varphi_i$, $i=1,2,3$, have the form \eqref{Sin_1-2-3}, i.e.,
$$
\varphi_1(y)=c\sin y,\quad  \varphi_2(y)=a\sin y,\quad \varphi_3(y)=b\sin y,\quad  a,\, b,\, c \in\R.
$$
Let $L, R_1, R_2\in C^1([t_+,\infty),\R)$,\, $I, U, G_3\in C([t_+,\infty),\R)$,\, $L(t), R_2(t), G_3(t)>0$ for $t\in [t_+,\infty)$,\,  the functions $G_3(t)$, $R_2(t)$ and numbers $a$, $b$ satisfy the condition \eqref{SinCond1} and $\inf\limits_{t\in [t_+,\infty)} R_1(t)=R_*>0$. Then, as shown in Section \ref{NestElCirc-TheorAnal}, for each initial point $(t_0,x_0)\in [t_+,\infty)\times \R^3$ satisfying \eqref{NestElSyst2x}, \eqref{NestElSyst3x} there exists a unique global solution of the DAE \eqref{DAE}, \eqref{NestCoef2DAE}, \eqref{Sin_1-2-3} with the
initial condition \eqref{ini}. It is readily verified that $t_0=0$ and $x_0=(0,0,0)^\T$ are consistent initial values if $I(0)=0$.

For the functions
\begin{equation}\label{Example_Sin1000_param}
I(t)=t,\; U(t)=t+1,\; G_3(t)=(t+1)^{-1},\; L(t)=1,\; R_1(t)=2+e^{-t},\; R_2(t)=0.1\, t+3
\end{equation}
and $\varphi_i$ ($i=1,2,3$) of the form \eqref{Sin_1-2-3} where $a=1/3$, $b=-1/2$ and $c=10$, and the consistent initial values $t_0=0$, $x_0=(0,0,0)^\T$, the components of the computed solution are plotted in Fig.~\ref{Example_1-3_Sin1000_h}.
 \begin{figure}[!h]%
 \centering
\includegraphics[width=4.7cm]{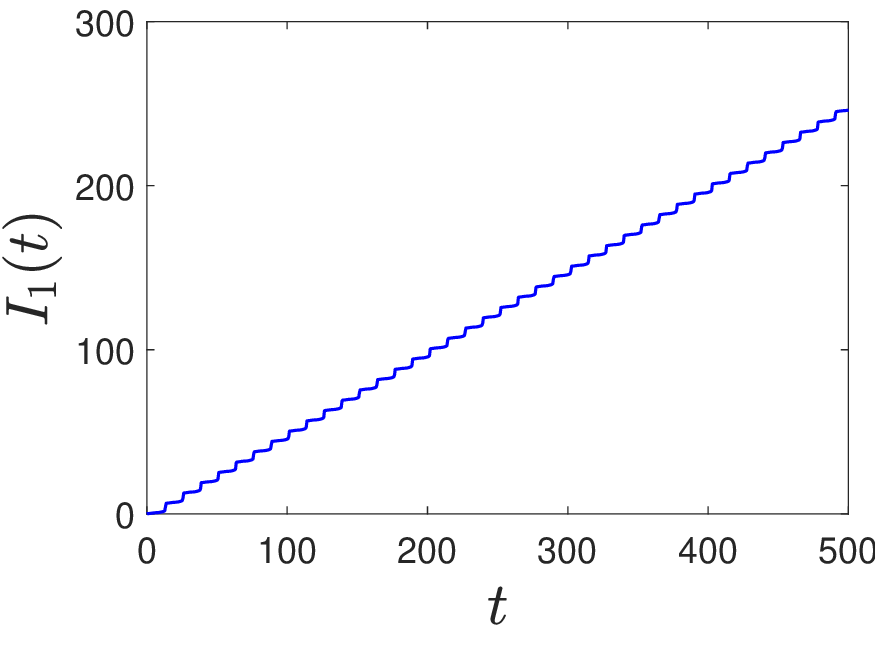}
\includegraphics[width=4.7cm]{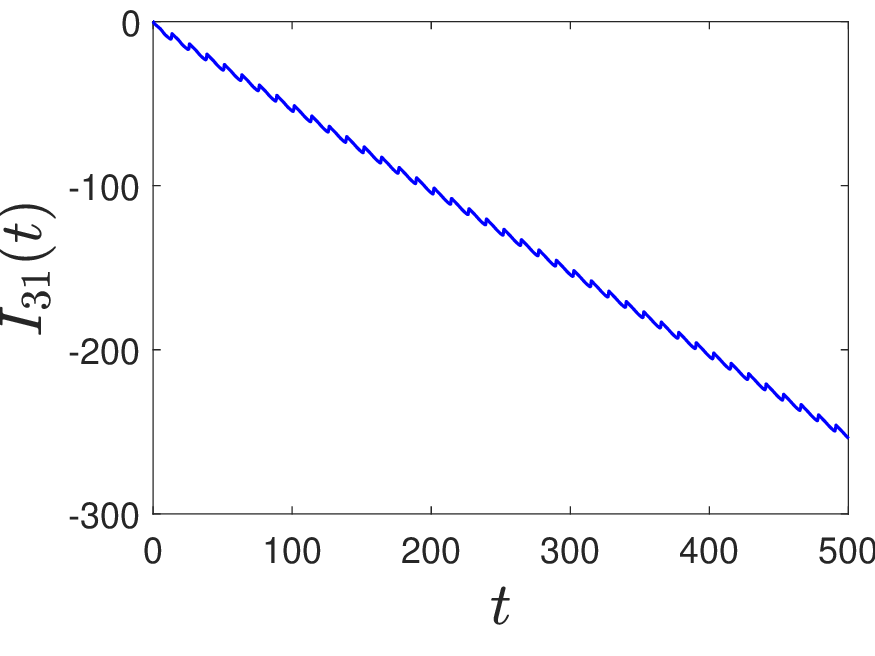}
\includegraphics[width=4.7cm]{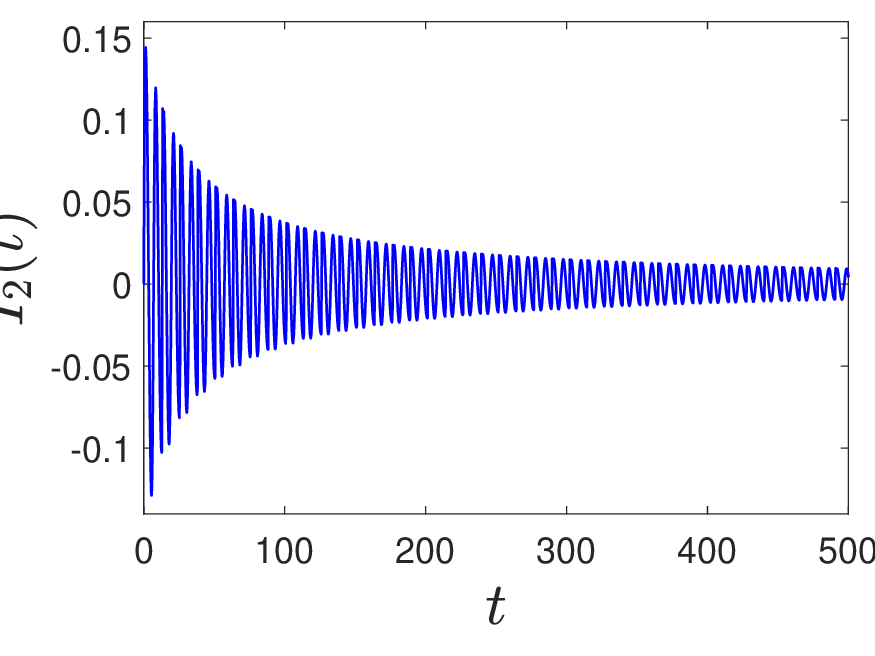}

 \vspace*{-3mm}

 \caption{The example of a global solution:  The plots of the components  of the numerical solution of the DAE \eqref{DAE}, \eqref{NestCoef2DAE} with the functions \eqref{Sin_1-2-3}, where $a=1/3$, $b=-1/2$ and $c=10$, and \eqref{Example_Sin1000_param}, and the initial values $t_0=0$, ${x_0=(0,0,0)^\T}$. The presented graphs demonstrate that the qualitative behavior of the numerical solution is consistent with the theoretical conclusion about the existence of the global exact solution (which, however, can be unbounded on $[0,\infty)$).}\label{Example_1-3_Sin1000_h}

\vspace*{-1mm}
 \end{figure}

Now, choose $\varphi_i$ ($i\!=\!1,2,3$) of the form \eqref{Sin_1-2-3}, where $a\!=\!1/3$, $b\!=\!-1/2$ and $c\!=\!5$, and
\begin{equation}\label{Example_Sin4_param}
I(t)\!=\! \sin t,\, U(t)\!=\!\frac{1}{(t\!+\! 1)^{5/2}} ,\,  G_3(t)\!=\! \frac{1}{t\!+\!1},\,  L(t)\!=\! 0.1\!+\!\frac{1}{t\!+\! 1},\, R_1(t)\!=\! 1\!+\! e^{-t},\, R_2(t)\!=\! \frac{\cos t}{2}\!+\!3.
\end{equation}
Then the conditions for the Lagrange stability, specified in Section \ref{NestElCirc-TheorAnal}, hold. The components of the solution computed for the consistent initial values $t_0\!=\!0$, ${x_0\!=\!(0,0,0)^\T}$ are plotted in Fig.~\ref{Example_1-3_Sin4_h}.
 \begin{figure}[!h]%
 \centering
\includegraphics[width=4.5cm]{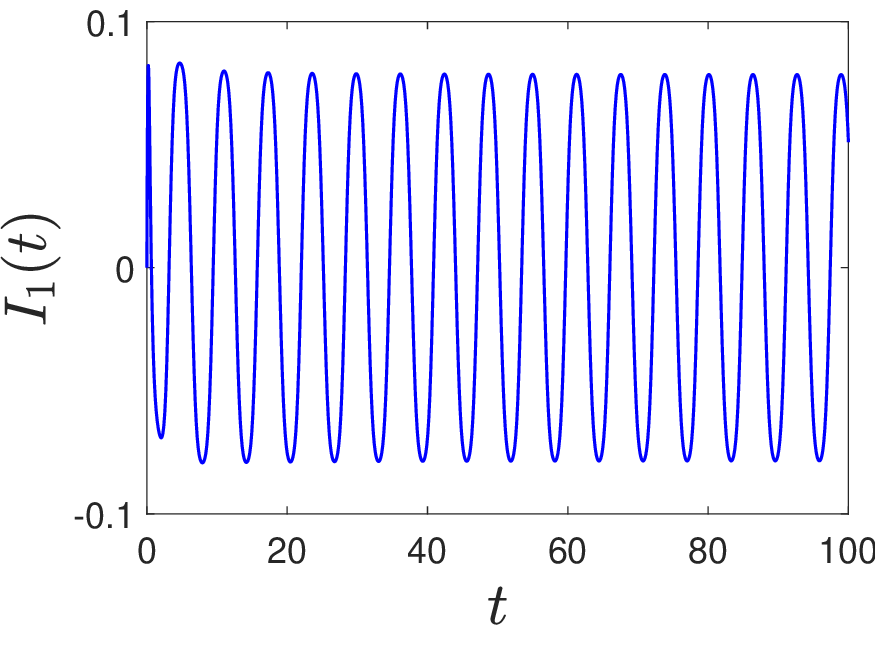}
\includegraphics[width=4.5cm]{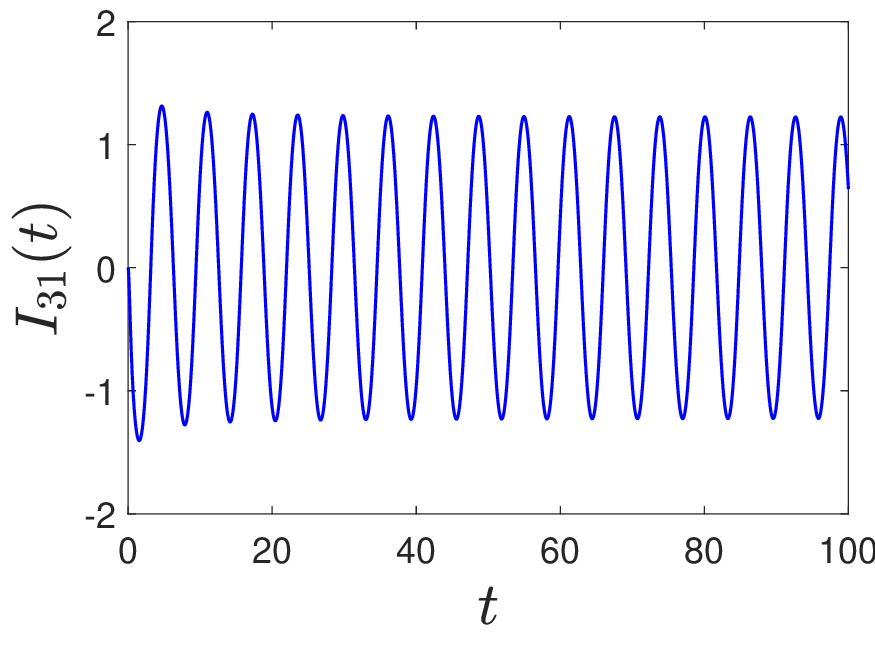}
\includegraphics[width=4.5cm]{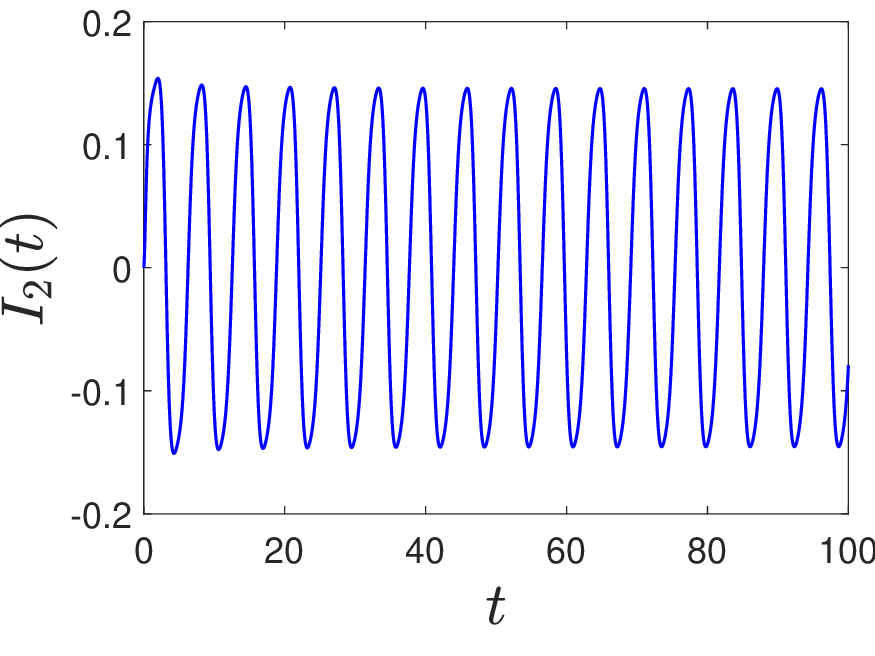}

\vspace*{-3mm}

\caption{The example of a Lagrange-stable solution:  The plots of the components of the numerical solution of the DAE \eqref{DAE}, \eqref{NestCoef2DAE} with the functions \eqref{Example_Sin4_param} and \eqref{Sin_1-2-3} where $a\!=\! 1/3$, $b\!=\! -1/2$, $c\!=\! 5$, and the initial values $t_0\!=\! 0$, $x_0\!=\! (0,0,0)^\T$. The plots demonstrate that the qualitative behavior of the numerical solution is consistent with the conclusion about the Lagrange stability of the exact solution.}\label{Example_1-3_Sin4_h}

\vspace*{-1mm}
 \end{figure}

The obtained numerical solutions show that the results of the theoretical research presented in Section~\ref{NestElCirc-TheorAnal} are consistent with the results of the numerical experiments.

\smallskip
We can conclude that methods 1, 2 are easy to implement, effective enough, and enable to carry out the numerical analysis of the global dynamics of mathematical models described by time-varying semilinear DAEs or the corresponding descriptor systems.

  \subsection{Comparative analysis of the methods and the experimental verification}\label{CompareMeth}

As stated in Theorems \ref{ThNum-meth}, \ref{ThNum-meth2}, methods 1 and 2 are convergent of order 1 and order 2 respectively, but at the same time method 2 requires greater smoothness for the functions in the equation, namely,  method~1 and  method~2 require that ${A, B\!\in\! C^2([t_0,T],\mathrm{L}(\Rn))}$, ${C_2\!\in\! C^2([t_0,T],(0,\infty))}$, ${f\!\in\! C^1([t_0,T]\times \Rn,\Rn)}$ and ${A, B\!\in\! C^3([t_0,T],\mathrm{L}(\Rn))}$, ${C_2\!\in\! C^3([t_0,T],(0,\infty))}$, ${f\!\in\! C^2([t_0,T]\times \Rn,\Rn)}$,  respectively.  However, if ${A, B\!\in\! C^1([t_+,\infty),\mathrm{L}(\Rn))}$, $C_2\!\in\! C^1([t_+,\infty),(0,\infty))$ and $f\!\in\! C([t_+,\infty)\times\Rn,\Rn)$ is such that $\partial f/\partial x$ is continuous on $[t_+,\infty)\times\Rn$, then the methods also converge, as stated in Propositions \ref{remNum-meth},~\ref{remModNum-meth}, and method 2 still converges faster.

In this section, we consider the DAE \eqref{DAE} (where for brevity we omit the dependence on $t$ in the notation of the variable vector $x(t)$\,)\, with $A(t)$, $B(t)$, $f(t,x)$ and $x$ of the form \eqref{NestCoef2DAE} where   \begin{gather*}
I(t)=\sin t,\;\; U(t)=(t+1)^{-1},\;\; G_3(t)=(t+1)^{-1},\;\; L(t)=500,\;\; R_1(t)=e^{-t},\;\; R_2(t)=2+e^{-t},  \\
\varphi_1(x_1) = x_1^3,\quad \varphi_2(x_3)= x_3^3,\quad \varphi_3(x_2)= x_2^3.
\end{gather*}
The physical interpretation of this DAE is given in Section~\ref{NestElCirc-TheorAnal}. We take the consistent initial values $t_0=0$, $x_0=(0,0,0)^\T$.

For the experimental verification and comparison of the rates of the convergence of methods 1 and 2, we will use the approach that is suitable for the case when an analytical (exact) solution cannot be obtained or its calculation is rather complicated.

Fig. \ref{Example_1-3_Com1}, \ref{Example_1-3_Com2} illustrate how the graphs of the components $x_1(t)\!=\!I_1(t)$ and $x_2(t)\!=\!I_{31}(t)$ of the solution \;$x(t)=(x_1(t),x_2(t),x_3(t))^\T=(I_1(t),I_{31}(t),I_2(t))^\T$,\;
computed by method~1 (the simple combined method \eqref{met1}--\eqref{met4}) and method~2 (the combined method with recalculation \eqref{NImpMet1}--\eqref{NImpMet6}) respectively, changes with the mesh refinement. These figures show that the graphs of the solution component $I_1(t)$ as well as $I_{31}(t)$ approach each other when decreasing a step size (${h=0.1,\,0.01,\,0.001}$), and hence method 1 as well as method 2 converges, at that, the graphs obtained by method 2 approach each other faster.
To show more clearly that the graphs of the components computed by method 2 for the step sizes $h=0.1$, $h=0.01$ and $h=0.001$ approach each other faster when decreasing the step size than those computed by method 1, the plots presented in Fig. \ref{Example_h_Com1}, \ref{Example_2h_Com1} and \ref{Example_h_Com2}, \ref{Example_2h_Com2},  are displayed on an enlarged scale in Fig.~\ref{Example_h-2h_Com1} and \ref{Example_h-2h_Com2} for both methods simultaneously.
It follows from Table~\ref{Table1} and Fig.~\ref{Example_1-3_Com1} that the rate of convergence of method 2 is higher than of method 1 when computing $I_1(t)$.
However, Table~\ref{Table2} and Fig.~\ref{Example_1-3_Com2} shows that there is not much difference in the rate of convergence of the methods when computing $I_{31}(t)$. This is due to the fact that $I_1(t)=x_1(t)$ is a component of the projection $x_{p_1}(t)=P_1(t)x(t)=(x_1(t),x_1(t),0)^\T$ in contrast to $I_{31}(t)=x_2(t)$ (see Section \ref{NestElCirc-TheorAnal}), and this will be explained in more detail below. Thus, the performed numerical experiment shows that method~2 converges faster than method~1.

The rate of convergence of method 2 increases mainly due to the faster convergence for the component $x_{p_1}(t)$, since in method 1 the method having the first order of convergence is applied to the ``differential part'' of the DAE (to the DE), and in method 2, due to recalculation, it has the second order of convergence. The Newton-type method with respect to $x_{p_2}(t)$  (which in general has the second order of convergence if functions in the equation are sufficiently smooth) is applied to the ``algebraic part'' of the DAE (to the AE) and the rate of convergence of method 2 for the component $x_{p_2}(t)$  increases due to the fact that the refined value of $x_{p_1}(t)$ is used in its recalculation.  This is shown in the proofs of Theorems \ref{ThNum-meth}, \ref{ThNum-meth2} as well as in the figures and tables above. The numerical experiments presented in Sections \ref{NumAnalPIMM}, \ref{NestElCirc-NumAnal} confirm the results of the above comparative analysis.
 \begin{figure}[!h]%
  \centering

\vspace*{-3mm}
\captionsetup{labelfont={rm}}
\subfloat[Method~1]{\includegraphics[width=6.8cm]{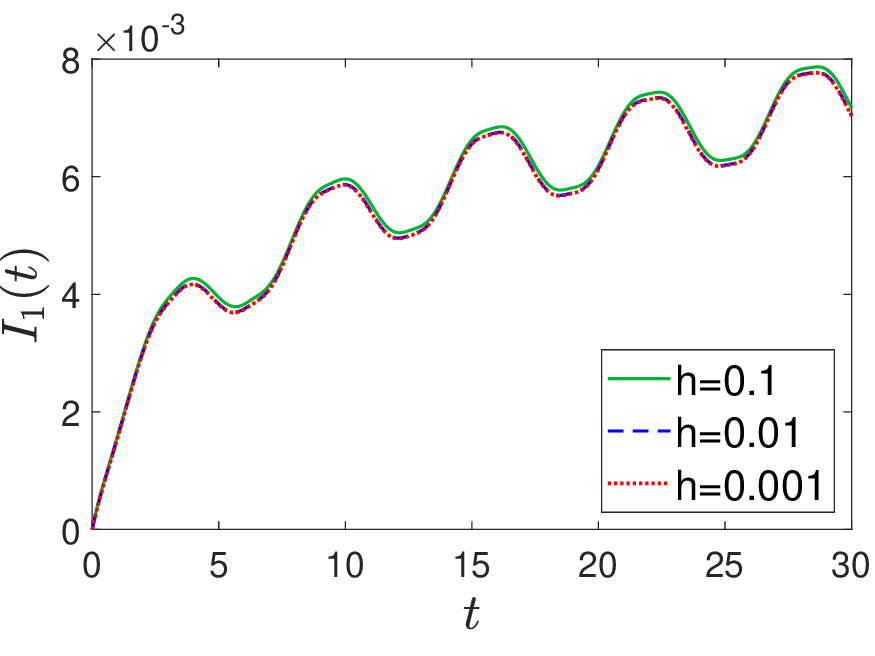}\label{Example_h_Com1}}
\subfloat[Method~2]{\includegraphics[width=6.8cm]{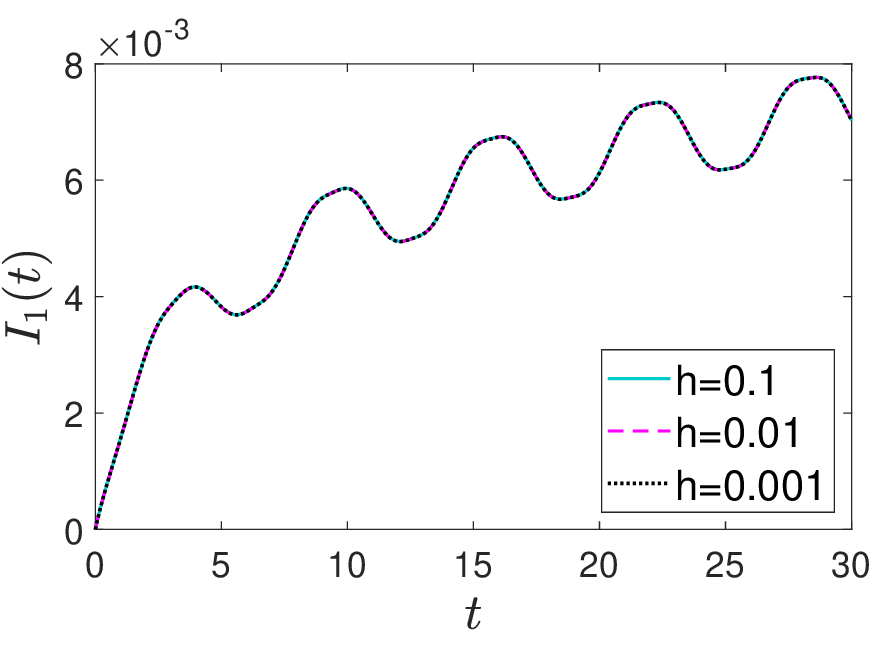}\label{Example_2h_Com1}}

\vspace*{-2mm}

\captionsetup{labelfont={rm}}
\subfloat[Methods~1 and 2 (an enlarged scale)]{\includegraphics[width=8.5cm]{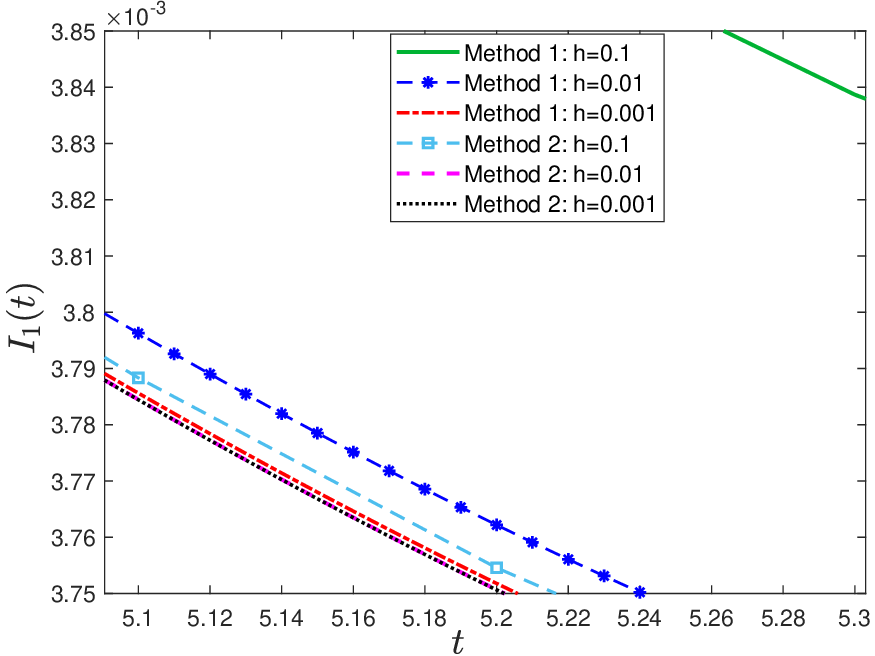}\label{Example_h-2h_Com1}}

\vspace*{-2mm}

\captionsetup{labelfont={sc}}
\caption{The plots of the solution component $I_1(t)$ computed by (a) method~1 and (b) method~2 with the step sizes $h=0.1$, $h=0.01$, $h=0.001$, and (c) the same plots on an enlarged scale for both methods simultaneously. It is shown that the graphs of the component approach each other when the step size decreases from $h=0.1$ to $h=0.001$, and hence methods 1 and 2 converge (since the graphs tend asymptotically to the graph of the component $I_1(t)$ of the exact solution as $h$ tends to $0$). In addition, the graphs obtained by method 2 approach each other faster when decreasing the step size than those obtained by method 1, and hence method 2 converges faster. To show this more clearly, the graphs are displayed on an enlarged scale (c), using the same colors.}\label{Example_1-3_Com1}
 \end{figure}

  \begin{table}[!h]%
 \captionsetup{font=small}
 \caption{\small The values of the solution component $I_1(t)$ obtained by methods 1 and 2 with the step size ${h=0.1,\, 0.01,\, 0.001}$ at ${t=0.2,\,0.4,\,0.6,\,0.8}$.
 The table shows that the values of the solution component obtained by method 2 for the step sizes $h=0.1$, $h=0.01$, $h=0.001$ approach each other (when decreasing the step size) faster than those obtained by method 1. Thus, the rate of convergence of method 2 is higher than of method 1.}\label{Table1}
\small

\vspace*{-2mm}

  \begin{tabular}{llllll}%
\hline\noalign{\smallskip}
& \multicolumn{1}{c}{$I_1(0.2)$} & \multicolumn{1}{c}{$I_1(0.4)$} & \multicolumn{1}{c}{$I_1(0.6)$} & \multicolumn{1}{c}{$I_1(0.8)$}\\
\cline{2-5} \rule{0cm}{0.3cm}
$h$     & \,Method 1\quad  Method 2 & \,Method 1\quad Method 2 & Method 1\; Method 2 & Method 1\; Method 2\!\!\! \\
\noalign{\smallskip}\hline\noalign{\smallskip}
$10^{-1}$\!\!  & 0.00038198\;  0.00036601 & 0.00070802\; 0.00068362 & 0.001006\; 0.000979 & 0.001296\; 0.001268\!\! \\
$10^{-2}$\!\!  & 0.00036690\; 0.00036530 & 0.00068447\; 0.00068202 & 0.000979\; 0.000976 & 0.001268\; 0.001265\!\! \\
$10^{-3}$\!\!  & 0.00036546\; 0.00036530 & 0.00068224\; 0.00068200 & 0.000977\; 0.000976 & 0.001265\; 0.001265\!\! \\
\noalign{\smallskip}\hline
 \end{tabular}
  \end{table}

 \begin{figure}[!h]%
\vspace*{-1mm}
\centering
\captionsetup{labelfont={rm}}
\subfloat[Method~1]{\includegraphics[width=6.7cm]{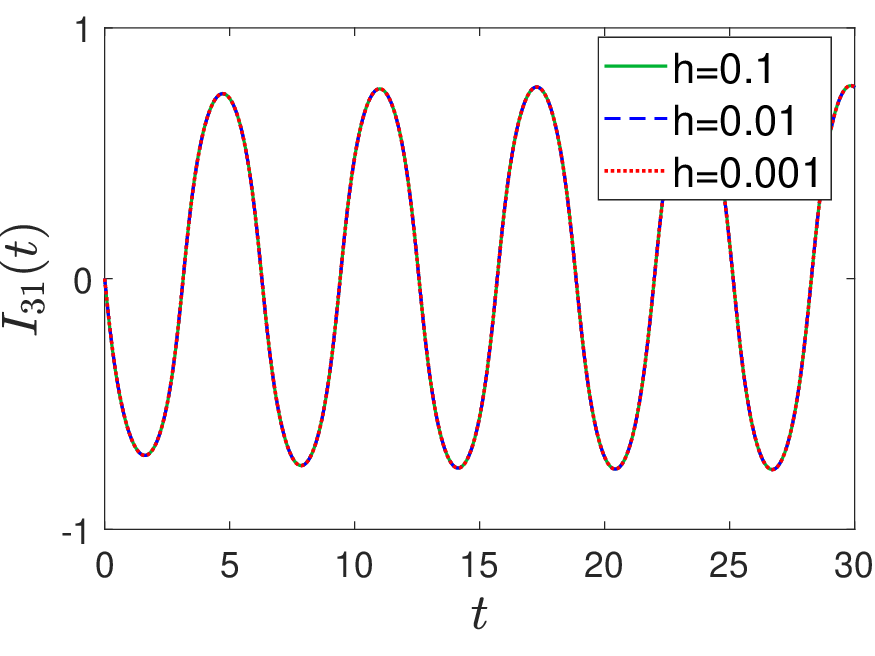}\label{Example_h_Com2}}
\subfloat[Method~2]{\includegraphics[width=6.7cm]{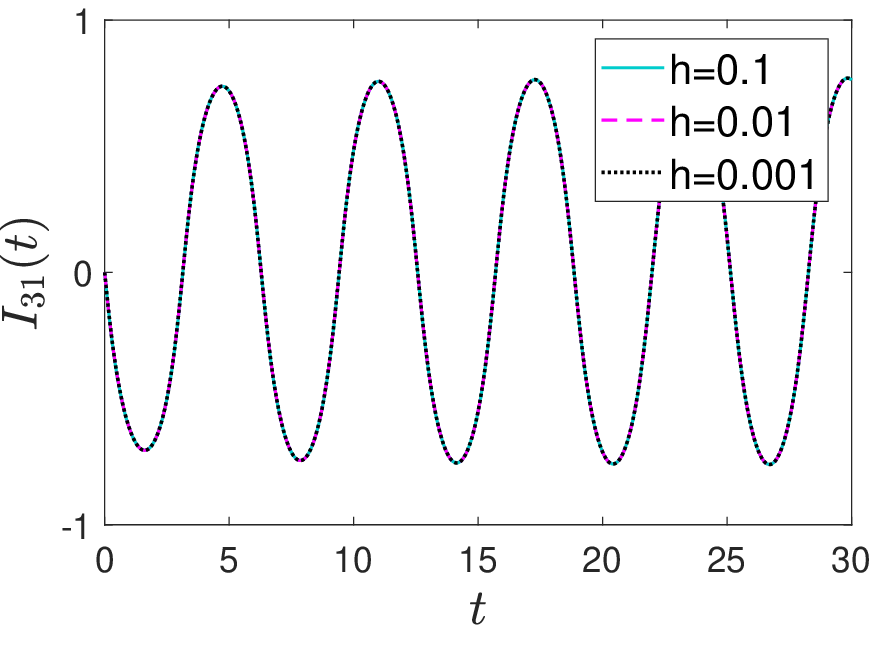}\label{Example_2h_Com2}}

\vspace*{-2mm}

\captionsetup{labelfont={rm}}
\subfloat[Methods~1 and 2 (an enlarged scale)]{\includegraphics[width=8.6cm]{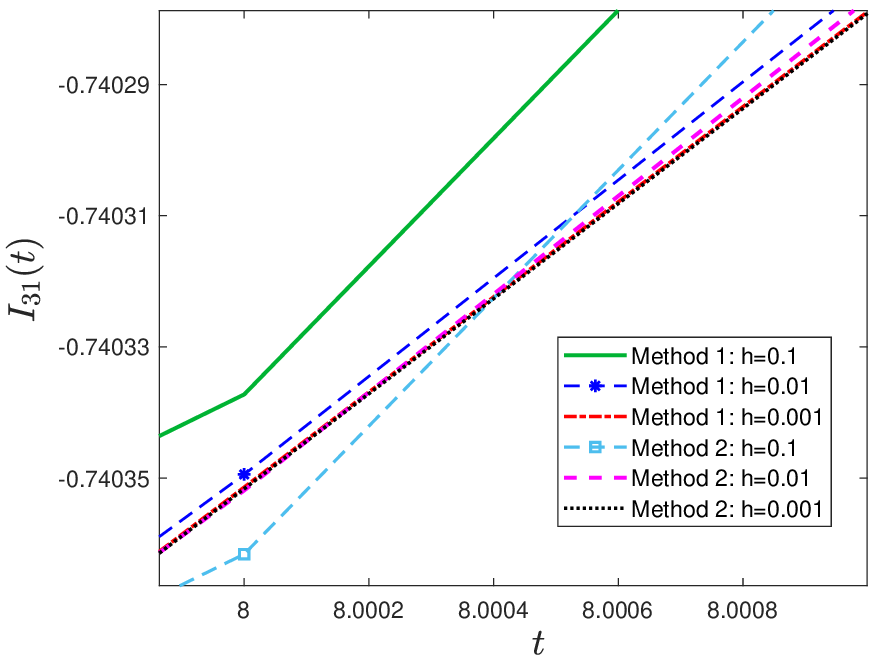}\label{Example_h-2h_Com2}}

\vspace*{-2mm}

\captionsetup{labelfont={sc}}
\caption{The plots of the solution component $I_{31}(t)$  computed by (a) method~1 and (b) method~2 with the step sizes $h=0.1$, $h=0.01$, $h=0.001$, and (c) the same plots on an enlarged scale for both methods simultaneously. The presented plots demonstrate the same as Fig. \ref{Example_1-3_Com1}, but for the component $I_{31}(t)$.}\label{Example_1-3_Com2}
 \end{figure}

  \begin{table}[!h]%
\vspace*{-2mm}

 \captionsetup{font=small}
\caption{\small The values of the solution component $I_{31}(t)$  obtained by methods 1 and 2 with the step size ${h=0.1,\, 0.01,\, 0.001,\, 0.0001}$ at $t=7.8,\, 7.9,\, 8$. The table shows that there is not much difference in the rate of convergence of the methods when computing $I_{31}(t)$, in contrast to the results obtained for the component~$I_1(t)$.}\label{Table2}
\small

 \vspace*{-2mm}

  \begin{tabular}{llll}%
\hline\noalign{\smallskip}
& \multicolumn{1}{c}{$I_{31}(7.8)$}
& \multicolumn{1}{c}{$I_{31}(7.9)$}
& \multicolumn{1}{c}{$I_{31}(8)$}\\
\cline{2-4} \rule{0cm}{0.3cm}
$h$     & \,Method 1\quad  Method 2 & \,Method 1\quad Method 2 & \,Method 1\quad Method 2\!\!\! \\
\noalign{\smallskip}\hline\noalign{\smallskip}
$10^{-1}$\!\!
  & -0.7446010\;  -0.7446247
  & -0.7449976\; -0.7450214
  & -0.7403373\; -0.7403616\!\! \\
$10^{-2}$\!\!
  & -0.7446068\; -0.7446091
  & -0.7450208\; -0.7450231
  & -0.7403495\; -0.7403518\!\! \\
$10^{-3}$\!\!
  & -0.7446089\; -0.7446091
  & -0.7450229\; -0.7450231
  & -0.7403514\; -0.7403516\!\! \\
$10^{-4}$\!\!
  & -0.7446091\; -0.7446091
  & -0.7450231\; -0.7450232
  & -0.7403516\; -0.7403516\!\! \\
 \noalign{\smallskip}\hline
 \end{tabular}
  \end{table}

\section*{Acknowledgments}
This work was partially supported by the Alexander von Humboldt Foundation (Project number 1153084-ESP-AHP) within the framework of the Alexander von Humboldt Professorship at the Friedrich-Alexander-Universit\"{a}t Erlangen-N\"{u}rnberg (the Chair for Dynamics, Control, Machine Learning and Numerics at) and the European Research Council (ERC) under the European Union's Horizon 2020 research and innovation programme (ERC Advanced Grant NEUROMORPH, no. 101018153).

\section*{Data and Code Availability}
The MATLAB code is available upon request from the author.


\begin{thebibliography}{50}
 \bibitem{Ascher-Petz}   Ascher, U.M.,  Petzold, L.R.,  Computer Methods for Ordinary Differential Equations and Differential-Algebraic Equations. SIAM, Philadelphia, PA (1998).

\bibitem{BogSob} Bogatyrev, S.V.,  Sobolev, V.A.,  Separating the rapid and slow motions in the problems of the dynamics of systems of rigid bodies and gyroscopes. J. of Appl. Math. and Mech. \textbf{52}(1), 41--48 (1988). \url{https://doi.org/10.1016/0021-8928(88)90057-3}

 \bibitem{BKT}  Borsche, R., Kocoglu, D., Trenn, S., A distributional solution framework for linear hyperbolic PDEs coupled to switched DAEs. Math. Control Signals Syst. \textbf{32}, 455--487 (2020). \url{https://doi.org/10.1007/s00498-020-00267-7}

\bibitem{Brenan-C-P} Brenan, K.E.,  Campbell, S.L., Petzold, L.R.,  Numerical Solution of Initial-Value Problems in Differential-Algebraic Equations. SIAM, Philadelphia, PA (1996).

\bibitem{CheAkLe}  Chernousko, F.L.,  Akulenko, L.D., Leshchenko, D.D.,  Evolution of motions of a rigid body about its center of mass, Springer, Cham (2017).

\bibitem{Chistyakov-Shcheglova} Chistyakov, V.F., Shcheglova, A.A., Selected Chapters of the Theory of Algebraic-Differential Systems. Nauka, Novosibirsk (2003). [in Russian]

\bibitem{Erickson-Maks} Erickson, R.W., Maksimovi\'c, D., Fundamentals of Power Electronics.  Kluwer Academic Publishers, Boston (2004).

\bibitem{Fil.DE-1} Filipkovskaya, M.S.,  Global solvability of time-varying semilinear differential-algebraic equations, boundedness and stability of their solutions. I. Differential Equations~\textbf{57}(1), 19--40 (2021).  \url{https://doi.org/10.1134/S0012266121010031}


\bibitem{Fil.DE-2} Filipkovskaya, M.S., Global solvability of time-varying semilinear differential-algebraic equations, boundedness and stability of their solutions. II. Differential Equations~\textbf{57}(2), 196--209 (2021). \url{https://doi.org/10.1134/S0012266121020099}

\bibitem{Fil.MPhAG} Filipkovska, M.S., Lagrange stability of semilinear differential-algebraic equations and application to nonlinear electrical circuits. J. of Math. Phys., Anal., Geom. \textbf{14}(2), 169--196 (2018). \url{https://doi.org/10.15407/mag14.02.169}

\bibitem{Fil.CombMeth}	Filipkovska, M.S., Two combined methods for the global solution of implicit semilinear differential equations with the use of spectral projectors and Taylor expansions. Int. J. of Computing Science and Mathematics~\textbf{15}(1), 1--29 (2022). \url{http://dx.doi.org/10.1504/IJCSM.2019.10025236}

\bibitem{Gear_Petz84} Gear, C.W., Petzold, L.R., ODE methods for the solution of differential/algebraic systems, SIAM J. Numer. Anal. \textbf{21}(4), 716--728 (1984). \url{https://doi.org/10.1137/0721048}

\bibitem{Hairer-W} Hairer, E., Wanner,  G., Solving Ordinary Differential Equations II. Stiff and Differential-Algebraic Problems. Springer, Berlin (2010).

\bibitem{IzgiCetin} \.Izgi, B., \c{C}etin,  C.,  Semi-implicit split-step numerical methods for a class of nonlinear stochastic differential equations with non-Lipschitz drift terms. J. Comput. Appl. Math.~\textbf{343}, 62--79 (2018). \url{https://doi.org/10.1016/j.cam.2018.03.027}

\bibitem{Kato-eng} Kato, T., Perturbation theory for linear operators. Springer-Verlag,  Berlin (1966).

\bibitem{Knorrenschild}  Knorrenschild, M.,  Differential/algebraic equations as stiff ordinary differential equations. SIAM J. Numer. Anal. \textbf{29}(6), 1694--1715  (1992). \url{https://doi.org/10.1137/0729096}

\bibitem{Kunkel_Mehrmann} Kunkel, P.,  Mehrmann, V., Differential-Algebraic Equations. Analysis and Numerical Solution. European Mathematical Society, Z\"urich (2006).

\bibitem{Lamour-Marz-Tisch}  Lamour,  R., M\"arz, R.,  Tischendorf, C., Differential-Algebraic Equations: A Projector Based Analysis. Differential-Algebraic Equations Forum. Springer, Berlin (2013).

\bibitem{Linh-Mehrmann}  Linh, V.H.,  Mehrmann, V.,  Efficient integration of strangeness-free non-stiff differential-algebraic equations by half-explicit methods. J. Comput. Appl. Math. \textbf{262},  346--360 (2014).

\bibitem{Riaza}  Riaza, R., Differential-Algebraic Systems. Analytical Aspects and Circuit Applications. World Scientific, Hackensack, NJ (2008).

\bibitem{RF1} Rutkas, A.G.,  Filipkovskaya, M.S., Extension of solutions of one class of differential-algebraic equations. J. of Computational \& Applied Mathematics \textbf{1}, 135--145 (2013). [in Russian]

\bibitem{SatoMdA} Sato Martin de Almagro, R.T., Convergence of Lobatto-type Runge–Kutta methods for partitioned differential-algebraic systems of index 2. BIT Numer. Math. \textbf{62}, 45--67 (2022). \url{https://doi.org/10.1007/s10543-021-00871-2}

\bibitem{TambueMukam} Tambue, A., Mukam, J.D., Strong convergence and stability of the semi-tamed and tamed Euler schemes for stochastic differential equations with jumps under non-global Lipschitz condition.  Int. J. Numer. Anal. Mod. \textbf{16}(6), 847--872 (2019).

\bibitem{Vlasenko1}  Vlasenko, L.A.,  Evolution Models with Implicit and Degenerate Differential Equations. System Technologies, Dnipropetrovsk (2006). [in Russian]

\bibitem{Vlasenko2000} Vlasenko,  L., Implicit linear time-dependent differential-difference equations and applications. Math. Meth. Appl. Sci. \textbf{23}, 937--948 (2000).
\end{thebibliography}
\end{document}